\documentclass[11pt,letterpaper]{amsart}

\usepackage{amsmath,amssymb,amsthm,amsxtra,mathrsfs}
\usepackage{graphicx}
\usepackage{cancel}
\usepackage[normalem]{ulem}
\usepackage{color,xcolor}
\usepackage[all,cmtip]{xy}
\usepackage{enumitem}
\usepackage{bm}
\usepackage{subcaption}
\usepackage{tikz}
\usetikzlibrary{arrows}
\usepackage{accents}

\renewcommand{\Re}{\operatorname{Re}\,}
\newcommand\const{\operatorname{const.}}
\renewcommand{\Im}{\operatorname{Im}\,}

\newcommand{\supp}{\operatorname{supp}}

\renewcommand\epsilon{\varepsilon}
\def\C{{\mathbb C}}
\def\H{\mathbb H}

\def\R{{\mathbb R}}

\def\Z{{\mathbb Z}}

\newtheorem{theorem}{Theorem}
\newtheorem{lemma}[theorem]{Lemma}
\newtheorem{corollary}[theorem]{Corollary}
\newtheorem{proposition}[theorem]{Proposition}

\theoremstyle{definition}

\theoremstyle{remark}
\newtheorem{remark}[theorem]{Remark}

\numberwithin{equation}{section}
\numberwithin{theorem}{section}
\numberwithin{problem}{section}

\title[Spectral projectors on hyperbolic surfaces]
{Boundedness of spectral projectors\\
on hyperbolic surfaces}

\author[J.-P. Anker, P. Germain and T. L\'eger]{Jean-Philippe Anker, Pierre Germain and Tristan L\'eger}

\address{Institut Denis Poisson, Universit\'{e} d'Orl\'{e}ans, Universit\'{e} de Tours \& CNRS, B\^{a}timent de Math\'{e}matiques, BP 6759, 45067, Orl\'{e}ans cedex 02, France
}
\email{anker@univ-orleans.fr}%

\address{Department of Mathematics, Huxley building, South Kensington campus, Imperial College London,London SW7 2AZ, United Kingdom}
\email{pgermain@ic.ac.uk}%

\address{Princeton University,  Mathematics  Department,  Fine Hall,Washington Road,  Princeton,  NJ 08544-1000,  USA}
\email{tleger@princeton.edu}%

\begin{document}
	
\begin{abstract} 
In this paper, we prove $L^2 \to L^p$ estimates, where $p>2$, for spectral projectors on a wide class of hyperbolic surfaces. More precisely, we consider projections in small spectral windows $[\lambda-\eta,\lambda+\eta]$ on geometrically finite hyperbolic surfaces of infinite volume. In the convex cocompact case, we obtain optimal bounds with respect to $\lambda$ and $\eta$, up to subpolynomial losses. The proof combines the resolvent bound of Bourgain-Dyatlov and improved estimates for the Schr\"odinger group (Strichartz and smoothing estimates) on hyperbolic surfaces.
\end{abstract}	

\maketitle 
	
\tableofcontents


\section{Introduction} Consider a (smooth, complete) Riemannian manifold $X$, with \textit{positive} Laplace-Beltrami operator $\Delta$. Spectral projectors on thin shells are defined through functional calculus by
\begin{equation}
\label{defP}
P^{\hspace{1pt}\prime}_{\lambda,\eta}=\mathbf{1}_{[\lambda-\eta,\lambda+\eta]}(\sqrt{\Delta})\,.
\end{equation}
We are interested in the following question:
\textbf{which bounds does the operator norm $\|P^{\hspace{1pt}\prime}_{\lambda,\eta}\|_{L^2 \to L^p}$
enjoy, for $p>2$, $\lambda$ large and $\eta$ small} ?

In this work, we are concerned with hyperbolic surfaces $X$, in other words smooth surfaces endowed with a complete Riemannian metric of constant Gaussian curvature $-1$.  A classical result of Hopf (see for instance \cite[Theorem 2.8]{Borthwick2016}) ensures that these surfaces are quotients of the hyperbolic plane $\mathbb{H}$ by a Fuchsian group $\Gamma$ (i.e. a discrete subgroup of $\text{PSL}(2,\R)$) with no elliptic element:
$$
X=\Gamma\backslash\H\,.
$$
We will start by providing some background to this problem, paying special attention to the relation between the bounds of the projectors and the underlying geometry of the manifold.

The notation $f\lesssim g$ between two nonnegative expressions
means that there exists a constant $C>0$ such that $f\le C g$\hspace{1pt};
and $f\sim g$ means that $f\lesssim g$ and $f\gtrsim g$.
We write $f\lesssim_{\hspace{1pt}a}g$ to specify that the constant $C$ depends on a parameter $a$.

\subsection{General manifolds} The fundamental theorem due to Sogge~\cite{Sogge2017} asserts that there exists $\eta_0>0$ such that, for any $\lambda>1$,
\begin{equation}
\label{soggestatement}
\|P^{\hspace{1pt}\prime}_{\lambda,\eta_0}\|_{L^2 \to L^p}
\sim\lambda^{\frac12-\frac2p}+\lambda^{\frac12\left(\frac12-\frac1p\right)}
\sim\begin{cases}
\lambda^{\frac12-\frac2p}
&\text{if }p\ge6\\
\lambda^{\frac12\left(\frac12-\frac1p\right)}
&\text{if }2<p\le6,\\
\end{cases}
\end{equation}
which will be abbreviated as
$$
\|P^{\hspace{1pt}\prime}_{\lambda,\eta_0}\|_{L^2\to L^p}\sim\lambda^{\gamma(p)}
\quad\mbox{with}\quad
\gamma(p)=\max\bigl\{\tfrac12\hspace{-1pt}-\hspace{-1pt}\tfrac2p\hspace{1pt},
\tfrac12\bigl(\tfrac12\hspace{-1pt}-\hspace{-1pt}\tfrac1p\bigr)\bigr\}.
$$
This statement is classical on compact Riemannian manifolds;
we show in Appendix \ref{AppendixSogge} how it can be extended
to complete Riemannian manifolds with bounded geometry.

For the round sphere, $P^{\hspace{1pt}\prime}_{\lambda,\eta}=P^{\hspace{1pt}\prime}_{\lambda,1}$ if $\lambda$ is an eigenvalue of $\sqrt{\Delta}$ and $\eta<1$, so that no improvement of the above bounds can be expected as $\eta$ decreases. However, for manifolds of nonpositive curvature, decay for the operator norm $\|P^{\hspace{1pt}\prime}_{\lambda,\eta}\|_{L^2 \to L^p}$ was established as $\eta$ decreases \cite{BlairHuangSogge2022, BlairSogge2019}, and even stronger estimates are expected.

\subsection{The Euclidean case} 
\label{euclideancase}
The Euclidean plane and its quotients share many features with the hyperbolic plane and its quotients; for this reason, it is helpful to review here what is known and expected in this case.

\subsubsection{The Euclidean plane}
For the Euclidean plane $\mathbb{R}^2$, optimal bounds are given by
$$
\|P^{\hspace{1pt}\prime}_{\lambda,\eta}\|_{L^2\to L^p}
\lesssim\lambda^{\frac12-\frac2p}\hspace{1pt}\eta^{\frac12}
+\lambda^{\frac12\left(\frac12-\frac1p\right)}\hspace{1pt}\eta^{\frac32\left(\frac12-\frac1p\right)}
\sim\begin{cases}
\,\lambda^{\frac12-\frac2p}\hspace{1pt}\eta^{\frac12}
&\text{if }p\ge6\\
\,\lambda^{\frac12\left(\frac12-\frac1p\right)}\hspace{1pt}\eta^{\frac32\left(\frac12-\frac1p\right)}
&\text{if }2<p\le6.\\
\end{cases}$$
This is equivalent to the Stein-Tomas theorem~\cite{Stein1993, Tomas1975}. That it is optimal can be seen from two examples: the term $\lambda^{\frac12-\frac2p}\hspace{1pt}\eta^{\frac12}$ comes from the ``radial example", where concentration occurs at a point (in physical space), while the term $\lambda^{\frac12 \left(\frac12-\frac1p\right)}\hspace{1pt}\eta^{\frac32\left(\frac12-\frac1p\right)}$ stems from the Knapp example, which is concentrated along a geodesic (segment).

\subsubsection{The Euclidean torus}

Considering a general torus $\mathbb{T}^2 = \mathbb{R}^2 / (a \Z + b \Z)$, 
where $a,b \in \R^2$ are linearly independent, the expectation is that optimal bounds are provided by 
\begin{equation}
\label{torus}
\|P^{\hspace{1pt}\prime}_{\lambda,\eta}\|_{L^2\to L^p}
\lesssim\lambda^{\frac12-\frac2p}\hspace{1pt}\eta^{\frac12}
+(\lambda\hspace{1pt}\eta)^{\frac12\left(\frac12-\frac1p\right)}
\end{equation}
up to subpolynomial factors. The first term corresponds to the "radial example" while the second is a variant of the Knapp example, where concentration occurs along a closed geodesic.  See~\cite{Bourgain1993, BourgainDemeter2013, BourgainDemeter2015, GermainMyerson2021, GermainMyerson2022b} for progress on this conjecture.

\subsubsection{The Euclidean cylinder}

For the Euclidean cylinder $\mathbb{R}^2 / (a \Z)$ or equivalently for the Cartesian product $\mathbb{T} \times \mathbb{R}$, the estimate~\eqref{torus} can be proved to hold, and it is furthermore optimal. See~\cite{GermainMyerson2022a}, where a concise proof relying on the $\ell^2$ decoupling theorem of Bourgain-Demeter~\cite{BourgainDemeter2015} is provided.

\subsection{Harmonic analysis on infinite area hyperbolic surfaces}

From this point on, we will only consider hyperbolic surfaces of infinite area, instead of general manifolds of dimension 2. As is customary in this context, we let
$$
D=\sqrt{\Delta\hspace{-1pt}-\hspace{-1pt}\tfrac14\,};
$$
the spectral projector $P^{\hspace{1pt}\prime}_{\lambda,{\eta}}$ is then slightly modified to become
$$
P_{\lambda,{\eta}}=\mathbf{1}_{[\lambda-\eta,\lambda+\eta]}(D)P_e,
$$
where $P_e$ projects on the essential spectrum, on which $\Delta\hspace{-1pt}-\hspace{-1pt}\frac14\hspace{-1pt}\ge\hspace{-1pt}0$\hspace{1pt}. This definition will simplify notations throughout the text, and, since we are interested in the range where $\lambda$ is large and $\eta$ small, bounds for $P_{\lambda,\eta}$ and $P^{\hspace{1pt}\prime}_{\lambda,\eta}$ are equivalent.
Here and in the sequel, we use the following comparison,
which follows from the $TT^*$ argument (\footnotemark).


\begin{lemma}\label{comparison}
Let \,$m_1,m_2\hspace{-1pt}:\hspace{-1pt}[\hspace{.5pt}0,+\infty)\hspace{-1pt}\to\hspace{-1pt}\C$
be two bounded measurable functions.
If \,$|m_1|\hspace{-1pt}\le\hspace{-1pt}|m_2|$\hspace{.5pt},
then \,$\|m_1(D)\|_{L^2\to L^p}\!\le\|m_2(D)\|_{L^2\to L^p}$
for every \,$p\hspace{-1pt}>\!1$\hspace{.5pt}.
\end{lemma}

\footnotetext{\,More precisely,
\begin{align*}
\|m_j(D)\|_{L^2\to L^p}^2
&=\|m_j(D)m_j(D)^*\|_{L^{p^{\hspace{.5pt}\prime}}\!\to L^p}
=\!\sup_{\substack{f\in\mathcal{C}_c(X)\\\|f\|_{p'}\le1}}\int_X
[\hspace{1pt}m_j(D)m_j(D)^*\!f\hspace{1pt}](x)\,\overline{f(x)}\,dx\\
&=\!\sup_{\substack{\,f\in\mathcal{C}_c(X)\\\|f\|_{p'}\le1}}\int_0^\infty\!
|m_j(\mu)|^2\underbrace{\langle d\hspace{1pt}\Pi(\mu)f,f\rangle}_{\ge0}
\end{align*}
for $j=1,2$.}

\subsubsection{$L^p\to L^q$ bounds for Fourier multipliers} When is a Fourier multiplier $m(D)$ bounded from $L^p$ to $L^q$ ? Already in the Euclidean case, more regularity is needed from $m$ if $2 < p,q < \infty$ or $1 < p,q <2$ than if $p < 2 < q$, as is illustrated by the classical Bochner-Riesz conjecture or the Fefferman ball multiplier theorem.

This behavior becomes much more striking in the case of the hyperbolic space - and, presumably, also in the case of its quotients. More specifically, it has been shown by Clerc and Stein~\cite{ClercStein1974} that $L^p$ multipliers  ($p \neq 2$) on the hyperbolic space enjoy a bounded analytic continuation of their symbol on a strip containing the real axis. The argument, which naturally extends to $L^p \to L^q$ boundedness with $2 < p,q < \infty$ or $1 < p,q <2$, is recalled in Appendix \ref{HolomorphicMultipliers}. See~\cite{FotiadisMarias2010} for positive results on multipliers on locally symmetric spaces.

\subsubsection{$L^2 \to L^p$ bounds} 
Since the cases $p,q<2$ or $p,q>2$ leads to this requirement of analyticity for the symbol of a bounded multiplier, we shall focus in the present article, on the case $p < 2 < q$, where the phenomenology is very different. We actually further restrict the problem, and ask for optimal bounds for the operator $\| P_{\lambda,\eta} \|_{L^2 \to L^p}$, which has symbol $\mathbf{1}_{[\lambda-\eta,\lambda+\eta]}$. This choice is natural if one thinks of $P_{\lambda,\eta}$ as a spectral projector; furthermore, understanding its boundedness properties allows to treat general multipliers, see~\eqref{generalmultipliers}. Finally, since $P_{\lambda,\eta} = P_{\lambda,\eta}^{\hspace{1pt}2}$, this immediately implies $L^p\to L^q$ bounds as well, with $p<2<q$, though they might not be optimal.

In the case of the hyperbolic plane, optimal $L^2 \to L^p$ projector bounds were obtained by Chen-Hassell~\cite{ChenHassell2018}, see also~\cite{GermainLeger2023} for an alternative proof and the dependence of the constant in $p$.

\subsubsection{Dispersive and Strichartz estimates} Strichartz estimates for the Schr\"odinger equation are important in their own right, and for applications to Partial Differential Equations. Furthermore, they will play an important role in some of our arguments in the present paper.

For the hyperbolic plane, dispersive and Strichartz estimates were first obtained by Banica~\cite{Banica2007}, Anker-Pierfelice~\cite{AnkerPierfelice2009} and Ionescu-Staffilani~\cite{IonescuStaffilani2009}. For convex cocompact quotients, the case of the exponent of convergence $\delta<\frac12$ was treated by Burq-Guillarmou-Hassel~\cite{BurqGuillarmouHassell2010}, see also~\cite{FotiadisMandouvalosMarias2018}, while the case $\delta\ge\frac12$ is due to Wang~\cite{Wang2019}. Finally, we mention~\cite{AMPVZ} for symmetric spaces of higher rank.

\subsection{Main results in this article}

\begin{theorem} \label{mainthm}
Consider a geometrically finite hyperbolic surface \,$X$ and assume that
\,$0\hspace{-1pt}<\hspace{-1pt}\eta\hspace{-1pt}<\!1\!<\hspace{-1pt}\lambda$\hspace{1pt},
$p\hspace{-1pt}>\hspace{-1pt}2$\hspace{1pt}.
\begin{itemize}
\item
If $X$ has infinite area and no cusps,
then for any \,$\epsilon\hspace{-1pt}>0\hspace{-1pt}$ and \hspace{1pt}$M\!>\hspace{-1pt}0$\hspace{.5pt},
there holds
$$
\|P_{\lambda,\eta}\|_{L^2\to L^p}
\lesssim_{\hspace{1pt}p,\epsilon,M}\lambda^{\gamma(p)+\hspace{.5pt}\epsilon}\,\eta^{\frac12}
\qquad\text{when \,$\lambda^{-M}\hspace{-2.5pt}<\hspace{-1pt}\eta\hspace{-1pt}<\!1$\hspace{.5pt}.}
$$
\item
If the exponent of convergence $\delta$ of $X$ is $<\frac12$
$(\footnote{\,Notice that the assumption $\delta\hspace{-1pt}<\hspace{-1pt}\frac12$ excludes cusps.})$,
then
$$
\|P_{\lambda,\eta}\|_{L^2\to L^p}
\lesssim\bigl(1\!+\hspace{-1.5pt}\tfrac1{\sqrt{p-2}}\bigr)\,
\lambda^{\gamma(p)}\,\eta^{\frac12}\hspace{1pt}.
$$
This bound is furthermore optimal for the hyperbolic plane.
\end{itemize}
\end{theorem}

The above statement implies boundedness results for general multipliers.
In the case $\delta\hspace{-1pt}<\hspace{-1pt}\frac12$, one obtains
\begin{equation}
\label{generalmultipliers}
\begin{split}
&\bigl\|\hspace{1pt}m(D)\hspace{.5pt}\bigr\|_{L^{p^{\hspace{.5pt}\prime}}\to L^p}
\lesssim\int_0^{\hspace{1pt}1}\!\lambda^2\hspace{1pt}|m(\lambda)|\,d\lambda
+\int_1^{\hspace{.5pt}\infty}\!\lambda^{2\gamma(p)}\hspace{1pt}|m(\lambda)|\,d\lambda\,,\\
&\bigl\|\hspace{1pt}m(D)\hspace{.5pt}\bigr\|_{L^2\to L^p}
\lesssim\Bigl(\int_0^{\hspace{1pt}1}\!\lambda^2\hspace{1pt}|m(\lambda)|^2\hspace{1pt}d\lambda
+\!\int_1^{\hspace{.5pt}\infty}\!\lambda^{2\gamma(p)}\hspace{1pt}|m(\lambda)|^2\hspace{1pt}d\lambda
\Bigr)^{\hspace{-1pt}\frac12}
\end{split}
\end{equation}
by using in addition the low frequency estimates
$$
\|P_{\lambda,\eta}\|_{L^{p^{\hspace{.5pt}\prime}}\to L^p}
\lesssim(\lambda^2\!+\hspace{-1pt}\eta^2)\,\eta
\quad\text{and}\quad
\|P_{\lambda,\eta}\|_{L^2\to L^p}
\lesssim(\lambda\hspace{-1pt}+\hspace{-1pt}\eta)\,\eta^{\frac12}.
$$
In the case $\delta\ge\frac12$, there is a similar statement for multipliers supported away from the origin,
which involves some loss in the power $\lambda$ and some regularity on $m$.

When the surface has cusps, we expect the spectral projectors to be unbounded. We prove this in the model case of the parabolic cylinder in Appendix \ref{unbounded-cusp}. The general case will be the object of a forthcoming paper.

\subsection{Finite area hyperbolic surfaces} Estimating the operator norm of $P_{\lambda,\eta}$ for finite area hyperbolic surfaces is a very hard problem. When a cusp is present, the obstruction pointed out in Appendix \ref{unbounded-cusp} is still present; it can be avoided, for instance by focusing on a compact subset of the surface.

Two lines of research have been developed to address the case of finite area hyperbolic surfaces.
The first possibility is to rely on semiclassical analysis and to leverage negative curvature,
see for instance~\cite{BlairHuangSogge2022, BlairSogge2017, BlairSogge2018, BlairSogge2019,
CanzaniGalkowski2020, CanzaniGalkowski2021, HassellTacy2015, HezariRiviere2016}.
These tools typically lead to improvements of a power of $\log\lambda$ over the universal theorem of Sogge.

The second possibility is to focus on arithmetic surfaces, for which sharper results can be obtained through the use of number theory, as was first observed by Iwaniec and Sarnak~\cite{IwaniecSarnak1995}. They conjectured, in the compact case, that an eigenfunction $\phi_\lambda$ associated with the eigenvalue $\lambda$ satisfies the bound
$$
\|\phi_\lambda\|_{L^\infty}\lesssim\lambda^\epsilon\,\|\phi_\lambda\|_{L^2},
$$
with a corresponding statement in the non-compact case. See also the review \cite{Sarnak2003} and further progress in this conjecture in~\cite{ButtcaneKhan2017,Humphries2018,HumphriesKhan2022}. 

In view of Theorem~\ref{mainthm}, it is tempting to ask whether, if $X$ is a (geometrically finite) finite area surface, there might hold
\begin{equation}
\label{compactcase}
\|P_{\lambda,\eta}\|_{L^2 \to L^p}\sim\lambda^{\gamma(p)}\eta^{\frac12}+1
\end{equation}
(up to subpolynomial factors).

The validity of~\eqref{compactcase} is supported by some observations. First, notice that the term $1$ on the right-hand side is necessary since the $L^p$ norm of a function on a finite area domain controls its $L^2$ norm, if $p \geq 2$. Furthermore, this bound matches the universal lower bound proved in Subsection~\ref{aglb} as well as the bound stated in Theorem~\ref{mainthm} in the case of infinite area surfaces. Finally, the typical spacing of eigenvalues is expected to be $\sim \lambda^{-1}$, so that the choice $\eta = \lambda^{-1}$ should corresponds to estimates on eigenfunctions. The bound~\eqref{compactcase} then gives $\|P_{\lambda,\lambda^{-1}}\|_{L^2 \to L^p}\sim 1$, which corresponds to the conjecture of Iwaniec-Sarnak~\cite{IwaniecSarnak1995} - once again, up to subpolynomial factors.

A further argument in favor of the bound~\eqref{compactcase} has to do with the role played by geodesics. As far as estimating  $\|P_{\lambda,\eta}\|_{L^2 \to L^p}$ goes, it seems that the most important geometric feature of the underlying Riemannian manifold is the presence of closed geodesics and their stability, more than the finiteness of the area. This is apparent in Euclidean geometry, as recalled in Subsection~\ref{euclideancase}\,: estimates for the torus or the cylinder differ from estimates from the whole space, and this can be related to the existence of closed geodesics in the former case, but not in the latter. Going back to hyperbolic geometry, we see that closed geodesics exist on quotients of the hyperbolic space, but that estimates are nevertheless the same for the hyperbolic space and its infinite area quotients. This can be explained by the instability of geodesics, and further supports the bound~\eqref{compactcase}, which does not seem to be altered by the presence of closed geodesics.


\subsection{Organization of the paper and general ideas of the proof} 
We begin by recalling known harmonic analysis results on the hyperbolic plane and its quotients in Section \ref{tools}.   

In the first part of the paper we consider the special case of the hyperbolic plane: we prove the desired upper bound on spectral projectors in Section \ref{upperbplane} by adapting the interpolation method of Stein-Tomas~\cite{Stein1993, Tomas1975}, and show that it is sharp in Section \ref{lowerb}. In Section \ref{DispersiveStrichartzH}, we also prove new improved dispersive, Strichartz, and $L^p$ smoothing estimates for functions with narrow support in frequency space. 

The corresponding results are deduced in the case of convex cocompact surfaces with exponent of convergence $\delta\in[\hspace{1pt}0,\frac12)$ in Section \ref{SurfacesSmallDelta}. In this case, the sum defining the periodized kernel of the various Fourier multipliers converges absolutely; this leads to $L^\infty$ bounds, from which one can argue as in the case of the hyperbolic plane.

Section  \ref{SurfacesLargeDelta} focuses on the case $\delta \in [\frac12,1)$. In this case, the key estimate is due to Bourgain-Dyatlov \cite{BourgainDyatlov2018}, which leads to a local smoothing estimate for the Schr\"odinger group. On the one hand, this smoothing estimate allows to bound the spectral projector when restricted to compact subsets of the surface. On the other hand, the improved Strichartz and $L^p$ smoothing estimates allow to deal with the hyperbolic ends. This approach is reminiscent of Staffilani-Tataru~\cite{StaffilaniTataru2002}, who were the first to leverage local smoothing to obtain Strichartz estimates.


Appendix~\ref{AppendixSogge} is dedicated to the extension of the Sogge theorem to the case of a complete manifold, Appendix \ref{unbounded-cusp} to the proof of unboundedness of projectors in the case of the parabolic cylinder, and Appendix~\ref{HolomorphicMultipliers} gives a version of the theorem of Clerc-Stein for $L^p\!\to\hspace{-1pt}L^q$ multipliers, $p\neq q$.

\subsection*{Acknowledgements} The authors are grateful to P. Sarnak and M. Zworski for many instructive conversations, and for pointing out very useful references.

During the preparation of this article, Pierre Germain was supported by the Simons Foundation Collaboration on Wave Turbulence, a start up grant from Imperial College, and a Wolfson fellowship. Tristan L\'{e}ger was supported by the Simons Foundation Collaboration on Wave Turbulence. 

\section{Harmonic analysis on the hyperbolic plane $\H$ and its quotients $\Gamma\backslash\H$}\label{tools}
The aim of this section is to recall some basic formulas and facts related to generalizations of the Fourier transform to hyperbolic surfaces. We use \cite{Helgason1984,Helgason2008,Koornwinder1984} as general references about analysis on symmetric spaces and \cite{Borthwick2016} for analysis on hyperbolic surfaces.

Before doing so, let us first fix the following normalization for the Fourier transform on the real line\,:
$$
\widehat{f}(\xi)=\tfrac1{\sqrt{2\pi}}\int_{-\infty}^{\hspace{.5pt}\infty}\!f(x)\,e^{-ix\xi}\,dx\hspace{1pt}.
$$
With this normalization, the Fourier transform is an isometry from $L^2(\mathbb{R})$ to itself.


\subsection{The spherical Fourier transform on the hyperbolic plane $\H$}
Consider a radial function $f$ only depending on the distance $r$ to the origin
(which is $i$ in the upper half-plane model).
The spherical Fourier transform and its inverse are given by
\begin{align*}
\widetilde{\mathcal{F}}f(\lambda)=\widetilde{f}(\lambda)=2\hspace{.5pt}\pi
\int_0^{\hspace{.5pt}\infty}\!f(r)\,\varphi_\lambda(r)\hspace{1pt}\sinh r\,dr\hspace{1pt},\quad
f(r)=\tfrac1{2\hspace{.5pt}\pi^2}\int_0^{\hspace{.5pt}\infty}\!\widetilde{f}(\lambda)\,
\varphi_\lambda(r)\,|\mathbf{c}(\lambda)|^{-2}\,d\lambda\hspace{1pt},
\end{align*}
where the spherical functions $\varphi_\lambda(r)$
and the Harish-Chandra function $\mathbf{c}(\lambda)$
are given by
\begin{align*}
\varphi_\lambda(r)
=\tfrac{\sqrt{2}}\pi\int_0^{\hspace{1pt}r}\!\cos(\lambda\hspace{.5pt}s)\,
(\cosh r\hspace{-1pt}-\hspace{-1pt}\cosh s)^{-\frac12}\,ds\hspace{1pt},\quad
\mathbf{c}(\lambda)
=\tfrac{\Gamma(i\lambda)}
{\sqrt{\pi}\,\Gamma(\frac12\hspace{-1pt}+\hspace{-1pt}i\lambda)}\,.
\end{align*}
Notice that
\begin{equation}\label{equivalents}\begin{gathered}
\varphi_\lambda(r)\sim\varphi_0(r)\sim(1\!+\hspace{-1pt}r)\,e^{-\frac r2}
\quad\text{when}\quad\lambda\hspace{.5pt}r\hspace{-1pt}<\!1\hspace{.5pt},\\
|\mathbf{c}(\lambda)|^{-2}\sim\,\begin{cases}
\,\lambda^2
&\text{for $\lambda$ small, say $0\hspace{-1pt}<\hspace{-1pt}\lambda\hspace{-1pt}<\!1$\hspace{.5pt},}\\
\,\lambda
&\text{for $\lambda$ large, say $\lambda\hspace{-1pt}\ge\!1$\hspace{.5pt}.}
\end{cases}
\end{gathered}\end{equation}
In this context, Plancherel's formula writes
\begin{equation*}\label{PlancherelSpherical}
\|f\|_{L^2(\mathbb{H})}^2=\tfrac1{2\pi^2}\int_0^{\hspace{.5pt}\infty}
|\widetilde{f}(\lambda)|^2\,|\mathbf{c}(\lambda)|^{-2}\,d\lambda\hspace{1pt}.
\end{equation*}

The spherical Fourier transform diagonalizes the Laplacian and allows to define its functional calculus. Setting
$$
D=\sqrt{\Delta-\tfrac14\,}, 
$$
and given an even function $m$, there holds
$$
m(D)f=\widetilde{\mathcal{F}}^{-1}\bigl(m\cdot\widetilde{\mathcal{F}}f\bigr).
$$
Finally, the operator $m(D)$ can be realized through its convolution kernel


$$
K_{\mathbb{H}}(r)=-\,\tfrac1{2\hspace{1pt}\pi^{\hspace{.5pt}3/2}}
\int_{\hspace{1pt}r}^{\hspace{.5pt}\infty}\!\tfrac\partial{\partial s}\hspace{1pt}\widehat{m}(s)\,
(\cosh s\hspace{-1pt}-\hspace{-1pt}\cosh r)^{-\frac12}\,ds\hspace{1pt},
$$
for which
$$
[\hspace{.5pt}m(D)f\hspace{.5pt}](x)
=[\hspace{.5pt}f\hspace{-1pt}*\hspace{-1pt}K_{\mathbb{H}}\hspace{.5pt}](x)
=\int_{\mathbb{H}}K_{\mathbb{H}}(\operatorname{dist}(x,y))\,f(x')\,dx'\hspace{1pt}.
$$
The formula above holds true for radial or non-radial functions $f$.

\subsection{The Fourier transform on the hyperbolic plane $\H$ in the upper-half plane model}
In this model, the hyperbolic plane is represented by the complex upper-half plane $\{z=x+iy\in\C\,|\,y>0\}$,
which  turns out to be more convenient for some computations. We follow here \cite[Ch. 4]{Borthwick2016}.
The generalized eigenfunctions are given by
$$
E(\lambda,\xi,z)=\mathbf{C}(\lambda)\,
\bigl(\tfrac y{(x\hspace{1pt}-\hspace{1pt}\xi)^2+\hspace{1pt}y^2}\bigr)^{\hspace{-1pt}\frac12+i\lambda}\,,
\quad\text{with}\quad
\mathbf{C}(\lambda)
=\tfrac1{2\pi i\hspace{1pt}\lambda\hspace{1pt}\mathbf{c}(\lambda)}
=\tfrac1{2\sqrt{\pi}}
\tfrac{\Gamma(\frac12+\hspace{1pt}i\lambda)}{\Gamma(1\hspace{.5pt}+\hspace{1pt}i\lambda)}\,.
$$
Notice that
$$
|\mathbf{C}(\lambda)|\sim\begin{cases}
\,1
&\text{for $\lambda$ small, say $0\hspace{-1pt}\le\hspace{-1pt}\lambda\hspace{-1pt}<\!1$\hspace{.5pt},}\\
\,\lambda^{-\frac12}
&\text{for $\lambda$ large, say $\lambda\hspace{-1pt}\ge\!1$\hspace{.5pt}.}
\end{cases}$$
The spectral projector $\mathbf{P}_{\hspace{-1pt}\lambda}=\delta_\lambda(D)$ has integral kernel
\begin{equation}\label{Kernel}
K_\lambda(z,z')=\tfrac2\pi\hspace{1pt}\lambda^2\hspace{1pt}
\Bigl(\int_{-\infty}^{\hspace{.5pt}\infty}E(\lambda,\xi,z)\,\overline{E(\lambda,\xi,z')}\,d\xi\Bigr)
=\tfrac1{2\pi^2}\hspace{1pt}|\mathbf{c}(\lambda)|^{-2}\,\varphi_\lambda(d(z,z')\hspace{-1pt})\hspace{1pt}.
\end{equation}
This leads to the following formulas for the Fourier transform, its inverse, and Plancherel's identity:
\begin{align*}
&\widetilde{f}(\lambda,\xi)
=\int_{\H}f(z')\,\overline{E(\lambda,\xi,z')}\,dz'\\
&f(z)=\tfrac2\pi\int_0^{\hspace{.5pt}\infty}\!\int_{-\infty}^{\hspace{.5pt}\infty}
\widetilde{f}(\lambda,\xi)\,E(\lambda,\xi,z)\,d\xi\,\lambda^2\hspace{1pt}d\lambda\\
&\bigl\|f\bigr\|_{L^2(\H)}^2=\tfrac2\pi\,
\bigl\|\hspace{1pt}\lambda\hspace{1pt}\widetilde{f}\,\bigr\|_{L^2_{\lambda,\xi}}^2.
\end{align*}


\subsection{The Fourier transform on hyperbolic surfaces
$X\hspace{-1pt}=\hspace{-1pt}\Gamma\hspace{.5pt}\backslash\hspace{.5pt}\H\hspace{1pt}$}\label{FTX}

Consider first the special case of the hyperbolic cylinder $C_{\ell}\hspace{-1pt}=\hspace{-.5pt}\Gamma_{\ell}\backslash\hspace{.5pt}\mathbb{H}$ with $\Gamma_{\ell}\hspace{-1pt}=\hspace{-1pt}\bigl\langle z\mapsto e^{\ell}z\bigr\rangle$. Let us express on $C_\ell$ the spectral measure in terms of the generalized eigenfunctions. The spectral measure for $D\hspace{-.5pt}=\hspace{-1pt}\sqrt{\Delta\hspace{-1pt}-\hspace{-1pt}\smash{\tfrac14}\vphantom{\big|}}$ is given by
\begin{align}\label{Stone}
d\hspace{1pt}\Pi_{\hspace{1pt}C_\ell}(\lambda)
=\tfrac{\lambda}{\pi i}\Bigl[R_{\hspace{1pt}C_\ell}(\tfrac12\hspace{-1pt}-\hspace{-1pt}i\lambda)-R_{\hspace{1pt}C_\ell}(\tfrac12\hspace{-1pt}+i\lambda)\Bigr]d\lambda\,,
\end{align}
according to Stone's formula, and the resolvent by the kernel
\begin{align*}
R_{\hspace{1pt}C_{\ell}}(\tfrac12\hspace{-1pt}\mp\hspace{-1pt}i\lambda;z,z')
=\sum_{k\in\Z}R_{\hspace{1pt}\H}(\tfrac12\hspace{-1pt}\mp\hspace{-1pt}i\lambda;z,e^{k\ell}z')\,.
\end{align*} 
Notice that the summand is $\text{O}(e^{-\const\ell|k|}),$ hence the sum converges.
Correspondingly, the generalized eigenfunctions are given by 
\begin{equation}\label{GeneralizedEigenfunctions}
E_{\hspace{1pt}C_{\ell}}(\lambda,\xi,z)=\sum_{k\in\Z}E\bigl(\lambda,\xi,e^{k\ell}z\bigr)\,.
\end{equation}
It follows from \eqref{Kernel} that the kernel of the spectral measure is given by
$$
d\hspace{1pt}\Pi_{\hspace{1pt}C_{\ell}}(\lambda;z,z')
=\frac1{2\pi^2}\sum_{k\in\mathbb{Z}}
\varphi_{\lambda}\bigl(d(z,e^{k \ell}z')\hspace{-1pt}\bigr)\,
|\mathbf{c}(\lambda)|^{-2} d\lambda\,.
$$
Let us express it in terms of the generalized eigenfunctions \eqref{GeneralizedEigenfunctions}.

\begin{lemma}\label{Stonecyl}
We have 
\begin{equation*}
d\hspace{1pt}\Pi_{\hspace{1pt}C_{\ell}}(\lambda;z,z')
=\tfrac2\pi\bigg[\int_{1<|\xi|<e^{\hspace{.5pt}\ell}}\!E_{\hspace{1pt}C_{\ell}}(\lambda,\xi,z)\,
\overline{E_{\hspace{1pt}C_{\ell}}(\lambda,\xi,z')}\,d\xi\bigg]\lambda^2\,d\lambda\,.
\end{equation*}
\end{lemma}
\begin{proof}
According to \cite[Proposition~4.5]{Borthwick2016}, we have
\begin{equation*}
R_{\hspace{1pt}\H}\bigl(\tfrac12-i\lambda;z,e^{k \ell}z'\bigr)
-R_{\hspace{1pt}\H}\bigl(\tfrac12+i\lambda;z,e^{k\ell}z'\bigr)
=i2\lambda\int_{-\infty}^{\hspace{.5pt}\infty}\!E(\lambda,\xi,z)\,
\overline{E\bigl(\lambda,\xi,e^{k\ell}z'\bigr)}\,d\xi\,.
\end{equation*}
By summing up over $k\in\Z$ and by using Stone's formula \eqref{Stone}, we obtain
\begin{equation}\label{Expression}
d\hspace{1pt}\Pi_{\hspace{1pt}C_{\ell}}(\lambda;z,z')=\tfrac2\pi\bigg[
\int_{-\infty}^{\hspace{.5pt}\infty}\!E(\lambda,\xi,z)\,
\overline{E_{\hspace{1pt}C_\ell}(\lambda,\xi,z')}\,d\xi\bigg]\lambda^2\,d\lambda\,.
\end{equation}
We conclude by splitting up
$$
\int_{-\infty}^{\hspace{.5pt}\infty}\!d\xi=\sum_{k\in\Z}\int_{e^{k\ell}<|\xi|<e^{(k+1)\ell}}\!d\xi
$$
in \eqref{Expression} and by using
$$
E_{\hspace{1pt}C_\ell}\bigl(\lambda,e^{k\ell}\xi,z\bigr)
=e^{-\left(\frac12+\hspace{.5pt}i\lambda\right)k\ell}E_{\hspace{1pt}C_\ell}(\lambda,\xi,z)\,,
$$
which follows from
$$
E\bigl(\lambda,e^{k\ell}\xi,e^{k\ell}z\bigr)
=e^{-\left(\frac12+\hspace{.5pt}i\lambda\right)k\ell}E(\lambda,\xi,z)\,.
$$
\end{proof}

Next we treat the general case, which is covered in \cite[Chapter~7]{Borthwick2016}.
Consider a hyperbolic surface $X\!=\hspace{-1pt}\Gamma\backslash\mathbb{H}$ of infinite area with $m\hspace{-1pt}\ge\!1$ funnel ends, whose boundary geodesics have lengthes $\ell_1,\dots,\ell_m$, and with $n\hspace{-1pt}\ge\hspace{-1pt}0$ cusps. Then (see \cite[Chapter~7]{Borthwick2016}), there exist generalized $\Delta_X$--\hspace{1pt}eigenfunctions\footnote{\,Note that Borthwick \cite{Borthwick2016} uses rather the notation $E^{\hspace{.5pt}\textnormal{f}}_j(\frac12\hspace{-1pt}+\hspace{-1pt}i\lambda\hspace{1pt};z,\theta)$.} $E_j(\lambda,\theta,z)$ and finitely many $\Delta_X$--\hspace{1pt}eigenfunctions $F_k(z)$ with eigenvalues $\frac14\hspace{-1pt}-\hspace{-1pt}\lambda_k^2\!\in\hspace{-1pt}[\hspace{1pt}0,\frac14)$ such that the spectral measure of $D_X\hspace{-.5pt}=\hspace{-1pt}\sqrt{\hspace{-.5pt}\Delta_X\!-\hspace{-1pt}\smash{\tfrac14}\hspace{1pt}\vphantom{\big|}}$ has integral kernel
\begin{align*}
d\hspace{1pt}\Pi_X(\lambda;z,z')
&=\tfrac1{\pi^2}\sum_{j=1}^m\ell_j\hspace{1pt}\Bigl(\int_0^{\hspace{1pt}2\pi}\!E_j(\lambda,\theta,z)\,
\overline{E_j(\lambda,\theta,z')}\,d\theta\Bigr)\hspace{1pt}\lambda^2\hspace{1pt}d\lambda\\
&+\tfrac2\pi\sum_{k=1}^nF_k(z)\,\overline{F_k(z')}\,\lambda_k^2\,d\hspace{.5pt}
\delta_{i\lambda_k}(\lambda)\hspace{1pt}.
\end{align*}
This formula leads to the following expressions
for the Fourier transform, its inverse, and Plancherel's identity\,:
\begin{align}
\label{FourierX}&\widetilde{\mathcal{F}}_{\hspace{-.5pt}X}:\,\begin{cases}
\,\widetilde{f}_j(\lambda,\theta)=\int_Xf(z')\,\overline{E_j(\lambda,\theta,z')}\,dz'
&(1\!\le\hspace{-1pt}j\hspace{-1pt}\le\hspace{-1pt}m),\\
\,\widetilde{f}(i\lambda_k)=\int_Xf(z')\,\overline{F_k(z')}\,dz'
&(1\!\le\hspace{-1pt}k\hspace{-1pt}\le\hspace{-1pt}n),\end{cases}\\
\label{FourierInversionX}&f(z)
=\tfrac1{\pi^2}\sum_{j=1}^m\ell_j\int_0^{\hspace{.5pt}\infty}\!\int_0^{\hspace{.5pt}2\pi}\!
\widetilde{f}_j(\lambda,\theta)\,E_j(\lambda,\theta,z)\,d\theta\,\lambda^2d\lambda
+\tfrac2\pi\sum_{k=1}^n\widetilde{f}(i\lambda_k)\,F_k(z)\hspace{1pt},\\
\label{PlancherelX}&\|f\|^2_{L^2(X)}
=\tfrac1{\pi^2}\sum_{j=1}^m\ell_j\,
\bigl\|\hspace{1pt}\lambda\hspace{1pt}\widetilde{f}_j\hspace{.5pt}\bigr\|_{L^2_{\lambda,\theta}}^2\!
+\tfrac2\pi\sum_{k=1}^n|\widetilde{f}(i\lambda_k)|^2\hspace{1pt}.
\end{align}

As far as it is concerned, the functional calculus is still given by
$$
m(D_X)f\hspace{-.5pt}=\widetilde{\mathcal{F}}^{\hspace{1pt}-1}_{\hspace{-1pt}X}
\bigl(m\cdot\widetilde{\mathcal{F}}_{\hspace{-.5pt}X}f\hspace{1pt}\bigr)\hspace{1pt}.
$$
The kernel $K_X$ of $m(D_X)$ on $X$ satisfies by definition
$[\hspace{1pt}m(D_X)f\hspace{1pt}](z)\hspace{-1pt}=\hspace{-1pt}\int_XK_X(z,z')\,f(z')\,dz'$,
and is, at least formally given by the series
$$
K_X(\Gamma z,\Gamma z')
=\sum_{\gamma\in\Gamma}K_{\mathbb{H}}(d(z,\gamma\hspace{1pt}z')\hspace{-1pt})\hspace{1pt}.
$$

Let us rewrite the continuous part of
\eqref{FourierX}, \eqref{FourierInversionX}, \eqref{PlancherelX}
in terms of the spectral projectors
$\mathbf{P}^{\hspace{.5pt}X}_{\hspace{-1pt}\lambda}\hspace{-1pt}=\hspace{-.5pt}\delta_\lambda(D_X)$,
restriction operators $\mathbf{R}^{\hspace{.5pt}X}_{\hspace{.5pt}\lambda}$ and
extension\footnote{\,The extension operator is called \textit{Poisson operator\/} in~\cite{Borthwick2016}.}
operators $\mathbf{E}^{\hspace{.5pt}X}_{\hspace{.5pt}\lambda}$,
which are defined for $\lambda\hspace{-1pt}\ge\hspace{-1pt}0$ by
\begin{align*}
&\mathbf{P}^{\hspace{.5pt}X}_{\hspace{-1pt}\lambda}\!f(z)
=\tfrac{\lambda^2}{\pi^2}\sum_{j=1}^m\ell_j\int_X
\int_0^{2\pi}\!E_j(\lambda,\theta,z)\,\overline{E_j(\lambda,\theta,z')}\,f(z')\,d\theta\,dz'\\
&\mathbf{E}^{\hspace{.5pt}X}_{\hspace{.5pt}\lambda}\hspace{-1pt}F(z)
=\tfrac\lambda\pi\sum_{j=1}^m\sqrt{\ell_j}
\int_0^{2\pi}\!F_j(\theta)\,E_j(\lambda,\theta,z)\,d\theta,\\
&\mathbf{R}^{\hspace{.5pt}X}_{\hspace{.5pt}\lambda,j}f(\theta)
=\tfrac{\lambda}\pi\sqrt{\ell_j}\int_Xf(z')\,\overline{E_j(\lambda,\theta,z')}\,dz'.
\end{align*}
These operators enjoy
$$
d\hspace{1pt}\Pi_X(\lambda)=\mathbf{P}^{\hspace{.5pt}X}_{\hspace{-1pt}\lambda}d\lambda\,,
\quad\mathbf{P}^{\hspace{.5pt}X}_{\hspace{-1pt}\lambda}\!
=\mathbf{E}^{\hspace{.5pt}X}_{\hspace{.5pt}\lambda}
\mathbf{R}^{\hspace{.5pt}X}_{\hspace{.5pt}\lambda}
\quad\mbox{and}\quad
\mathbf{R}^{\hspace{.5pt}X}_{\hspace{.5pt}\lambda}\!
=\hspace{-1pt}\bigl(\mathbf{E}^{\hspace{.5pt}X}_{\hspace{.5pt}\lambda}\bigr)^*,
\quad\mathbf{E}^{\hspace{.5pt}X}_{\hspace{.5pt}\lambda}\!
=\hspace{-1pt}\bigl(\mathbf{R}^{\hspace{.5pt}X}_{\hspace{.5pt}\lambda}\bigr)^*.
$$

Let us finally point out some relations between operator norms of
\hspace{1pt}$\mathbf{P}^{\hspace{.5pt}X}_{\hspace{-1pt}\lambda}$\hspace{-1pt},
$\mathbf{E}^{\hspace{.5pt}X}_{\hspace{.5pt}\lambda}$\hspace{-1pt},
$\mathbf{R}^{\hspace{.5pt}X}_{\hspace{.5pt}\lambda}$
\hspace{-1pt}and $P_{\lambda,\eta}^X\hspace{1pt}.$
It follows from
$$
P^{\hspace{.5pt}X}_{\lambda,\eta}\hspace{1pt}=\hspace{-1pt}
\int_{\lambda-\eta}^{\lambda+\eta}\!\mathbf{P}^{\hspace{.5pt}X}_{\hspace{-1pt}\mu}\,d\mu
\quad\text{and}\quad
\mathbf{P}^{\hspace{.5pt}X}_{\hspace{-1pt}\lambda}=\lim_{\eta\hspace{1pt}\searrow\hspace{1pt}0}
\tfrac1{2\hspace{1pt}\eta}\hspace{1pt}P^{\hspace{.5pt}X}_{\lambda,\eta}
$$
that
\begin{equation*}
\|P^{\hspace{.5pt}X}_{\lambda,\eta}\|_{L^{p^{\hspace{.5pt}\prime}}\!\to L^p}
\le\int_{\lambda-\eta}^{\lambda+\eta}
\|\mathbf{P}^{\hspace{.5pt}X}_{\hspace{-1pt}\mu}\|_{L^{p^{\hspace{.5pt}\prime}}\to L^p}\,d\mu
\quad\text{and}\quad
\left\|\mathbf{P}^{\hspace{.5pt}X}_{\hspace{-1pt}\lambda}\right\|_{L^{p^{\hspace{.5pt}\prime}}\to L^p}\le\liminf_{\eta\hspace{1pt}\searrow\hspace{1pt}0}\tfrac1{2\hspace{1pt}\eta}\hspace{1pt}
\|P^{\hspace{.5pt}X}_{\lambda,\eta}\|_{L^{p^{\hspace{.5pt}\prime}}\!\to L^p}
\end{equation*}
for every $\lambda\hspace{-1pt}>\!1\!>\hspace{-1pt}\eta\hspace{-1pt}>\hspace{-1pt}0$
and $p\hspace{-1pt}>\hspace{-1pt}2$\hspace{1pt}.
Next result is deduced by combining these inequalities with the usual \hspace{1pt}$TT^*$ argument.

\begin{proposition}\label{EquivalentOperatorNormEstimates}
The following estimates are equivalent for
\,$\lambda\hspace{-1pt}>\!1\!>\hspace{-1pt}\eta\hspace{-1pt}>\hspace{-1pt}0$
$($with implicit constants only depending on $p\hspace{-1pt}>\hspace{-1pt}2\hspace{1pt}):$
\begin{itemize}
\item[\rm(i)]
$\|\hspace{.5pt}P^{\hspace{.5pt}X}_{\lambda,\eta}\hspace{.5pt}\|_{L^{p^{\hspace{.5pt}\prime}}\!\to L^p}
\lesssim\lambda^{2 \gamma(p)}\hspace{1pt}\eta$\hspace{1pt},
\item[\rm(ii)]
$\|\hspace{.5pt}P^{\hspace{.5pt}X}_{\lambda,\eta}\hspace{.5pt}\|_{L^2\to L^p}
\lesssim\lambda^{\gamma(p)}\hspace{1pt}\eta^{\frac12}$\hspace{1pt},
\item[\rm(iii)]
$\|\hspace{1pt}\mathbf{P}^{\hspace{.5pt}X}_{\hspace{-1pt}\lambda}\|_{L^{p^{\hspace{.5pt}\prime}}\to L^p}\lesssim\lambda^{2\gamma(p)}$,
\item[\rm(iv)]
$\|\hspace{1pt}\mathbf{E}^{\hspace{.5pt}X}_\lambda\|_{L^2\to L^p}
=\|\hspace{1pt}\mathbf{R}^{\hspace{.5pt}X}_\lambda\|_{L^{p^{\hspace{.5pt}\prime}}\!\to L^2}
\lesssim\lambda^{\gamma(p)}$\hspace{1pt}.
\end{itemize}
\end{proposition}


\section{Lower bounds on $\mathbb{H}$}\label{lowerb}

\subsection{A general lower bound}\label{aglb}

\begin{proposition}
For any complete Riemannian manifold with uniformly bounded geometry,
there exists ${\eta_0}>0$ such that, for any $\lambda>{1}$ and $\eta<1$,
$$
\sup_{\mu\in[\lambda-{\eta_0},\lambda+{\eta_0}]}\|P_{\mu,\eta}\|_{L^{2}\to L^p}
\gtrsim \lambda^{\gamma(p)} \eta^{1/2}.
$$
\end{proposition}

\begin{proof}
By the classical $TT^*$ argument, Sogge's estimate \eqref{soggestatement} amounts to
\begin{equation}\label{SoggeEstimate}
{\|P_{\lambda,\eta_0}\|_{L^{p^{\hspace{.5pt}\prime}}\to L^p}}\sim \lambda^{2 \gamma(p)}
\end{equation}
for $\lambda$ large, say $\lambda>1$.
We are concerned with $\eta$ small and we may therefore assume that $\eta\le\tfrac{\eta_0}4$.
Let $N=\lceil\tfrac{\eta_0}{2\eta}\rceil$ and $c=\tfrac{\eta_0}{2\eta N}$.
Notice that $N\ge2$ and $\tfrac12<c\le1$.
Let us split up $[\lambda-\eta_0,\lambda+\eta_0]$ into
$N$ disjoint intervals $I_j$ of length $2c\eta$ centered at $\mu_j$. Then
$$
P_{\lambda,\eta_0}=\sum\nolimits_{j=1}^NP_{I_j}
$$
is the sum of the spectral projector associated with the intervals $I_j$.
Hence
\begin{equation}\label{UpperEstimate}\begin{aligned}
\|P_{\lambda,\eta_0}\|_{L^{p^{\hspace{.5pt}\prime}}\!\to L^p}\!
\le\sum_{j=1}^N\|P_{I_j}\|_{L^{p^{\hspace{.5pt}\prime}}\!\to L^p}\!
\le\sum_{j=1}^N\|P_{\mu_j,\eta}\|_{L^{p^{\hspace{.5pt}\prime}}\!\to L^p}\!
\le N\hspace{-2mm}\sup_{\mu\in[\lambda-\eta_0,\lambda+\eta_0]}
\hspace{-1mm}\|P_{\mu,\eta}\|_{L^{p^{\hspace{.5pt}\prime}}\!\to L^p}.
\end{aligned}\end{equation}
By using \eqref{SoggeEstimate}, \eqref{UpperEstimate} and the fact that $N\sim\tfrac1\eta$,
we deduce that
$$
\sup_{\mu\in[\lambda-\eta_0,\lambda+\eta_0]}\|P_{{\mu},\eta}\|_{L^{p^{\hspace{.5pt}\prime}}\to L^p}
\gtrsim\lambda^{2 \gamma(p)}\,\eta\hspace{1pt}.
$$
Applying once again the $TT^*$ argument, {we conclude that}
$$
\sup_{\mu\in[\lambda-\eta_0,\lambda+\eta_0]}\|P_{{\mu},\eta}\|_{L^2\to L^p}
\gtrsim \lambda^{\gamma(p)}\,\eta^{\frac12}\hspace{1pt}.
$$
\end{proof}

In the remainder of this section we prove lower bounds on the $L^2-L^p$ norm of spectral projectors in the case of the real hyperbolic space. Note that the results differ from those of \cite{GermainLeger2023}. Indeed we prove the sharpness of our estimates on spectral projectors and not on the restriction operator. In particular they are not mere corollaries of \cite{GermainLeger2023}. 

\subsection{The spherical example for the hyperbolic plane}

\begin{proposition}\label{SphericalExample}
Assume that \,$\lambda\hspace{-1pt}\ge\hspace{-1pt}0$\hspace{1pt},
$0\hspace{-1pt}<\hspace{-1pt}\eta\hspace{-1pt}<\!1$
and \,$p\hspace{-1pt}>\hspace{-1pt}2$\hspace{1pt}.
Then
$$
\|P_{\lambda,\eta}\|_{L^2 \to L^p}\gtrsim\begin{cases}
\,\lambda^{\frac12-\frac2p}\,\eta^{\frac12}
&\text{if $\lambda\hspace{-1pt}>\!1$,}\\
\hspace{1pt}(\lambda\hspace{-1pt}+\hspace{-1pt}\eta)\,\sqrt{\eta}
&\text{if $0\hspace{-1pt}\le\hspace{-1pt}\lambda\hspace{-1pt}\le\!1$.}\\
\end{cases}$$
\end{proposition}

\begin{remark}
Proposition \ref{SphericalExample} will serve to discuss the sharpness of the upper bounds
in Proposition \ref{LambdaSmallH} for small $\lambda$
and in Theorem \ref{L2LpBoundsProjectorsH} when $\lambda$ is large
and $2\hspace{-1pt}<\hspace{-1pt}p\hspace{-1pt}\le\hspace{-1pt}6$\hspace{.5pt}.
\end{remark}

\begin{proof}
Consider the radial function $f$ on $\H$ with spherical Fourier transform
$$
\widetilde{f}=\mathbf{1}_{[\lambda-\frac\eta2,\lambda+\frac\eta2]}
+\mathbf{1}_{[-\lambda-\frac\eta2,-\lambda+\frac\eta2]}\,.
$$
On the one hand, by Plancherel's theorem for the spherical Fourier transform and~\eqref{equivalents},
$$
\|f\|_{L^2(\H)}=\Bigl(\tfrac1{2\hspace{.5pt}\pi^2}\!
\int_0^{\hspace{.5pt}\infty}\!|\widetilde{f}(\mu)|^2\,
|\mathbf{c}(\mu)|^{-2}\,d\mu\Bigr)^{\hspace{-1pt}\frac12}
\sim\,\begin{cases}
\sqrt{\lambda\hspace{1pt}\eta\hspace{1pt}}
&\text{if $\lambda\hspace{-1pt}>\!1$,}\\
\hspace{.5pt}(\lambda\hspace{-1pt}+\hspace{-1pt}\eta)\sqrt{\eta\hspace{1pt}}
&\text{if \hspace{1pt}$0\hspace{-1pt}\le\hspace{-1pt}\lambda\hspace{-1pt}\le\!1$.}\\
\end{cases}
$$
On the other hand,
when \hspace{1pt}$r\hspace{-1pt}<\hspace{-1pt}\tfrac1{\lambda+\frac\eta2}$\hspace{.5pt},
we deduce from \eqref{equivalents} that
$$
f(r)\sim\varphi_0(r)\int_0^{\hspace{.5pt}\infty}\!\widetilde{f}(\mu)\,|\mathbf{c}(\mu)|^{-2}\,d\mu
\sim\begin{cases}
\,\lambda\,\eta
&\text{if $\lambda\hspace{-1pt}>\!1$,}\\
\hspace{1pt}(\lambda\hspace{-1pt}+\hspace{-1pt}\eta)^2\,\eta\,(1\!+\hspace{-1pt}r)\,e^{-\frac r2}
&\text{if $0\hspace{-1pt}\le\hspace{-1pt}\lambda\hspace{-1pt}\le\!1$.}\\
\end{cases}$$
Hence
$$
\|f\|_{L^p(\H)}\ge\|f\|_{L^p(B(0,\frac1{\lambda+\frac\eta2}))}
\gtrsim\begin{cases}
\,\lambda^{1-\frac2p}\,\eta
&\text{if $\lambda\hspace{-1pt}>\!1$,}\\
\hspace{1pt}(\lambda\hspace{-1pt}+\hspace{-1pt}\eta)^2\,\eta
&\text{if $0\hspace{-1pt}\le\hspace{-1pt}\lambda\hspace{-1pt}\le\!1$.}\\
\end{cases}$$
Finally, since $P_{\lambda,\eta}f=f$,
$$
\|P_{\lambda,\eta}\|_{L^2\to L^p}\ge\frac{\|f\|_{L^p}}{\|f\|_{L^2}}
\gtrsim\begin{cases}
\,\lambda^{\frac12-\frac2p}\,\eta^{\frac12}
&\text{if $\lambda\hspace{-1pt}>\!1$,}\\
\hspace{1pt}(\lambda\hspace{-1pt}+\hspace{-1pt}\eta)\,\sqrt{\eta}
&\text{if $0\hspace{-1pt}\le\hspace{-1pt}\lambda\hspace{-1pt}\le\!1$.}\\
\end{cases}$$
\end{proof}

\subsection{The Knapp example}

\begin{proposition}\label{KnappExample}
Assume that \,$\lambda\hspace{-1pt}>\!1$,
$\eta\hspace{-.5pt}\lesssim\hspace{-1.5pt}\frac1{\log\lambda}$
and \,$p\hspace{-1pt}-\hspace{-1pt}2\gtrsim\hspace{-1pt}\frac1{\log\lambda}$.
Then
$$
\|P_{\lambda,\eta}\|_{L^2\to L^p}\gtrsim\tfrac1{\sqrt{p-2}}\,
\lambda^{\frac12(\frac12-\frac1p)}\hspace{1pt}\eta^{\frac12}\hspace{.5pt}.
$$
\end{proposition}

\begin{remark}
The condition on \hspace{1pt}$p$ \hspace{1pt}is harmless,
as Proposition~\ref{KnappExample} will serve to discuss the sharpness of
the upper bound in Theorem \ref{L2LpBoundsProjectorsH}
when \hspace{1pt}$p\hspace{-1pt}\ge\hspace{-1pt}6$\hspace{.5pt}.
\end{remark}

\begin{proof}
Define $f=f_{\lambda,\eta}$ on $\H$ by its Fourier transform
$$
\widetilde{f}(\mu,\xi)=\mathbf{1}_{[\lambda-\eta,\lambda+\eta]}(\mu)\,\mathbf{1}_{[-1,1]}(\xi)\hspace{1pt}.
$$
Then, by the Plancherel formula,
\begin{equation}
\label{L2norm}
\|f\|_{L^2(\H)}
=\tfrac2\pi\,\|\hspace{.5pt}\mu\hspace{.5pt}\widetilde{f}\hspace{1pt}\|_{L^2_{\mu,\xi}}
\sim\lambda\,\eta^{\frac12}\hspace{.5pt}.
\end{equation}
By the inverse Fourier transform, we get on the physical space side
\begin{equation}
\label{equationf}
f(z)=\tfrac2\pi\int_{\lambda-\eta}^{\lambda+\eta}\!
\int_{-1}^{\hspace{.5pt}1}\!E\bigl(\mu,\xi,z)\,d\xi\,\mu^2\,d\mu\,,
\end{equation}
where the generalized eigenfunction can be rewritten
\begin{equation}\label{OscillatingExponential}
E (\mu,\xi,z)=\mathbf{C}(\mu)\,
\Bigl(\tfrac y{(x\hspace{1pt}-\hspace{1pt}\xi)^2+\hspace{1pt}y^2}\Bigr)^{\hspace{-1pt}\frac12}
e^{\hspace{.5pt}i\hspace{1pt}\mu\log\left(\frac{y}{(x-\xi)^2+y^2} \right)}\hspace{.5pt}.
\end{equation}
Let us restrict to the region
\begin{equation}\label{condition1}
1\ll x\ll y
\end{equation}
in $\H$\hspace{.5pt},
where
\begin{equation*}
\tfrac y{(x\hspace{1pt}-\hspace{1pt}\xi)^2+\hspace{1pt}y^2}\sim\tfrac1y\,,
\end{equation*}
hence
\begin{equation*}
\bigl|E\bigl(\mu,\xi,z)\bigr|\sim\mu^{-\frac12}\hspace{1pt}y^{-\frac12}\hspace{.5pt}.
\end{equation*}
We want to determine what is the range of $z$ for which the phase in \eqref{OscillatingExponential} is not oscillating over the domain of integration in~\eqref{equationf} - or to be more explicit, when the phase does not vary by more than $\ll 1$ on the domain of integration. This is the case if
\vspace{5pt}

\noindent$\bullet$
\,$\eta\left|\hspace{1pt}\partial_\mu\!\left[\hspace{.5pt}\mu
\log\left(\tfrac y{(x\hspace{1pt}-\hspace{1pt}\xi)^2+\hspace{1pt}y^2}\right)
\right]\right|\ll1$\hspace{.5pt},
which follows from 
\begin{equation}
\label{condition2}
\eta\hspace{1pt}\log y\ll1\hspace{1pt},
\end{equation}
\noindent$\bullet$
\,$\left|\hspace{1pt}\partial_\xi\!\left[\hspace{.5pt}\mu
\log\left(\tfrac y{(x\hspace{1pt}-\hspace{1pt}\xi)^2+\hspace{1pt}y^2}\right)
\right]\right|\ll1$\hspace{.5pt},
which follows from 
\begin{equation}
\label{condition3}
\mu\,x\ll y^2.
\end{equation}
Under the conditions~\eqref{condition1},~\eqref{condition2},~\eqref{condition3}, we find that
\begin{equation*}
|f(z)|\sim\int_{\lambda-\eta}^{\lambda+\eta}\!\int_{-1}^{\hspace{.5pt}1}
\bigl|E\bigl(\mu,\xi,z)\bigr|\,d\xi\,\mu^2\,d\mu
\sim\lambda^{\frac32}\hspace{1pt}y^{-\frac12}\!
\int_{\lambda-\eta}^{\lambda+\eta}\!\int_{-1}^{\hspace{.5pt}1}d\xi\,d\mu
\sim\lambda^{\frac32}\hspace{1pt}\eta\,y^{-\frac12}\hspace{1pt}.
\end{equation*}
Notice that the conditions~\eqref{condition1},~\eqref{condition2},~\eqref{condition3} imply that $y\hspace{-1pt}>\!\sqrt{\lambda\hspace{1pt}}$. Furthermore, since we are assuming $\eta\hspace{-.5pt}\lesssim\hspace{-1.5pt}\tfrac1{\log\lambda}$, we can integrate in $y$ over the interval $\bigl[\sqrt{\lambda\hspace{.5pt}},\hspace{-.5pt}\lambda\hspace{.5pt}\bigr]$ and obtain this way
\begin{align*}
\|f\|_{L^p(\H)}
&\gtrsim\lambda^{\frac32}\,\eta\,\biggl(\hspace{1pt}
\iint_{\substack{\sqrt{\lambda\hspace{1pt}}\le\hspace{1pt}y\hspace{1pt}\le\hspace{1pt}\lambda\\
1\hspace{.5pt}<\hspace{1pt}x\hspace{1pt}<\hspace{1pt}y^2\hspace{-1pt}/\hspace{-.5pt}\lambda}}
y^{-\frac p2}\,\tfrac{dx\hspace{1pt}dy}{y^2}\biggr)^{\!\frac1p}\!
\sim\lambda^{\frac32}\,\eta\,\biggl(\tfrac1\lambda
\int_{\sqrt\lambda}^{\hspace{.5pt}\lambda}\,y^{-\frac p2}\,dy\biggr)^{\!\frac1p}\\
&\sim\lambda^{\frac32}\,\eta\,\biggl(
\frac{\lambda^{-\frac12-\frac p4}\hspace{-1.5pt}-\hspace{-1pt}\lambda^{-\frac p2}}
{p\hspace{-1pt}-\hspace{-1pt}2}\biggr)^{\hspace{-1pt}\frac1p}\!
=(p\hspace{-1pt}-\hspace{-1pt}2)^{-\frac1p}\,\lambda^{\frac54-\frac1{2p}}\,\eta\,
\bigl(1\!-\hspace{-1pt}\lambda^{-\frac{p-2}4}\bigr)^{\frac1p}\hspace{.5pt},
\end{align*}
with $(p\hspace{-1pt}-\hspace{-1pt}2)^{-\frac1p}\!
\gtrsim\hspace{-1pt}(p\hspace{-1pt}-\hspace{-1pt}2)^{-\frac12}$,
so that
$$
\|f\|_{L^p}\gtrsim\tfrac1{\sqrt{p-2}}\,\lambda^{\frac54-\frac1{2p}}\,\eta
\qquad\text{if }\hspace{1pt}
(p\hspace{-1pt}-\hspace{-1pt}2)\log\lambda\hspace{-1pt}\gtrsim\!1\hspace{.5pt}.
$$
Recalling~\eqref{L2norm}, we conclude that
$$
\frac{\|f\|_{L^p}}{\|f\|_{L^2}}\gtrsim
\tfrac1{\sqrt{p-2}}\,\lambda^{\frac14-\frac1{2p}}\,\eta^{\frac12}\hspace{.5pt}.
$$
\end{proof}

\section{Upper bounds on $\mathbb{H}$} \label{upperbplane}

In the sequel we use to say that a variable,
such as the time $t$ or the distance $r$ to the origin,
is small or large depending whether it is smaller or larger to $1$.
Note that this lemma is already contained in \cite{GermainLeger2023}. However we have decided to keep this statement here since the proof will be refined so that it generalizes to convex cocompact hyperbolic surfaces with $0 \le \delta < \frac12.$ In particular the pointwise estimates obtained on kernels that appear in the proof are more precise, as we will remark below.
\begin{theorem}[high frequency]\label{L2LpBoundsProjectorsH}
The spectral projectors on \,$\H$ enjoy the following bounds for
\,$0\hspace{-1pt}<\hspace{-1pt}\eta\hspace{-1pt}<\!1\!<\hspace{-1pt}\lambda$
and \,$p\hspace{-1pt}>\hspace{-1pt}2:$
$$
\|P_{\lambda,\eta}\|_{L^2\to L^p}
\lesssim\bigl(1\!+\!\tfrac1{\sqrt{p-2}}\bigr)\hspace{1pt}
\lambda^{\gamma(p)}\,\eta^{\frac12}\hspace{1pt}.
$$
\end{theorem}

\begin{remark}
Notice that, in the range
\hspace{1pt}$2\hspace{-1pt}<\hspace{-1pt}p\hspace{-1pt}<\hspace{-1pt}6$\hspace{1pt},
the power \hspace{1pt}$\eta^{\frac12}$ for the hyperbolic plane $\H$ is better than the power
\hspace{1pt}$\eta^{\frac32(\frac12-\frac1p)}$ for the Euclidean plane $\mathbb{R}^2$.
\end{remark}

\begin{proof}
We estimate successively
\begin{enumerate}
\item
the spectral projectors
$$
\mathscr{P}_{\lambda,\eta}=\chi\bigl(\tfrac{D-\lambda}\eta\bigr)+\chi\bigl(\tfrac{D+\lambda}\eta\bigr)
$$
associated with an even Schwartz function $\chi$ on $\R$\hspace{1pt}, whose Fourier transform satisfies
$$
\widehat{\chi}\hspace{-.5pt}=\hspace{-1pt}1\text{ \,on \,}[-1,1\hspace{.5pt}]
\quad\text{and}
\quad\supp\widehat{\chi}\hspace{-1pt}\subset\hspace{-1pt}[-2,2\hspace{1pt}]\hspace{1pt},
$$
\item
the spectral projectors $\mathbf{P}_{\hspace{-1pt}\lambda}f=\delta_\lambda(D)f
=|\mathbf{c}(\lambda)|^{-2}\hspace{1pt}f*\varphi_\lambda$\hspace{1pt},
\item
the spectral projectors $P_{\lambda,\eta}$\hspace{1pt}.
\end{enumerate}
Notice that
$$
\|P_{\lambda,\eta}\|_{L^2\to L^p}=\|P_{\lambda,\eta}^2\|_{L^{p^{\hspace{.5pt}\prime}}\to L^p}^{1/2}
$$
according to the $TT^*$ argument.
\smallskip

\noindent
(1) In this item, we allow $0<\eta\le\lambda$\hspace{1pt}.
Consider $\psi=\chi-\tfrac12\hspace{1pt}\chi(\frac12\,\cdot\,)$,
whose Fourier transform
$\widehat{\psi}=\widehat{\chi}-\widehat{\chi}(2\,\cdot\,)$
is supported in $[-2,-\frac12]\cup[\frac12,2]$,
and the associated spectral projectors
$\mathscr{Q}_{\lambda,\eta}=\psi\bigl(\tfrac{D-\lambda}\eta\bigr)+\psi\bigl(\tfrac{D+\lambda}\eta\bigr)$.
Denote by $p_{\lambda,\eta}$ and $q_{\lambda,\eta}$
the radial convolution kernels of $\mathscr{P}_{\lambda,\eta}$ and $\mathscr{Q}_{\lambda,\eta}$.
It follows from Lemma \ref{busard} below that
\begin{equation*}\begin{cases}
\,\|\mathscr{P}_{\lambda,\eta}\|_{L^1\to L^{\hspace{.5pt}\infty}}
=\|p_{\lambda,\eta}\|_{L^{\hspace{.5pt}\infty}}
\lesssim\lambda\,\eta,\\
\,\|\mathscr{Q}_{\lambda,\eta}\|_{L^1\to L^{\hspace{.5pt}\infty}}
=\|q_{\lambda,\eta}\|_{L^{\hspace{.5pt}\infty}}
\lesssim\lambda^{\frac12}\,\eta\,(1\!+\hspace{-1pt}\eta)^{\frac12}\,e^{-\frac\epsilon{8\eta}},
\end{cases}\end{equation*}
for any fixed $0<\epsilon<1$.

By interpolation with the trivial estimates
\begin{equation*}\begin{cases}
\,\|\mathscr{P}_{\lambda,\eta}\|_{L^2\to L^2}\lesssim1,\\
\,\|\mathscr{Q}_{\lambda,\eta}\|_{L^2\to L^2}\lesssim1,
\end{cases}\end{equation*}
we deduce that
\begin{equation}\label{FirstEstimateLpprimeLp}\begin{cases}
\,\|\mathscr{P}_{\lambda,\eta}\|_{L^{p^{\hspace{.5pt}\prime}}\to L^p}
\lesssim\lambda^{1-\frac2p}\hspace{1pt}\eta^{1-\frac2p},\\
\,\|\mathscr{Q}_{\lambda,\eta}\|_{L^{p^{\hspace{.5pt}\prime}}\to L^p}
\lesssim\lambda^{\frac12-\frac1p}\hspace{1pt}
\eta^{1-\frac2p}\hspace{1pt}(1\!+\hspace{-1pt}\eta)^{\frac12-\frac1p}\hspace{1pt}
e^{-(\frac12-\frac1p)\frac\epsilon{4\eta}}.
\end{cases}\end{equation}
(2) The key idea consists in writing
\begin{equation*}
\delta_{\pm\lambda}(\xi\hspace{-1pt}\mp\hspace{-1pt}\lambda)
=2^{k_0}\hspace{1pt}\chi\bigl(2^{k_0}(\xi\hspace{-1pt}\mp\hspace{-1pt}\lambda)\bigr)
+\sum\nolimits_{k>k_0}2^k\hspace{1pt}\psi\bigl(2^k(\xi\hspace{-1pt}\mp\hspace{-1pt}\lambda)\bigr)\,,
\end{equation*}
hence
\begin{equation}\label{KeyFormula}
\mathbf{P}_{\hspace{-1pt}\lambda}=2^{k_0}\mathscr{P}_{\lambda,2^{-k_0}}
+\sum\nolimits_{k>k_0}2^k\,\mathscr{Q}_{\lambda,2^{-k}}\,,
\end{equation}
where $k_0=-\lceil\tfrac{\log\lambda}{\log 2}\rceil$, i.e.,
$-k_0$ is the smallest positive integer such that $2^{-k_0}\ge\lambda$ (\footnote{\,Hence \,$\lambda\sim2^{-k_0}$.}).
Then
\begin{align*}
\|\hspace{1pt}\mathbf{P}_{\hspace{-1pt}\lambda}\|_{L^{p^{\hspace{.5pt}\prime}}\to L^p}
&\le2^{k_0}\|\mathscr{P}_{\lambda,2^{-k_0}}\|_{L^{p^{\hspace{.5pt}\prime}}\to L^p}\\
&+\|\underbrace{\sum\nolimits_{k_0<k\le0}2^k\,\mathscr{Q}_{\lambda,2^{-k}}}_{\Sigma_0}\|_{L^{p^{\hspace{.5pt}\prime}}\to L^p}
+\|\underbrace{\sum\nolimits_{k>0}2^k\,\mathscr{Q}_{\lambda,2^{-k}}}_{\Sigma_\infty}\|_{L^{p^{\hspace{.5pt}\prime}}\to L^p}\,,
\end{align*}
where, according to \eqref{FirstEstimateLpprimeLp},
\begin{gather}
2^{k_0}\hspace{1pt}\|\mathscr{P}_{\lambda,2^{-k_0}}\|_{L^{p^{\hspace{.5pt}\prime}}\to L^p}
\lesssim\lambda^{1-\frac2p}\hspace{1pt}2^{\frac2pk_0}\sim\lambda^{1-\frac4p},
\label{FirstTerm}\\\begin{aligned}
\|\Sigma_0\|_{L^{p^{\hspace{.5pt}\prime}}\to L^p}
&\le\sum\nolimits_{k_0<k\le0}\overbrace{
2^k\,\|\mathscr{Q}_{\lambda,2^{-k}}\|_{L^{p^{\hspace{.5pt}\prime}}\to L^p}
}^{\lambda^{\frac 1 2\hspace{.2mm}- \frac 1 p}\,2^{\hspace{.2mm}3\hspace{.2mm}(\frac 1 p\hspace{.2mm}- \frac 1 6)k}}\\
&\lesssim\begin{cases}
\,\lambda^{\frac12-\frac1p}
&\text{if \,}2<p<6,\\
\,\lambda^{\frac13}\hspace{1pt}(-k_0)\sim
(1+\log\lambda)\,\lambda^{\frac13}
&\text{if \,}p=6,\\
\,\lambda^{\frac12-\frac1p}\,2^{\hspace{.2mm}3\hspace{.2mm}(\frac1p-\frac16)k_0}
\sim\lambda^{1-\frac4p}
&\text{if \,}p>6,\\
\end{cases}\end{aligned}\label{SigmaZero}
\end{gather}
and, finally,
\begin{equation}
\begin{split}
\|\Sigma_\infty\|_{L^{p^{\hspace{.5pt}\prime}}\to L^p}
& \le\sum\nolimits_{k>0}\underbrace{2^k\,\|\mathscr{Q}_{\lambda,2^{-k}}\|_{L^{p^{\hspace{.5pt}\prime}}\to L^p}
}_{\lambda^{\frac 1 2\hspace{.5pt}- \frac 1 p}\,2^{\hspace{.5pt}2k/p}\hspace{1pt}e^{-(\frac 1 2 \hspace{.5pt}- \frac 1 p)\hspace{.5pt}\epsilon\hspace{.5pt}2^{k-2}}} \\
&\lesssim\lambda^{\frac12-\frac1p}\sum\nolimits_{k>0}2^k\hspace{1pt}e^{-(\frac12-\frac1p)\frac\epsilon42^k}\\
&\lesssim\lambda^{\frac12-\frac1p}\!\int_1^{\hspace{.5pt}\infty}\!e^{-(\frac12-\frac1p)\frac\epsilon4x}\,dx
\lesssim(p\hspace{-.5pt}-\hspace{-1pt}2)^{-1}\lambda^{\frac12-\frac1p}\,.
\end{split}
\label{SigmaInfinity}
\end{equation}
Hence
\begin{equation*}
\|\hspace{1pt}\mathbf{P}_{\hspace{-1pt}\lambda}\|_{L^{p^{\hspace{.5pt}\prime}}\to L^p}\lesssim\begin{cases}
\frac1{p-2}\,\lambda^{\frac12-\frac1p}
&\text{if \,}2<p<6,\\
(1+\log\lambda)\,\lambda^{\frac13}
&\text{if \,}p=6,\\
\lambda^{1-\frac4p}
&\text{if \,}p>6.
\end{cases}\end{equation*}

Finally we may remove the logarithmic factor in the limit case $p=6$
by considering the analytic family of operators
\begin{equation}\label{AnalyticFamily}
R_{\lambda,z}
=\bigl(2^z-\tfrac12\bigr)
\sum\nolimits_{k_0<k\le0}2^{(z+1)k}\,\mathscr{Q}_{\lambda,2^{-k}}
\end{equation}
in the vertical strip $-1\le\Re z\le\frac12$ and by applying Stein's interpolation theorem
(see for instance \cite[Ch.~5, Theorem~4.1]{SteinWeiss1971}).
Let us elaborate\,:

\noindent$\bullet$
\,$R_{\lambda,z}$ is \,$i\tfrac{2\pi}{\log 2}\Z$\,--\,periodic in $z$.

\noindent$\bullet$ If $z =0$, $R_{\lambda,0} = \frac12 \Sigma_0$.

\noindent$\bullet$
\,If $\Re z=-1$, the $L^2\to L^2$ norm of $R_{\lambda,z}$ is controlled by the $L^{\hspace{.5pt}\infty}$ norm of
\begin{equation*}
\theta_{k_0,\hspace{1pt}\Im\hspace{-1pt}z}(\mu)
=\bigl(2^{\hspace{1pt}i\Im\hspace{-1pt}z}-1\bigr)\sum\nolimits_{k_0<k\le0}
2^{\hspace{1pt}i\hspace{1pt}(\Im\hspace{-1pt}z)\hspace{1pt}k}\,\psi(2^k\mu),
\end{equation*}
which is a smooth even function $\R$, with
\begin{align*}
\theta_{k_0,\hspace{1pt}\Im\hspace{-1pt}z}(\mu)
&=2^{\hspace{1pt}i\Im z}\,\psi(\mu)
-2^{\hspace{1pt}i\hspace{1pt}(\Im\hspace{-1pt}z)(k_0+1)}\,\psi(2^{k_0+1}\mu)\\
&-\sum\nolimits_{k_0+1<k\le0}2^{\hspace{1pt}i\hspace{1pt}(\Im\hspace{-1pt}z)\hspace{1pt}k}\,
\underbrace{\bigl[\psi(2^k\mu)-\psi(2^{k-1}\mu)\bigr]}_{\int_{2^{k-1}\mu}^{2^k\mu}\psi'(\nu)\,d\nu}
\end{align*}
if $\mu>0$ and
\begin{equation*}
\theta_{k_0,\hspace{1pt}\Im\hspace{-1pt}z}(0)
=\bigl[2^{\hspace{1pt}i\Im\hspace{-1pt}z}
-2^{\hspace{1pt}i\hspace{1pt}(\Im\hspace{-1pt}z)\hspace{1pt}k_0}\bigr]
\,\psi(0),
\end{equation*}
hence
\begin{equation*}
\|\theta_{k_0,\hspace{1pt}\Im\hspace{-1pt}z}\|_{L^{\hspace{.5pt}\infty}}\le2\,\|\psi\|_{L^{\hspace{.5pt}\infty}}+\|\psi'\|_{L^1}.
\end{equation*}

\noindent$\bullet$
If $\Re z=\frac12$, the $L^1\to L^{\hspace{.5pt}\infty}$ norm of $R_{\lambda,z}$ is controlled by the $L^{\hspace{.5pt}\infty}$ norm of
\begin{equation*}
\sigma_{\lambda,\hspace{1pt}\Im\hspace{-1pt}z}
=\sum\nolimits_{k_0<k\le0}
2^{\hspace{1pt}(\frac32+i\Im\hspace{-1pt}z)\hspace{1pt}k}\,q_{\lambda,2^{-k}}
\end{equation*}
which is $\text{O}\bigl(\sqrt{\lambda}\,\bigr)$,
according to Lemma \ref{busard}.

\medskip

In conclusion,
\begin{equation*}
\|\Sigma_0\|_{L^{p^{\hspace{.5pt}\prime}}\to L^p}=2\,\|R_{\lambda,0}\|_{L^{p^{\hspace{.5pt}\prime}}\to L^p}\lesssim\lambda^{\frac13}.
\end{equation*}

\medskip

\noindent (3) follows from (2).
By using Minkowsky's integral inequality, we deduce indeed from
\begin{equation*}
P_{\lambda,\eta}=\int_{\lambda-\eta}^{\lambda+\eta}\mathbf{P}_{\hspace{-1pt}\mu}\,d\mu
\end{equation*}
that
\begin{equation*}
\bigl\|P_{\lambda,\eta}\bigr\|_{L^{p^{\hspace{.5pt}\prime}}\to L^p}
\le\int_{\lambda-\eta}^{\lambda+\eta}\bigl\|\hspace{1pt}\mathbf{P}_{\hspace{-1pt}\mu}
\bigr\|_{L^{p^{\hspace{.5pt}\prime}}\to L^p}\hspace{1pt}d\mu
\lesssim\eta\times\begin{cases}
\,\lambda^{\frac12-\frac2p}
&\text{if \,}p\ge6\\
\hspace{1pt}\frac1{p-2}\,\lambda^{\frac12(\frac12-\frac1p)}
&\text{if \,}2<p\le6
\end{cases}\end{equation*}
for $\lambda$ large, say \hspace{1pt}$\lambda\hspace{-1pt}>\hspace{-2pt}1$\hspace{.5pt},
and \hspace{1pt}$0\hspace{-1pt}<\hspace{-1pt}\eta\hspace{-1pt}\le\hspace{-2pt}\tfrac\lambda4$\hspace{-.5pt}.
\end{proof}
\begin{lemma}[Pointwise kernel bounds]\label{busard}
The following estimates hold for $\lambda>1$, $0<\eta\le\lambda$, $r\ge0$, $0<\epsilon<1$ and $y\in\R:$
\begin{align}
\,\bigl|p_{\lambda,\eta}(r)\bigr|
&\lesssim\begin{cases}
\,\lambda\,\eta&\text{if \,$r$ is small, say \,$r\le1$,}\\
\,\lambda^{\frac12}\,\eta\,e^{-\frac r2}&\text{if \,$r$ is large, say \,$r>1$,}\\
\end{cases}\label{KernelEstimatePtildelambdaeta}\\
\,\bigl|q_{\lambda,\eta}(r)\bigr|
&\lesssim_\epsilon\,\lambda^{\frac12}\,\eta\,(1\!+\hspace{-1pt}\eta)^{\frac12}\,
e^{-\frac\epsilon{8\eta}}\hspace{1pt}e^{-\frac{1-\epsilon}2\hspace{.5pt}r},
\label{KernelEstimateQlambdaeta}\\
\,\bigl|\sigma_{\lambda,y}(r)\bigr|
&\lesssim\begin{cases}
\,\lambda^{\frac12}
&\text{if \,}r\le2,\\
\,0&\text{if \,}r>2.
\end{cases}\label{KernelEstimateSigmalambday}
\end{align}
\end{lemma}

\begin{remark}
As we mentioned above, the estimates obtained in this lemma are more precise than in \cite{GermainLeger2023} (compare with Lemma 4.1 of that paper). Indeed here we carefully track the dependence on $r$ in the estimates. This will allow us to extend our results to the case of convex cocompact hyperbolic surfaces with $0 \le \delta < \frac12.$
\end{remark}


\begin{proof}
According to \cite[Section 2]{Koornwinder1984},
the kernels are given by
\begin{align}
p_{\lambda,\eta}(r)&=-\,\pi^{-\frac32}\hspace{1pt}\eta\int_r^{\hspace{.5pt}\infty}
\tfrac\partial{\partial s}
\bigl[\cos(\lambda\hspace{.5pt}s)\hspace{1pt}\widehat{\chi}(\eta\hspace{.5pt}s)\bigr]
(\cosh s-\cosh r)^{-\frac12}\,ds,\label{Kernelptildelambdaeta}\\
q_{\lambda,\eta}(r)&=-\,\pi^{-\frac32}\hspace{1pt}\eta\int_r^{\hspace{.5pt}\infty}
\tfrac\partial{\partial s}
\bigl[\cos(\lambda\hspace{.5pt}s)\hspace{1pt}\widehat{\psi}(\eta\hspace{.5pt}s)\bigr]
(\cosh s-\cosh r)^{-\frac12}\,ds,\label{Kernelqlambdaeta}\\
\sigma_{\lambda,y}(r)&=-\,\pi^{-\frac32}\int_r^{\hspace{.5pt}\infty}
\tfrac\partial{\partial s}
\bigl[\cos(\lambda\hspace{.5pt}s)\,\vartheta_{k_0}(s)\bigr]
(\cosh s-\cosh r)^{-\frac12}\,ds,\label{Kernelsigmalambday}
\end{align}
where
$$
\vartheta_{k_0}(s)=\sum\nolimits_{k_0<k\le0}2^{(\frac12+iy)k}\,\widehat{\psi}(2^{-k}s)
$$
is a smooth twist of $s^{\frac12+iy}\hspace{1pt}\mathbf{1}_{[\frac1\lambda,1]}(s)$.
We shall use repeatedly the following behavior, for $0\le r<s$\,:
\begin{equation}\label{AuxiliaryEstimate}\begin{aligned}
\cosh s-\cosh r&=2\sinh\tfrac{s+r}2\sinh\tfrac{s-r}2\\
&\sim\begin{cases}
\,e^{\hspace{1pt}s}&\text{if $s-r$ is large,}\\
\,(s-r)\,e^{\hspace{1pt}s}&\text{if $r$ is large and $s-r$ is small,}\\
\,s^2-r^2&\text{if $r$ and $s$ are both small.}\\
\end{cases}\end{aligned}\end{equation}

\medskip

\noindent\underline{Estimate of \eqref{Kernelptildelambdaeta}},
which vanishes if \hspace{1pt}$r\hspace{-2pt}\ge\hspace{-2pt}\tfrac2\eta$\hspace{1pt}.
If $r$ is large, say \hspace{1pt}$r\hspace{-2pt}>\hspace{-2pt}1$\hspace{1pt},
\eqref{Kernelptildelambdaeta} is estimated by splitting up
$$
\int_r^{\hspace{.5pt}\infty}=\int_r^{\hspace{1pt}r+\lambda^{-1}}+\int_{r+\lambda^{-1}}^{\hspace{1pt}r+1}+\int_{r+1}^{\hspace{.5pt}\infty}.
$$
More precisely, by using
\begin{equation}\label{EstimateDerivative}
\tfrac\partial{\partial s}
\bigl[\cos(\lambda\hspace{.5pt}s)\hspace{1pt}\widehat{\chi}(\eta\hspace{.5pt}s)\bigr]
=\text{O}(\lambda),
\end{equation}
the contribution of the first integral is bounded by
$$
\lambda\hspace{1pt}\eta
\int_r^{\hspace{1pt}r+\lambda^{-1}}(s-r)^{-\frac12}\hspace{1pt}e^{-\frac s2}\,ds
\lesssim\lambda^{\frac12}\hspace{1pt}\eta\,e^{-\frac r2}.
$$
On the other hand, after performing an integration by parts, 
the contributions of the second and third integrals are bounded by
$$
\eta\,(s-r)^{-\frac12}\hspace{1pt}e^{-\frac s2}\,\Big|_{s=r+\lambda^{-1}}
+\eta\,\int_{r+\lambda^{-1}}^{\hspace{1pt}r+1}(s-r)^{-\frac32}\hspace{1pt}e^{-\frac s2}\hspace{1pt}ds
+\eta\,\int_{r+1}^{\hspace{.5pt}\infty} e^{-\frac s2}\hspace{1pt}ds
\lesssim\lambda^{\frac12}\hspace{1pt}\eta\,e^{-\frac r2}\hspace{1pt}.
$$
If $r$ is small, say $r\le1$, \eqref{Kernelptildelambdaeta} is estimated by splitting up
$$
\int_r^{\hspace{.5pt}\infty}=\int_r^{\sqrt{r^2+\lambda^{-2}}}+\int_{\sqrt{r^2+\lambda^{-2}}}^{\hspace{1pt}r+1}+\int_{r+1}^{\hspace{.5pt}\infty}.
$$
More precisely, by using this time
$$
\tfrac\partial{\partial s}\bigl[\cos(\lambda s)\hspace{1pt}\widehat{\chi}(\eta\hspace{1pt}s)\bigr]
=\text{O}(\lambda^2s),
$$
the contribution of the first integral is bounded by
$$
\lambda^2\eta\int_r^{\sqrt{r^2+\lambda^{-2}}}(s^2\!-r^2)^{-\frac12}s\,ds
=\lambda^2\eta\sqrt{s^2-r^2}\,\Big|_{s=r}^{s=\sqrt{r^2+\lambda^{-2}}}
=\lambda\hspace{1pt}\eta.
$$
On the other hand, after an integration by parts,
the contributions of the second and third integrals are bounded by
$$
-\,\eta\,(s^2\!-r^2)^{-\frac12}\hspace{1pt}\Big|_{s=\sqrt{r^2+\lambda^{-2}}}
+\eta\,\int_{\sqrt{r^2+\lambda^{-2}}}^{\hspace{1pt}r+1}(s^2\!-r^2)^{-\frac32}s\,ds
+\eta\,\int_{r+1}^{\hspace{.5pt}\infty} e^{-\frac s2}ds
\lesssim\lambda\hspace{1pt}\eta.
$$
This concludes the proof of \eqref{KernelEstimatePtildelambdaeta}.

\medskip

\noindent\underline{Estimate of \eqref{Kernelqlambdaeta},} which vanishes again if \hspace{1pt}$r\hspace{-2pt}\ge\hspace{-2pt}\tfrac2\eta$\hspace{1pt}.
Assume first that \hspace{1pt}$1\hspace{-2pt}\le\hspace{-1pt}\eta\hspace{-1pt}\le\hspace{-1pt}\lambda$\hspace{2pt}.
If \hspace{1pt}$r\hspace{-1pt}\le\hspace{-1pt}\tfrac1{4\eta}$\hspace{1pt},
\eqref{Kernelqlambdaeta} becomes, after an integration by parts,
\begin{equation}\label{Kernelqlambdaetabis}
q_{\lambda,\eta}(r)=\pi^{-\frac32}\eta\int_{\frac1{2\eta}}^{\frac2\eta}
\cos(\lambda\hspace{.5pt}s)\,\widehat{\psi}(\eta\hspace{.5pt}s)\,
\tfrac\partial{\partial s}(\cosh s-\cosh r)^{-\frac12}\,ds,
\end{equation}
with
\begin{equation*}
-\tfrac\partial{\partial s}(\cosh s-\cosh r)^{-\frac12}
=\tfrac12(\cosh s-\cosh r)^{-\frac32}\sinh s
\sim s\,(s^2\!-r^2)^{-\frac32}
\end{equation*}
under the present assumptions.
Hence
\begin{equation*}
|q_{\lambda,\eta}(r)|
\lesssim\eta\int_{\frac1{2\eta}}^{\frac2\eta}(s^2\!-r^2)^{-\frac32}\,s\,ds
\lesssim\eta^2
\le\lambda^{\frac12}\,\eta^{\frac32}.
\end{equation*}
If \hspace{1pt}$\tfrac1{4\eta}\hspace{-2pt}<\hspace{-1pt}r\hspace{-1pt}<\hspace{-2pt}\tfrac2\eta$\hspace{1pt},
we split up
\begin{equation*}
\int_r^{\frac3\eta}
=\,\int_r^{\hspace{1pt}r+\lambda^{-1}}\hspace{-2pt}
+\,\int_{r+\lambda^{-1}}^{\frac3\eta}
\end{equation*}
in \eqref{Kernelqlambdaeta} and use \eqref{EstimateDerivative}, together with the behavior
\begin{equation*}
\cosh s-\cosh r\sim s^2\!-r^2\sim\eta^{-1}(s-r)
\end{equation*}
under the present assumptions.
On the one hand, the contribution of the first integral is bounded by
\begin{equation*}
\lambda\,\eta^{\frac32}\int_r^{\hspace{1pt}r+\lambda^{-1}}\hspace{-1pt}(s-r)^{-\frac12}\hspace{1pt}ds
\lesssim\lambda^{\frac12}\,\eta^{\frac32}.
\end{equation*}
On the other hand, after an integration by parts, the second integral yields the sum
\begin{gather*}
-\,\pi^{\frac32}\eta\,\cos(\lambda s)\hspace{1pt}\widehat{\psi}(\eta\hspace{1pt}s)\hspace{1pt}
(\cosh s-\cosh r)^{-\frac12}\hspace{1pt}\Big|_{s=r+\lambda^{-1}}\\
-\,\tfrac12\hspace{1pt}\pi^{\frac32}\eta\int_{r+\lambda^{-1}}^{\frac3\eta}
\cos(\lambda s)\hspace{1pt}\widehat{\psi}(\eta\hspace{1pt}s)\hspace{1pt}(\cosh s-\sinh r)^{-\frac 32}\sinh s\,ds,
\end{gather*}
which is $\text{O}\bigl(\lambda^{\frac12}\hspace{1pt}\eta^{\frac32}\bigr)$ too.
This concludes the proof of \eqref{KernelEstimateQlambdaeta}
when \hspace{1pt}$1\hspace{-2pt}\le\hspace{-1pt}\eta\hspace{-1pt}\le\hspace{-1pt}\lambda$\hspace{1pt}.
Assume next that \hspace{1pt}$\eta\hspace{-1pt}<\hspace{-2pt}1$\hspace{1pt}.
If \hspace{1pt}$r\hspace{-1pt}\le\hspace{-1pt}\tfrac1{4\eta}$\hspace{1pt},
we use again the expression \eqref{Kernelqlambdaetabis} with, this time,
\begin{equation*}
-\hspace{1pt}\tfrac\partial{\partial s}\hspace{1pt}(\cosh s\hspace{-1pt}-\hspace{-1pt}\cosh r)^{-\frac12}
=\tfrac12\hspace{1pt}(\cosh s\hspace{-1pt}-\hspace{-1pt}\cosh r)^{-\frac32}\sinh s
\sim e^{-\frac s2}\,.
\end{equation*}
Hence
\begin{equation*}
|q_{\lambda,\eta}(r)|
\lesssim\eta\int_{\frac1{2\eta}}^{\frac2\eta}e^{-\frac s2}\,ds
\lesssim\eta\,e^{-\frac1{4\eta}}
\lesssim\eta\,e^{-\frac1{8\eta}}\hspace{1pt}e^{-\frac r2}\,.
\end{equation*}
If $\tfrac1{4\eta}\hspace{-1pt}<\hspace{-1pt}r\hspace{-1pt}<\hspace{-1pt}\tfrac2{\eta}$\hspace{1pt},
we split up
\begin{equation*}
\int_r^{\frac3\eta}=\,\int_r^{\hspace{1pt}r+\lambda^{-1}}\hspace{-2pt}
+\,\int_{r+\lambda^{-1}}^{\hspace{1pt}r+1}+\,\int_{r+1}^{\frac3\eta}.
\end{equation*}
in \eqref{Kernelqlambdaeta}.
On the one hand, by using \eqref{EstimateDerivative} and
\begin{equation}\label{AuxiliaryEstimateBis}
\cosh r\hspace{-1pt}-\hspace{-1pt}\cosh s\sim(s\hspace{-1pt}-\hspace{-1pt}r)\,e^{\hspace{1pt}r},
\end{equation}
the contribution of the first integral is bounded by
\begin{equation*}
\lambda\,\eta\,e^{-\frac r2}\int_r^{\hspace{1pt}r+\lambda^{-1}}\!(s-r)^{-\frac12}\,ds
\lesssim\sqrt{\lambda}\,\eta\,e^{-\frac r2}
\lesssim\sqrt{\lambda}\,\eta\,e^{-\frac\epsilon{8\eta}}\,e^{-\frac{1-\epsilon}2\hspace{.5pt}r}\,.
\end{equation*}
On the other hand, after an integration by parts, the second and third integrals yield the sum
\begin{equation}\label{Sum}\begin{aligned}
&-\,\pi^{\frac32}\eta\,\cos(\lambda s)\hspace{1pt}\widehat{\psi}(\eta\hspace{1pt}s)\hspace{1pt}
(\cosh s-\cosh r)^{-\frac12}\hspace{1pt}\Big|_{s=r+\lambda^{-1}}\\
&-\,\tfrac12\hspace{1pt}\pi^{\frac32}\eta\int_{r+\lambda^{-1}}^{\hspace{1pt}r+1}
\cos(\lambda s)\hspace{1pt}\widehat{\psi}(\eta\hspace{1pt}s)\hspace{1pt}(\cosh s-\sinh r)^{-\frac 32}\sinh s\,ds\\
&-\,\tfrac12\hspace{1pt}\pi^{\frac32}\eta\int_{r+1}^{\frac3\eta}
\cos(\lambda s)\hspace{1pt}\widehat{\psi}(\eta\hspace{1pt}s)\hspace{1pt}(\cosh s-\sinh r)^{-\frac 32}\sinh s\,ds\,.
\end{aligned}\end{equation}
By using \eqref{AuxiliaryEstimateBis} when
\hspace{1pt}$r\hspace{-1pt}\le\hspace{-1pt}s\hspace{-1pt}\le\hspace{-1pt}r\hspace{-1pt}+\!1$
\hspace{.5pt}and
$$
\cosh r\hspace{-1pt}-\hspace{-1pt}\cosh s\sim e^{\hspace{1pt}s}
$$
when \hspace{1pt}$r\hspace{-1pt}+\!1\!\le\hspace{-1pt}s\hspace{-1pt}
\le\hspace{-1pt}\tfrac3\eta$\hspace{.5pt},
we estimate the first line of \eqref{Sum} by $\lambda^{\frac12}\hspace{1pt}\eta\,e^{-\frac r2}$,
the second line by
$$
\eta\,e^{-\frac r2}\int_r^{\hspace{1pt}r+\lambda^{-1}}\!(s-r)^{-\frac32}\hspace{1pt}ds
\lesssim\lambda^{\frac12}\hspace{1pt}\eta\,e^{-\frac r2}
$$
and the third line by \hspace{1pt}$\eta\,e^{-\frac r2}$.
Overall, we obtain again the bound
\begin{equation*}
\lambda^{\frac12}\hspace{1pt}\eta\,e^{-\frac r2}\lesssim
\lambda^{\frac12}\hspace{1pt}\eta\,e^{-\frac\epsilon{8\eta}}\,e^{-\frac{1-\epsilon}2\hspace{.5pt}r}\,.
\end{equation*}
This concludes the proof of \eqref{KernelEstimateQlambdaeta}.

\medskip \noindent \underline{Proof of \eqref{KernelEstimateSigmalambday}.} Notice first of all that
\begin{equation*}
\supp\vartheta_{k_0}\subset[-2,-\tfrac1{2\lambda}]\cup[\tfrac1{2\lambda},2]
\quad\text{with}\quad|\vartheta_{k_0}(s)|\lesssim|s|^{\frac12},
\quad|(\vartheta_{k_0})'(s)|\lesssim|s|^{-\frac12},
\end{equation*}
as $\supp\widehat{\psi}(2^{-k}\,\cdot\,)\subset[2^{k-1},2^{k+1}]$.
If $r\le\tfrac1{4\lambda}$, we obtain
\begin{equation*}
\sigma_{\lambda,y}(r)=-\,\pi^{-\frac32}\int_{\frac1{2\lambda}}^2\cos(\lambda s)\,\vartheta_{k_0}(s)
(\cosh s-\cosh r)^{-\frac32}\hspace{1pt}\sinh s\,ds
\end{equation*}
after an integration by parts and estimate straightforwardly
\begin{equation*}
|\sigma_{\lambda,y}(r)|
\lesssim\int_{\frac1{2\lambda}}^2(s^2\!-r^2)^{-\frac32}\hspace{1pt}s^{\frac32}\hspace{1pt}ds
=r^{-\frac12}\int_{\frac1{2\lambda r}}^{\frac 2r}(s^2\!-1)^{-\frac32}\hspace{1pt}s^{\frac32}\hspace{1pt}ds
\lesssim r^{-\frac12}\int_{\frac1{2\lambda r}}^{\hspace{.5pt}\infty} s^{-\frac32}\hspace{1pt}ds
\lesssim\lambda^{\frac12}\hspace{1pt}.
\end{equation*}
If $\tfrac1{2\lambda}\le r\le2$, we split up
\begin{equation*}
\int_r^{\hspace{1pt}4}=\,\int_r^{\hspace{1pt}r+\lambda^{-1}}\hspace{-2pt}+\,\int_{r+\lambda^{-1}}^{\hspace{1pt}4}
\end{equation*}
in \eqref{Kernelsigmalambday}.
On the one hand, the contribution of the first integral is bounded by the sum of
\begin{equation*}
I=\int_r^{\hspace{1pt}r+\lambda^{-1}}\!\lambda\,s^{\frac12}\hspace{1pt}(s^2\!-r^2)^{-\frac12}\hspace{1pt}ds
=\lambda\,r^{\frac12}\!\int_1^{\hspace{1pt}1+\frac1{\lambda r}}(s^2\!-1)^{-\frac12}\hspace{1pt}s^{\frac12}\hspace{1pt}ds
\end{equation*}
and
\begin{equation*}
I\!I=\int_r^{\hspace{1pt}r+\lambda^{-1}}\!s^{-\frac12}\hspace{1pt}(s^2\!-r^2)^{-\frac12}\hspace{1pt}ds
=r^{-\frac12}\!\int_1^{\hspace{1pt}1+\frac1{\lambda r}}(s^2\!-1)^{-\frac12}\hspace{1pt}s^{-\frac12}\hspace{1pt}ds.
\end{equation*}
As
\begin{equation*}
\int_1^{\hspace{1pt}1+\frac1{\lambda r}}(s^2\!-1)^{-\frac12}\hspace{1pt}s^{\pm\frac12}\hspace{1pt}ds
\sim\int_1^{\hspace{1pt}1+\frac1{\lambda r}}(s-1)^{-\frac12}\hspace{1pt}ds
\sim\tfrac1{\sqrt{\lambda\,r}},
\end{equation*}
we conclude that
\begin{equation*}
I\lesssim\lambda^{\frac12}
\quad\text{and}\quad
I\!I\lesssim\tfrac1{\sqrt{\lambda}\,r}\lesssim\lambda^{\frac12}
\end{equation*}
under the present assumptions.
On the other hand, after an integration by parts,
the second integral yields the sum of
\begin{equation*}
-\,\pi^{\frac32}\hspace{1pt}\cos(\lambda s)\,\vartheta_{k_0}(s)\hspace{1pt}
(\cosh s-\cosh r)^{-\frac12}\hspace{1pt}\Big|_{s=r+\lambda^{-1}},
\end{equation*}
which is $\text{O}\bigl(\sqrt{\lambda}\,\bigr)$,
and of
\begin{equation*}
-\tfrac12\,\pi^{\frac32}\int_{r+\lambda^{-1}}^{\hspace{1pt}4}
\cos(\lambda s)\,\vartheta_{k_0}(s)\hspace{1pt}
(\cosh s-\cosh r)^{-\frac32}\hspace{1pt}\sinh s\,ds,
\end{equation*}
which is bounded by 
\begin{align*}
\int_{r+\lambda^{-1}}^{\hspace{1pt}4}(s^2\!-r^2)^{-\frac32}\,s^{\frac32}\,ds\,
=\,\underbrace{r^{-\frac12}\!\int_{1+\frac1{\lambda r}}^{\hspace{1pt}2}(s^2\!-1)^{-\frac32}\,s^{\frac32}\,ds}_{I\!I\!I}
\,+\,\underbrace{r^{-\frac12}\!\int_2^{\frac4r}(s^2\!-1)^{-\frac32}\,s^{\frac32}\,ds}_{I\hspace{-1pt}V}.
\end{align*}
We conclude by observing that
\begin{equation*}
|I\!I\!I|\lesssim r^{-\frac12}\!\int_{1+\frac1{\lambda r}}^{\hspace{1pt}2}(s-1)^{-\frac32}\,ds\sim\lambda^{\frac12}
\end{equation*}
and
\begin{equation*}
|I\hspace{-1pt}V|
\lesssim r^{-\frac12}\!\int_2^{\hspace{.5pt}\infty}s^{-\frac32}\,ds
\lesssim\lambda^{\frac12}
\end{equation*}
under the present assumptions.
\end{proof}

Let us conclude this subsection with a low frequency bound for the operators
$$
P_{\lambda,\eta}=\mathbf{1}_{[\lambda-\eta,\lambda+\eta]}(D)
\quad{and}\quad
\mathscr{P}_{\lambda,\eta}=\chi(\tfrac{D-\lambda}\eta)+\chi(\tfrac{D+\lambda}\eta)\hspace{1pt},
$$
where $\chi$ is an even Schwartz function.

\begin{proposition}[low frequency]\label{LambdaSmallH}
Assume that \,$\lambda$ and \,$\eta$ are both small,
say \,$0\hspace{-1pt}\le\hspace{-1pt}\lambda\hspace{-1pt}\le\!1$
and \,$0\hspace{-1pt}<\hspace{-1pt}\eta\hspace{-1pt}<\!1$.
Then
\begin{equation*}
\bigl\|P_{\lambda,\eta}\bigr\|_{L^2\to L^p}
\lesssim(\lambda\hspace{-1pt}+\hspace{-1pt}\eta)\,\eta^{\frac12}
\quad\text{and}\quad
\bigl\|\mathscr{P}_{\lambda,\eta}\bigr\|_{L^2\!\to L^p}
\lesssim(\lambda\hspace{-1pt}+\hspace{-1pt}\eta)\,\eta^{\frac12}
\end{equation*}
for every $2<p\le\infty$.
\end{proposition}

\begin{remark}
These bounds are optimal, according to Proposition \ref{SphericalExample}.
\end{remark}

\begin{proof}
Let us first estimate the kernel
\begin{equation*}
p_{\lambda,\eta}(r)=-\,\pi^{-\frac32}\hspace{1pt}\eta\int_r^{\hspace{.5pt}\infty}\!
\tfrac\partial{\partial s}\bigl[\cos(\lambda s)\hspace{1pt}\widehat{\chi}(\eta\hspace{.5pt}s)\bigr]
\hspace{1pt}(\cosh s-\cosh r)^{-\frac12}\,ds
\end{equation*}
of $\mathscr{P}_{\lambda,\eta}$ by splitting up
\begin{equation*}
\int_r^{\hspace{.5pt}\infty}=\int_r^{\hspace{1pt}r+1}+\int_{r+1}^{\hspace{.5pt}\infty}
\end{equation*}
and by using \eqref{AuxiliaryEstimate} together with
$$
\tfrac\partial{\partial s}\bigl[\cos(\lambda s)\hspace{1pt}\widehat{\chi}(\eta\hspace{.5pt}s)\bigr]
=-\,\lambda\hspace{1pt}\underbrace{\sin(\lambda\hspace{.5pt}s)}_{\text{O}(\lambda s)}
\widehat{\chi}(\eta\hspace{.5pt}s)
+\cos(\lambda s)\,\eta\,
\underbrace{\widehat{\chi}^{\hspace{1pt}\prime}(\eta\hspace{.5pt}s)}_{\text{O}(\eta\hspace{.5pt}s)}
=\text{O}\bigl((\lambda^2\hspace{-2pt}+\hspace{-1pt}\eta^2)\hspace{.5pt}s\bigr).
$$
This way we obtain
\begin{equation*}
\bigl|p_{\lambda,\eta}(r)\bigr|
\lesssim\eta\,(\lambda^2\hspace{-2pt}+\hspace{-1pt}\eta^2)\hspace{-1pt}\underbrace{
\int_r^{\hspace{1pt}r+1}\!(s^2\hspace{-2pt}-\hspace{-1pt}r^2)^{-\frac12}\hspace{1pt}s\,ds
}_{\lesssim\,1}
+\,\eta\,(\lambda^2\hspace{-2pt}+\hspace{-1pt}\eta^2)\hspace{-1pt}\underbrace{
\int_{r+1}^{\hspace{.5pt}\infty}\!e^{-\frac s2}\hspace{1pt}s\,ds
}_{\lesssim\,1}
\lesssim\eta\,(\lambda^2\hspace{-2pt}+\hspace{-1pt}\eta^2)
\end{equation*}
when $r$ is small, and
\begin{equation*}
\bigl|p_{\lambda,\eta}(r)\bigr|
\lesssim\eta\,(\lambda^2\hspace{-2pt}+\hspace{-1pt}\eta^2)\hspace{-1pt}\underbrace{
\int_r^{\hspace{1pt}r+1}\!e^{-\frac s2}\hspace{1pt}
(s\hspace{-1pt}-\hspace{-1pt}r)^{-\frac12}\hspace{1pt}s\,ds
}_{\lesssim\;r\,e^{-\frac r2}}
+\,\eta\,(\lambda^2\hspace{-2pt}+\hspace{-1pt}\eta^2)\hspace{-1pt}\underbrace{
\int_{r+1}^{\hspace{.5pt}\infty}\!e^{-\frac s2}\hspace{1pt}s\,ds
}_{\lesssim\;r\,e^{-\frac r2}}
\lesssim\eta\,(\lambda^2\hspace{-2pt}+\hspace{-1pt}\eta^2)\,r\,e^{-\frac r2}
\end{equation*}
when $r$ is large. Altogether
\begin{equation*}
|p_{\lambda,\eta}(r)|\lesssim(\lambda^2\!+\hspace{-1pt}\eta^2)\,\eta\,(1\!+\hspace{-1pt}r)\,e^{-\frac r2}
\qquad\forall\,r\hspace{-1pt}\ge\hspace{-1pt}0\,.
\end{equation*}

From this kernel estimate, we deduce on the one hand
\begin{equation*}
\bigl\|\mathscr{P}_{\lambda,\eta}\bigr\|_{L^1\to L^\infty}\hspace{-1pt}
\lesssim(\lambda^2\!+\hspace{-1pt}\eta^2)\,\eta
\end{equation*}
and, on the other hand,
\begin{equation*}
\bigl\|\mathscr{P}_{\lambda,\eta}\bigr\|_{L^{p^{\hspace{.5pt}\prime}}\!\to L^p}\hspace{-1pt}
\lesssim(\lambda^2\!+\hspace{-1pt}\eta^2)\,\eta
\qquad\forall\;2\hspace{-1pt}<\hspace{-1pt}p\hspace{-1pt}<\hspace{-1pt}\infty\hspace{1pt},
\end{equation*}
by using the Kunze-Stein phenomenon on $\H$, as stated in Lemma \ref{KunzeSteinH} below.
By the $TT^*$ argument, we conclude that
\begin{equation*}
\bigl\|\mathscr{P}_{\lambda,\eta}\bigr\|_{L^2\to L^p}\hspace{-1pt}
\lesssim(\lambda^2\!+\hspace{-1pt}\eta^2)\,\eta
\qquad\forall\;2\hspace{-1pt}<\hspace{-1pt}p\hspace{-1pt}\le\hspace{-1pt}\infty\hspace{1pt}.
\end{equation*}

Finally, the bounds for \hspace{1pt}$P_{\lambda,\eta}$ can be
either deduced from the results for \hspace{1pt}$\mathscr{P}_{\lambda,\eta}$\hspace{1pt},
as explained in the first step of Appendix \ref{AppendixSogge},
or proved directly as above.
\end{proof}

\begin{lemma}[see Lemma 5.1 in \cite{AnkerPierfeliceVallarino2012}]\label{KunzeSteinH}
We have
\begin{equation*}
\|\hspace{.5pt}f\hspace{-1pt}*\hspace{-.5pt}K\hspace{1pt}\|_{L^p}
\lesssim\|f\|_{L^{p^{\hspace{.5pt}\prime}}}\hspace{1pt}
\Bigl[\hspace{1pt}\int_0^{\hspace{.5pt}\infty}\!|K(r)|^{\frac p2}\hspace{1pt}
(1\!+\hspace{-1pt}r)\,e^{-\frac r2}\hspace{1pt}\sinh r\,dr\hspace{1pt}\Bigr]^{\frac2p}
\end{equation*}
for every \,$2\hspace{-1pt}\le\hspace{-1pt}p\hspace{-1pt}<\hspace{-1pt}\infty$\hspace{1pt},
for every \,$f\hspace{-1pt}\in\hspace{-1pt}L^{p^{\hspace{.5pt}\prime}}\hspace{-1pt}(\H)$
and for every radial measurable function \,$K$ on~\,$\H$.
In the limit case \,$p\hspace{-.5pt}=\hspace{-1pt}\infty$\hspace{.5pt},
this inequality boils down to
\,$\|\hspace{.5pt}f\hspace{-1pt}*\hspace{-.5pt}K\hspace{1pt}\|_{L^\infty}\!
\le\hspace{-.5pt}\|f\|_{L^1}\hspace{1pt}\|K\|_{L^\infty}$.
\end{lemma}

\begin{remark}
From Proposition \ref{LambdaSmallH} we can recover
the bound (1.19) at low frequency in \cite[Theorem~1.6]{ChenHassell2018}.
Indeed,
\begin{equation*}
\|\hspace{1pt}\mathbf{P}_{\hspace{-1pt}\lambda}\|_{L^{p^{\hspace{.5pt}\prime}}\!\to L^p}\leq
\liminf_{\eta\to0}\frac{\Vert P_{\lambda,\eta}\Vert_{L^{p^{\hspace{.5pt}\prime}}\!\to L^p}}{2 \eta}
\lesssim\lambda^2.
\end{equation*}
\end{remark}

\begin{remark}
The results in this section extend straightforwardly to
real hyperbolic spaces of dimension $d\ge2$ considered in \cite{GermainLeger2023}
and more generally to all hyperbolic spaces (as well as Damek--Ricci spaces).
In this case,
the high frequency bound in Theorem \ref{L2LpBoundsProjectorsH} becomes
\begin{equation}\label{Generalization}
\|P_{\lambda,\eta}\|_{L^2\to L^p}
\lesssim\bigl(1\!+\!\tfrac1{\sqrt{p-2}}\bigr)\hspace{1pt}\lambda^{\gamma(p)}\,\eta^{\frac12}
\quad\text{with}\quad
\gamma(p)=\max\hspace{1pt}\bigl\{\tfrac{d-1}2\hspace{-1pt}-\hspace{-1pt}\tfrac dp\hspace{1pt},
\tfrac{d-1}2\bigl(\tfrac12\hspace{-1pt}-\hspace{-1pt}\tfrac1p\bigr)\hspace{-1pt}\bigr\}
\end{equation}
while the low frequency bound in Proposition \ref{LambdaSmallH} remains the same.
\end{remark}

\section{Refined dispersive and Strichartz estimates on $\mathbb{H}$}\label{DispersiveStrichartzH}


In this section, we prove successively kernel, dispersive and Strichartz estimates for the operators
$e^{\hspace{.5pt}i\hspace{.75pt}t\hspace{.5pt}\Delta}\hspace{1pt}\mathscr{P}_{\lambda\hspace{.5pt},1}$
on the hyperbolic plane $\H$, where $t\hspace{-1pt}\in\hspace{-1pt}\R^*$, $\lambda\hspace{-1pt}>\!1$
and $\chi\hspace{-1pt}\in\hspace{-1pt}\mathcal{S}(\R)$ is an even bump function.
\smallskip
Again in the estimates we keep track of the dependence on $r$ so that our results generalize easily to the case of convex cocompact hyperbolic surfaces with $0 \le \delta < \frac12.$
\smallskip

By symmetry we may assume that $t\hspace{-1pt}>\hspace{-1pt}0$\hspace{1pt}.
Recall that the kernel of
$e^{\hspace{.5pt}i\hspace{.75pt}t\hspace{.5pt}\Delta}\hspace{1pt}\mathscr{P}_{\lambda\hspace{.5pt},1}$
is given by
\begin{equation*}
K(r)=\const\int_r^{\hspace{.5pt}\infty}\tfrac\partial{\partial s}\,\widehat{m}(s)\,(\cosh s-\cosh r)^{-\frac12}\hspace{1pt}ds,
\end{equation*}
where
\begin{equation*}
\widehat{m}(s)
= \tfrac1{\sqrt{2\pi}} \int_{-\infty}^{\hspace{.5pt}\infty}
e^{is\xi}\,e^{-it\xi^2}\hspace{1pt}
\bigl[\chi(\xi-\lambda)+\chi(\xi+\lambda)\bigr]\hspace{1pt}d\xi
\end{equation*}
is the Fourier transform of the symbol
\begin{equation*}
m(\xi)=e^{-it\xi^2}\hspace{1pt}\bigl[\chi(\xi-\lambda)+\chi(\xi+\lambda)\bigr].
\end{equation*}
Let us split up dyadically $\chi=\sum_{j=0}^{\hspace{.5pt}\infty}\chi_j$,
and $m=\sum_{j=0}^{\hspace{.5pt}\infty} m_j$, $K=\sum_{j=0}^{\hspace{.5pt}\infty}K_j$ accordingly.
More precisely, given a smooth even bump function $\theta:\R\to[0,1]$
such that $\theta=1$ on $[-1,1]$ and $\supp\theta\subset[-2,2]$,
we set $\chi_0(\xi)=\chi(\xi)\,\theta(\xi)$ and
$\chi_j(\xi)=\chi(\xi)\bigl[\theta(2^{-j}\xi)-\theta(2^{1-j}\xi)\bigr]$
$\forall\,j\!\in\mathbb{N}^*$.

\subsection{Kernel estimates}
As we in the previous section \ref{upperbplane}, the key are pointwise estimates on the kernels of the operators just introduced. They are stated and proved in the remainder of this subsection.
\begin{lemma}\label{LemmaH2}
Let $M,N\!\in\mathbb{N}$.
Then the following kernel estimates hold,
for $r\ge0$, $t>0$, $\lambda>1$ and $j\in\mathbb{N}:$
\begin{equation}\label{KernelEstimateH2}
|K_j(r)|\lesssim2^{-Mj}\hspace{1pt}(1\!+\hspace{-1pt}t)^{-\frac32}\,e^{-\frac r2}
\times\begin{cases}
\,r
&\text{if \,$r$ and \,$t$ are large,}\\
\,\lambda
&\text{otherwise.}
\end{cases}
\end{equation}
Moreover, the following estimate holds,
under the assumptions that \,$2^{-j}\lambda>16$, that \,$t\lambda$ is large,
and that \,$r\notin[t\lambda,3t\lambda]$\,:
\begin{equation}\label{ImprovedKernelEstimateH2}
|K_j(r)|\lesssim2^{-Mj}\hspace{1pt}(t\lambda)^{-N}\hspace{1pt}\lambda\,e^{-\frac r2}\hspace{1pt}.
\end{equation}
\end{lemma}

\begin{proof}
(a) To begin with, let us estimate for $s\ge0$ the oscillatory integral
\begin{equation}\label{mjhat}
\widehat{m_j}(s)= \tfrac1{\sqrt{2\pi}}\int_{-\infty}^{\hspace{.5pt}\infty}
e^{is\xi}\,e^{-it\xi^2}\hspace{1pt}
\bigl[\chi_j(\xi-\lambda)+\chi_j(\xi+\lambda)\bigr]\hspace{1pt}d\xi
\end{equation}
and its derivative
\begin{equation}\label{mjhatprime}
\tfrac\partial{\partial s}\hspace{1pt}\widehat{m_j}(s)
= \tfrac i{\sqrt{2\pi}}\int_{-\infty}^{\hspace{.5pt}\infty}e^{is\xi}\,\xi\,
e^{-it\xi^2}\hspace{1pt}\bigl[\chi_j(\xi-\lambda)+\chi_j(\xi+\lambda)\bigr]
\hspace{1pt}d\xi\hspace{1pt},
\end{equation}
which becomes
\vspace{-5mm}
\begin{equation}\label{mjhatprimebis}
\tfrac\partial{\partial s}\hspace{1pt}\widehat{m_j}(s)
=\tfrac i2\hspace{.5pt}\tfrac st\,\widehat{m_j}(s)
+\tfrac1{2t}\hspace{1pt}\overbrace{\tfrac1{\sqrt{2\pi}}
\int_{-\infty}^{\hspace{.5pt}\infty}e^{is\xi}\,e^{-it\xi^2}\hspace{1pt}
\bigl[\chi_j^{\hspace{.5pt}\prime}(\xi-\lambda)+\chi_j^{\hspace{.5pt}\prime}(\xi+\lambda)\bigr]
\,d\xi}^{\widehat{\undertilde{m}_j}(s)}\hspace{1pt},
\end{equation}
after an integration by parts based on
\begin{equation*}
\xi\,e^{-it\xi^2}=\tfrac i{2t}\,\tfrac\partial{\partial\xi}\,e^{-it\xi^2}.
\end{equation*}
\noindent
Observe on the one hand that
the phase $\Phi(\xi)=s\xi-t\xi^2$ in \eqref{mjhat}, \eqref{mjhatprime} and \eqref{mjhatprimebis}
has a single stationary point $\xi=\frac s{2t}$ and that $\Phi''=-2t$.
Observe on the other hand that the amplitude $a_j(\xi)=\chi_j(\xi-\lambda)+\chi_j(\xi+\lambda)$
satisfies\footnote{\,To this end, we write $\xi=(\xi\mp\lambda)\pm\lambda$
and use the fact that $\chi$ is a Schwartz function.}
\begin{equation*}
\bigl|\bigl(\tfrac\partial{\partial\xi}\bigr)^k\bigl[\xi^\ell a_j(\xi)\bigr]\bigr|\lesssim2^{-Mj}\lambda^\ell
\quad\forall\,k,\ell,M\!\in\hspace{-1pt}\mathbb{N}\hspace{1pt}.
\end{equation*}
By using the van der Corput lemma\footnote{
$|\int_Ie^{i\Phi(\xi)}a(\xi)\,d\xi\,|\le C\,T^{-\frac12}(\|a\|_{L^{\hspace{.5pt}\infty}(I)}+\|a'\|_{L^1(I)})$,
where $|\Phi''|\ge T$ on $I$.}
(see for instance \cite[Ch.~VIII, \S~1.2]{Stein1993}) for $t$ large
and a trivial estimate for $t$ small, we obtain
\begin{equation}\label{Estimate1mhat}
\bigl|\bigl(\tfrac\partial{\partial s}\bigr)^\ell\hspace{1pt}\widehat{m_j}(s)\bigr|
\lesssim2^{-Mj}\tfrac{\lambda^\ell}{\sqrt{1+t}}\hspace{1pt}.
\end{equation}
Let us improve this result for $\ell=1$.
Firstly,
as \eqref{Estimate1mhat} holds also for $\widehat{\undertilde{m}_j}$,
we deduce from \eqref{mjhatprimebis}
the following improved bound for $t$ large\,:
\begin{equation}\label{Estimate1bismhat}
\bigl|\tfrac\partial{\partial s}\,\widehat{m_j}(s)\bigr|
\lesssim2^{-Mj}\tfrac{1+s}{t\sqrt{1+t}}\hspace{1pt}.
\end{equation}
\noindent
Secondly,
as $\tfrac\partial{\partial s}\,\widehat{m_j}(s)$ is an odd function,
we deduce from
\begin{equation*}
\tfrac\partial{\partial s}\,\widehat{m_j}(s)
=\int_0^s\bigl(\tfrac\partial{\partial u}\bigr)^2\hspace{1pt}\widehat{m_j}(u)\,du
\end{equation*}
and \eqref{Estimate1mhat} the following improved bound for $s$ small\,:
\begin{equation}\label{Estimate1termhat}
\bigl|\tfrac\partial{\partial s}\,\widehat{m_j}(s)\bigr|
\lesssim2^{-Mj}\tfrac{\lambda^2s}{\sqrt{1+t}}\hspace{1pt}.
\end{equation}

Thirdly we improve upon \eqref{Estimate1mhat} for $t$ large and $s$ small.
First note that using \eqref{Estimate1mhat}, we have
\begin{align}\label{Estimatesottildemj}
\bigg \vert  \frac{s}{t} \widehat{m_j}(s)  \bigg \vert \lesssim 2^{-Mj}  \frac{ s }{t \sqrt{1+t}}.
\end{align}
Next notice that since $\widehat{\undertilde{m}_j}$ is odd we have
$\widehat{\undertilde{m}_j}(s)=\int_0^s\frac\partial{\partial u}\widehat{\undertilde{m}_j}(u)\,du.$ 

Using \eqref{Estimate1mhat} we deduce that
\begin{align} \label{Estimatemjsubtildehat}
|\widehat{\undertilde{m}_j}(s)|\lesssim2^{-Mj}\frac{\lambda s}{\sqrt{1+t}}.
\end{align}
Putting \eqref{Estimatesottildemj} and \eqref{Estimatemjsubtildehat} together, we can deduce the following improved bound in the regime $s \le 1 \le t$ from \eqref{mjhatprimebis}\,:
\begin{equation}\label{Estimate1quatermhat}
\bigl|\tfrac\partial{\partial s}\,\widehat{m_j}(s)\bigr|
\lesssim2^{-Mj}\tfrac{\lambda\,s}{t\sqrt{1+t}}\hspace{1pt}.
\end{equation}

%


So far \eqref{Estimate1bismhat}, \eqref{Estimate1termhat}, \eqref{Estimate1quatermhat}
hold true for all $s\ge0$, $t>0$, $\lambda>1$.
Let us next improve \eqref{Estimate1mhat}
under the additional assumption $|\tfrac s{2t}-\lambda|\ge2^{j+2}$.
Then
\begin{equation*}
\bigl|\tfrac s{2t}-\xi\bigr|
\ge\bigl|\tfrac s{2t}-|\xi|\bigr|
\ge\bigl|\tfrac s{2t}-\lambda\hspace{1pt}\bigr|-\underbrace{\bigl|\lambda-|\xi|\bigr|}_{\le\,2^{j+1}}
\ge\tfrac12\,\bigl|\tfrac s{2t}-\lambda\hspace{1pt}\bigr|
\quad\forall\,\xi\!\in\supp a_j\hspace{1pt}.
\end{equation*}
Hence
\begin{equation}\label{Estimate2mhat}
\bigl|\bigl(\tfrac\partial{\partial s}\bigr)^\ell\hspace{1pt}\widehat{m_j}(s)\bigr|
\lesssim2^{-Mj}\tfrac{\lambda^\ell}{\sqrt{1+t}}\,|s\hspace{-1pt}-\hspace{-1pt}2\hspace{1pt}t\lambda|^{-N}
\end{equation}
after performing $N$ integrations by parts based on
\begin{equation*}
e^{i\Phi(\xi)}=\tfrac i{2t}\,\tfrac1{\xi-\frac s{2t}}\,\tfrac\partial{\partial\xi}\,e^{i\Phi(\xi)}.
\end{equation*}
Moreover, the following improved bound can be obtained as before for $\ell=1$\,:
\begin{equation}\label{Estimate2bismhat}
\bigl|\tfrac\partial{\partial s}\,\widehat{m_j}(s)\bigr|
\lesssim2^{-Mj}\,|s\hspace{-1pt}-\hspace{-1pt}2\hspace{1pt}t\lambda|^{-N}\min\hspace{1pt}
\bigl\{{\tfrac{1+s}{t\sqrt{1+t}},\tfrac{\lambda^2s}{\sqrt{1+t}},\hspace{1pt}}
\tfrac{\lambda\,s}{t\sqrt{1+t}}\bigr\}
\end{equation}


\noindent
(b) Let us next estimate
\begin{equation}\label{kj}
K_j(r)=\const\int_r^{\hspace{.5pt}\infty}\tfrac\partial{\partial s}\,\widehat{m_j}(s)\,(\cosh s-\cosh r)^{-\frac12}\hspace{1pt}ds
\end{equation}
by using (a) and the estimate~\eqref{AuxiliaryEstimate}, which we will use repeatedly.

\medskip

\noindent \underline{Case 1.}
\,Assume that $r$ is small, say $0\le r\le1$.
\smallskip

\noindent $\bullet$
\,If $t$ is small, say $0<t\le1$, we split up
\begin{equation*}
\int_r^{\hspace{.5pt}\infty}
=\int_r^{\sqrt{r^2+\lambda^{-2}}}+\int_{\sqrt{r^2+\lambda^{-2}}}^{\hspace{.5pt}\infty}
\end{equation*}
in \eqref{kj}.
By using \eqref{AuxiliaryEstimate} together with \eqref{Estimate1termhat},
the first integral is bounded by
\begin{equation*}
2^{-Mj}\hspace{1pt}\lambda^2
\int_r^{\sqrt{r^2+\lambda^{-2}}}(s^2\!-\hspace{-1pt}r^2)^{-\frac12}\hspace{1pt}s\,ds
=2^{-Mj}\hspace{1pt}\lambda\hspace{1pt}.
\end{equation*}
After an integration by parts, the second integral becomes
\begin{align*}
&\widehat{m_j}(s)\,(\cosh s-\cosh r)^{-\frac12}\Big|_{s=\sqrt{r^2+\lambda^{-2}}}\\
&+\tfrac12\int_{\sqrt{r^2+\lambda^{-2}}}^{\hspace{1pt}r+1}\widehat{m_j}(s)\,(\cosh s-\cosh r)^{-\frac32}\sinh s\,ds\\
&+\tfrac12\int_{r+1}^{\hspace{.5pt}\infty}\widehat{m_j}(s)\,(\cosh s-\cosh r)^{-\frac32}\sinh s\,ds,
\end{align*}
which is also $\text{O}\bigl(2^{-Mj}\lambda\bigr)$,
according to \eqref{AuxiliaryEstimate} and \eqref{Estimate1mhat}.
Hence
\begin{equation}\label{EstimateCase1tSmall}
|K_j(r)|\lesssim2^{-Mj}\lambda\,.
\end{equation}
$\bullet$
\,If $t$ is large, say $t\ge1$,
we obtain
\begin{equation}\label{EstimateCase1tLarge}
|K_j(r)|\lesssim2^{-Mj}\lambda\,t^{-\frac32}
\end{equation}
by splitting up
\begin{equation*}
\int_r^{\hspace{.5pt}\infty}
=\int_r^{r+1}+\int_{r+1}^{\hspace{.5pt}\infty}
\end{equation*}
in \eqref{kj} and by using \eqref{Estimate1bismhat}, \eqref{Estimate1quatermhat}
instead of \eqref{Estimate1mhat}, \eqref{Estimate1termhat}.
{More precisely,
\begin{align*}
|K_j(r)|\lesssim2^{-Mj}\hspace{1pt}\lambda\,t^{-\frac32}
\underbrace{\int_r^{\hspace{1pt}r+1}(s^2-r^2)^{-\frac12}\hspace{1pt}s\,ds}_{\lesssim\,1}
+\,2^{-Mj}\hspace{1pt}t^{-\frac32}
\underbrace{\int_{r+1}^{\hspace{.5pt}\infty}e^{-\frac s2}\hspace{1pt}s\,ds}_{\lesssim\,1}
\lesssim2^{-Mj}\hspace{1pt}\lambda\,t^{-\frac32}\hspace{1pt}.
\end{align*}
}\smallskip

\noindent \underline{Case 2.}
\,Let us improve \eqref{EstimateCase1tSmall} and \eqref{EstimateCase1tLarge}
when $r$ is small while $t\lambda$ and $2^{-j}\lambda$ are large.
\smallskip

\noindent$\bullet$
\,Assume that $r$ is small while $t$ and $2^{-j}\lambda$ are large,
say $0\le r\le1\le t$ and $\lambda\ge2^{j+4}$.
Then
\begin{equation}\label{EstimateSTLambda}
|s-2\hspace{1pt}t\lambda|\ge\tfrac{t\,\lambda}2
\quad\text{and}\quad
\bigl|\tfrac s{2t}-\lambda\hspace{1pt}\bigr|\ge\tfrac\lambda4\ge2^{j+2}
\qquad\forall\,s\in\bigl[r,\tfrac32\hspace{1pt}t\lambda\bigr]
\end{equation}
and we obtain
\begin{equation}\label{EstimateCase2bis}
|K_j(r)|\lesssim2^{-Mj}\,(t\lambda)^{-N}
\qquad\forall\,M,N\!\ge\hspace{-1pt}0
\end{equation}
by splitting up
\begin{equation*}
\int_r^{\hspace{.5pt}\infty}=\int_r^{\hspace{1pt}r+1}
+\int_{r+1}^{\frac32t\lambda}+\int_{\frac32t\lambda}^{\hspace{.5pt}\infty}
\end{equation*}
in \eqref{kj} and by using \eqref{AuxiliaryEstimate} together with
\eqref{Estimate1mhat}, \eqref{Estimate2mhat}, \eqref{Estimate2bismhat}.
{More precisely,
\begin{align*}
|K_j(r)|
&\lesssim2^{-Mj}\hspace{1pt}(t\lambda)^{-N-1}\hspace{1pt}\lambda\,\overbrace{
\int_r^{\hspace{1pt}r+1\vphantom{\frac00}}\!(s^2\!-\hspace{-1pt}r^2)^{-\frac12}\hspace{1pt}s\,ds
}^{\lesssim\,1}+\,2^{-Mj}\hspace{1pt}(t\lambda)^{-N-1}\hspace{1pt}\lambda\,
\overbrace{\int_{r+1}^{\frac32t\lambda}\!e^{-\frac s2}\,ds}^{\lesssim\,1}\\
&+2^{-Mj}\hspace{1pt}\lambda\hspace{-4mm}\underbrace{
\int_{\frac32t\lambda}^{\hspace{.5pt}\infty}\!e^{-\frac s2}\,ds
}_{\lesssim\,e^{-\frac34t\lambda}\,\lesssim\,(t\lambda)^{-N-1}}\hspace{-4mm}
\lesssim2^{-Mj}\hspace{1pt}(t\lambda)^{-N}.
\end{align*}
}

\noindent$\bullet$
\,Assume that $r$ and $t$ are small while $t\lambda$ and $2^{-j}\lambda$ are large,
say $0\hspace{-1pt}\le\hspace{-1pt}r\le\!1$, $\frac1\lambda\!\le\hspace{-1pt}t\hspace{-1pt}\le\!1$
and $\lambda\ge2^{j+4}$.
Then \eqref{EstimateSTLambda} holds and we obtain
\begin{equation}\label{EstimateCase2bis}
|K_j(r)|\lesssim2^{-Mj}\hspace{1pt}(t\lambda)^{-N}\hspace{1pt}\lambda
\qquad\forall\,M,N\!\ge\hspace{-1pt}0
\end{equation}
by following the proof of Case 1: namely, we split up
\begin{equation}\label{SplittedIntegral}
\int_r^{\hspace{.5pt}\infty}=\int_r^{\sqrt{r^2+\lambda^{-2}}}
+\int_{\sqrt{r^2+\lambda^{-2}}}^{\hspace{1pt}r+\frac12}+\int_{r+\frac12}^{\frac32t\lambda}
+\int_{\frac32t\lambda}^{\hspace{.5pt}\infty}
\end{equation}
in \eqref{kj} and use \eqref{AuxiliaryEstimate} together with
\eqref{Estimate1mhat}, \eqref{Estimate2mhat}, \eqref{Estimate2bismhat}.
{More precisely,
the contribution of the first integral in \eqref{SplittedIntegral} is bounded by
\begin{equation*}
2^{-Mj}\hspace{1pt}(t\lambda)^{-N}\hspace{1pt}\tfrac\lambda t\underbrace{
\int_r^{\sqrt{r^2+\lambda^{-2}}}(s^2\!-r^2)^{-\frac12}\hspace{1pt}s\,ds
}_{=\,\frac1\lambda}
\lesssim2^{-Mj}\hspace{1pt}(t\lambda)^{-N}\hspace{1pt}\lambda
\end{equation*}
while, after an integration by parts,
the contribution of the last three integrals in \eqref{SplittedIntegral} is bounded by
\begin{align*}
&|\widehat{m_j}(s)|\,(\cosh s-\cosh r)^{-\frac12}\Big|_{s=\sqrt{r^2+\lambda^{-2}}}
+\tfrac12\int_{\sqrt{r^2+\lambda^{-2}}}^{\hspace{.5pt}\infty}\hspace{1pt}
|\widehat{m_j}(s)|\,(\cosh s-\cosh r)^{-\frac32}\sinh s\,ds\\
&\lesssim2^{-Mj}\,(t\lambda)^{-N}\hspace{1pt}\lambda
+2^{-Mj}\,(t\lambda)^{-N}\hspace{-1pt}\overbrace{
\int_{\sqrt{r^2+\lambda^{-2}}}^{\,r+\frac12}\hspace{1pt}(s^2\!-r^2)^{-\frac32}s\,ds
}^{\lesssim\,\lambda}
+\,2^{-Mj}\hspace{1pt}(t\lambda)^{-N}\hspace{-1pt}
\overbrace{\int_{\,r+\frac12}^{\frac32t\lambda}e^{-\frac s2}\hspace{1pt}ds}^{\lesssim\,1}\\
&+2^{-Mj}\hspace{-7.5pt}\underbrace{
\int_{\frac32t\lambda}^{\hspace{.5pt}\infty}e^{-\frac s2}\hspace{1pt}ds
}_{\lesssim\,e^{-\frac34t\lambda}\,\lesssim\,(t\lambda)^{-N}}\hspace{-7pt}
\lesssim2^{-Mj}\hspace{1pt}(t\lambda)^{-N}\hspace{1pt}\lambda\hspace{1pt}.
\end{align*}}

\noindent \underline{Case 3.}
\,Assume that $r$ is large, say $r\hspace{-1pt}\ge\!1$.
Then we obtain
\begin{equation}\label{EstimateCase3}
|K_j(r)|\lesssim2^{-Mj}\,e^{-\frac r2}
\times\begin{cases}
\,\lambda
&\text{if \,$0\hspace{-1pt}<\hspace{-1pt}t\hspace{-1pt}\le\!1$}\\
\,t^{-\frac32}\,r
&\text{if \,$t\hspace{-1pt}\ge\!1$}\\
\end{cases}\end{equation}
by splitting up
\begin{equation*}
\int_r^{\hspace{.5pt}\infty}
=\int_r^{\hspace{1pt}r+1}
+\int_{r+1}^{\hspace{.5pt}\infty}
\end{equation*}
in \eqref{kj} and by using \eqref{AuxiliaryEstimate}
together with \eqref{Estimate1mhat}, \eqref{Estimate1bismhat}.
{More precisely,
\begin{equation*}
|K_j(r)|\lesssim2^{-Mj}\hspace{1pt}\lambda\,e^{-\frac r2}\underbrace{
\int_r^{\hspace{1pt}r+1}(s\hspace{-1pt}-\hspace{-1pt}r)^{-\frac12}\,ds
}_{\lesssim\,1\vphantom{e^{\frac00}}}
+\,2^{-Mj}\hspace{1pt}\lambda\underbrace{\int_{r+1}^{\hspace{.5pt}\infty}
e^{-\frac s2}\,ds}_{\lesssim\,e^{-\frac r2}}
\lesssim\,2^{-Mj}\hspace{1pt}\lambda\,e^{-\frac r2}
\end{equation*}
if $0\hspace{-1pt}<\hspace{-1pt}t\hspace{-1pt}\le\!1$, while
\begin{equation*}
|K_j(r)|\lesssim2^{-Mj}\,t^{-\frac32}\,r\,e^{-\frac r2}\underbrace{
\int_r^{\hspace{1pt}r+1}(s\hspace{-1pt}-\hspace{-1pt}r)^{-\frac12}\,ds
}_{\lesssim\,1\vphantom{e^{\frac00}}}
+\,2^{-Mj}\,t^{-\frac32}\underbrace{\int_{r+1}^{\hspace{.5pt}\infty}
e^{-\frac s2}\,s\,ds}_{\lesssim\;r\,e^{-\frac r2}}
\lesssim\,2^{-Mj}\,t^{-\frac32}\,r\,e^{-\frac r2}
\end{equation*}
if $t\hspace{-1pt}\ge\!1$.
}\smallskip

\noindent\underline{Case 4.}
\,Let us improve \eqref{EstimateCase3}
when $t\lambda$\hspace{.5pt}, $2^{-j}\lambda$ are both large
and when $\tfrac r{t\hspace{.5pt}\lambda}$ stays away
\linebreak
from~$2$.

\smallskip

\noindent$\bullet$
\,Assume that $\lambda\hspace{-1pt}\ge\hspace{-1pt}2^{j+4}$
and $1\!\le\hspace{-1pt}r\hspace{-1pt}\le \hspace{-1pt}t\lambda$\hspace{.5pt}.
Then \eqref{EstimateSTLambda} holds again and we obtain
\begin{equation}\label{EstimateCase4bis}
|K_j(r)|\lesssim2^{-Mj}\,
(t\lambda)^{-N}\,\lambda \;e^{-\frac r2}
\end{equation}
by splitting up
\begin{equation*}
\int_r^{\hspace{.5pt}\infty}=\int_r^{\hspace{1pt}r+\frac12}
+\int_{r+\frac12}^{\frac32t\lambda}+\int_{\frac32t\lambda}^{\hspace{.5pt}\infty}
\end{equation*}
in \eqref{kj} and by using \eqref{AuxiliaryEstimate}
together with \eqref{Estimate1mhat}, \eqref{Estimate2mhat}.
{More precisely,
\begin{align*}
|K_j(r)|
&\lesssim2^{-Mj}\hspace{1pt}
(t\lambda)^{-N}\hspace{1pt}\lambda\,
e^{-\frac r2}\overbrace{\int_r^{\hspace{1pt}r+\frac12}(s-r)^{-\frac12}\,ds}^{\lesssim\,1}
+\,2^{-Mj}\hspace{1pt}(t\lambda)^{-N}\hspace{1pt}\lambda
\overbrace{\int_{\hspace{1pt}r+\frac12}^{\frac32t\lambda}e^{-\frac s2}\,ds
}^{\lesssim\,e^{-\frac r2}}\\
&+2^{-Mj}\hspace{1pt} \lambda 
\hspace{-5.5mm}\underbrace{\int_{\frac32t\lambda}^{\hspace{.5pt}\infty}e^{-\frac s2}\,ds
}_{\lesssim\,e^{-\frac34t\lambda}\,\lesssim\,(t\lambda)^{-N}\,e^{-\frac r2}}\hspace{-5.5mm}
\lesssim2^{-Mj}\hspace{1pt}
(t\lambda)^{-N}\hspace{1pt}\lambda\,
e^{-\frac r2}\hspace{1pt}.
\end{align*}
}$\bullet$
\,Assume that $\lambda\ge 2^{j+3}$ and that
$r\ge3\,t\lambda\hspace{-1pt}\ge\hspace{-1pt}3$\hspace{.51pt}.
Again
\begin{equation*}
|s-2\hspace{1pt}t\lambda|\ge t\lambda
\quad\text{and}\quad
\bigl|\tfrac s{2t}-\lambda\hspace{1pt}\bigr|\ge\tfrac\lambda2\ge2^{j+2}
\qquad\forall\,s\hspace{-1pt}\ge\hspace{-1pt}r
\end{equation*}
and we obtain \eqref{EstimateCase4bis} by splitting up
\begin{equation*}
\int_r^{\hspace{.5pt}\infty}=\int_r^{\hspace{1pt}r+1}+\int_{r+1}^{\hspace{.5pt}\infty}
\end{equation*}
in \eqref{kj} and by using \eqref{AuxiliaryEstimate}
together with \eqref{Estimate2mhat}.
{More precisely,
\begin{align*}
|K_j(r)|
&\lesssim2^{-Mj}\hspace{1pt}(t\lambda)^{-N}\hspace{1pt}\lambda\,e^{-\frac r2}
\overbrace{\int_r^{\hspace{1pt}r+1}(s-r)^{-\frac12}\,ds}^{\lesssim\,1}
+\,2^{-Mj}\hspace{1pt}(t\lambda)^{-N}\hspace{1pt}\lambda\,
\overbrace{\int_{\hspace{1pt}r+1}^{\hspace{.5pt}\infty}e^{-\frac s2}\,ds}^{\lesssim\,e^{-\frac r2}}\\
&\lesssim2^{-Mj}\hspace{1pt}(t\lambda)^{-N}\hspace{1pt}\lambda\,e^{-\frac r2}\hspace{1pt}.
\end{align*}
}\end{proof}

\begin{corollary}\label{CorollaryH}
Let \hspace{1pt}$N\!\ge\hspace{-1pt}0$\hspace{1pt}.
Then the following kernel estimates hold,
for $r\ge0$, $t>0$ and $\lambda>1:$
\begin{equation*}
|K(r)|\lesssim(1\!+\hspace{-1pt}r)^{N+1}\hspace{1pt}e^{-\frac r2}
\times\begin{cases}
\,\lambda\,(1\!+\hspace{-1pt}t\lambda)^{-N}
&\text{if \,$0\hspace{-1pt}<\hspace{-1pt}t\hspace{-1pt}\le\!1$\hspace{.5pt},}\\
\,t^{-\frac32}\,\lambda^{-N}
&\text{if \,$t\hspace{-1pt}\ge\!1$\hspace{.5pt}.}\\
\end{cases}\end{equation*}
\end{corollary}

\begin{proof}
We use the notations of Lemma \ref{LemmaH2}.
By summing up \eqref{KernelEstimateH2} over $j\in\mathbb{N}$, we obtain
\begin{equation*}
|K(r)|\le\sum\nolimits_{j=0}^{\hspace{.5pt}\infty}|K_j(r)|
\lesssim(1\!+\hspace{-1pt}t)^{-\frac32}\,e^{-\frac r2}\times\begin{cases}
\,r
&\text{if $r\hspace{-1pt}\ge\!1$ and $t\hspace{-1pt}\ge\!1$\hspace{.5pt},}\\
\,\lambda
&\text{otherwise,}\\
\end{cases}\end{equation*}
which implies Corollary \ref{CorollaryH} when
$0\hspace{-1pt}<\hspace{-1pt}t\hspace{-1pt}\le\hspace{-1pt}\tfrac1\lambda$\hspace{1pt}.
When $1\!\le\hspace{-1pt}t\lambda\hspace{-1pt}\le\hspace{-1pt}r\hspace{-1pt}
\le\hspace{-1pt}3\hspace{1pt}t\lambda$\hspace{.5pt},
the above estimate yields
\begin{equation*}
|K(r)|\lesssim\bigl(\tfrac r{t\lambda}\bigr)^{\hspace{-.5pt}N}\hspace{1pt}e^{-\frac r2}
\times\begin{cases}
\,\lambda
&\text{if $\frac1\lambda\!\le\hspace{-1pt}t\hspace{-1pt}\le\!1$\hspace{.5pt},}\\
\,r
&\text{if $t\hspace{-1pt}\ge\!1$\hspace{.5pt}.}\\
\end{cases}\end{equation*}
In the remaining cases,
we split up
\begin{equation}\label{SplittedSum}
|K(r)|\le\underbrace{\sum\nolimits_{\hspace{1pt}2^j\le\frac\lambda{16}}|K_j(r)|}_{\Sigma_1}\,
+\,\underbrace{\sum\nolimits_{\hspace{1pt}2^j>\frac\lambda{16}}|K_j(r)|}_{\Sigma_2}
\end{equation}
and use again \eqref{KernelEstimateH2} to estimate the second sum in \eqref{SplittedSum}.
More precisely,
\begin{equation}\label{PartialEstimate1}
\Sigma_2\lesssim t^{-\frac32}\,r\,e^{-\frac r2}\hspace{1pt}\underbrace{
\sum\nolimits_{\hspace{1pt}2^j>\frac\lambda{16}}\!2^{-Mj}
}_{\lesssim\,\lambda^{-M}}
\lesssim t^{-\frac32}\,\lambda^{-M}\,r\,e^{-\frac r2}
\end{equation}
when \hspace{1pt}$r\hspace{.5pt},\hspace{1pt}t$ are both large and
\begin{equation}\label{PartialEstimate2}
\Sigma_2\lesssim(1\!+\hspace{-1pt}t\hspace{.5pt})^{-\frac32}\,e^{-\frac r2}\hspace{1pt}
\lambda\hspace{1pt}\underbrace{
\sum\nolimits_{\hspace{1pt}2^j>\frac\lambda{16}}\!2^{-(M+1)j}}_{\lesssim\,\lambda^{-M-1}}
\lesssim(1\!+\hspace{-1pt}t)^{-\frac32}\hspace{1pt}\lambda^{-M}\,e^{-\frac r2}
\end{equation}
otherwise.
On the other hand,
we use \eqref{ImprovedKernelEstimateH2}
to estimate the first sum in \eqref{SplittedSum}.
More precisely,
\begin{equation}\label{PartialEstimate3}
\Sigma_1\lesssim(t\lambda)^{-N}\hspace{1pt}\lambda\,
\underbrace{\sum\nolimits_{\hspace{1pt}2^j\le\frac\lambda{16}}\!2^{-Mj}}_{\lesssim\,1} \qquad \mbox{when $0\hspace{-1pt}\le\hspace{-1pt}r\hspace{-1pt}\le\!1$ and
$\tfrac1\lambda\!\le\hspace{-1pt}t\hspace{-1pt}\le\hspace{-1pt}1$}
\end{equation}
\begin{equation}\label{PartialEstimate4}
\Sigma_1\lesssim(t\lambda)^{-N-1}\,\lambda\,
\underbrace{\sum\nolimits_{\hspace{1pt}2^j\le\frac\lambda{16}}\!2^{-Mj}}_{\lesssim\,1}
\lesssim\,(t\lambda)^{-N} \qquad \mbox{when $0\hspace{-1pt}\le\hspace{-1pt}r\hspace{-1pt}\le\!1\!\le\hspace{-1pt}t$\hspace{.5pt}}
\end{equation}
\begin{equation}\label{PartialEstimate5}
\Sigma_1\lesssim(t\lambda)^{-N}\hspace{1pt}\lambda\,e^{-\frac r2}\,
\underbrace{\sum\nolimits_{\hspace{1pt}2^j\le\frac\lambda{16}}\!2^{-Mj}}_{\lesssim\,1} \qquad \mbox{when $r\hspace{-1pt}\ge\!1$, $t\hspace{-1pt}\ge\!\tfrac1\lambda$\hspace{.5pt},
$r\hspace{-1pt}\notin\hspace{-1pt}[\hspace{1pt}t\lambda,3\hspace{1pt}t\lambda\hspace{1pt}]$}.
\end{equation}
By combining \eqref{SplittedSum} with \eqref{PartialEstimate1}, \eqref{PartialEstimate2}, \eqref{PartialEstimate3}, \eqref{PartialEstimate4}, \eqref{PartialEstimate5},
we obtain the desired bounds when \hspace{1pt}$t\lambda\hspace{-1pt}\ge\!1$
\hspace{.5pt}and \hspace{1pt}$r\hspace{-1pt}\notin\hspace{-1pt}
[\hspace{1pt}t\lambda,3\hspace{1pt}t\lambda\hspace{1pt}]$\hspace{1pt}.
\end{proof}

\subsection{Dispersive and Strichartz estimates}

Corollary \ref{CorollaryH} implies the following dispersive estimates.

\begin{proposition}\label{DispersiveH}
Assume that \,$t\hspace{-1pt}>\hspace{-1pt}0$\hspace{1pt}, $\lambda\hspace{-1pt}>\!1$ and
\,$2\hspace{-1pt}<\hspace{-1pt}q\hspace{-.5pt}\le\hspace{-1pt}\infty$\hspace{1pt}.
Then, for every \,$N\hspace{-2pt}\ge\hspace{-1pt}0$\hspace{1pt},
\begin{equation*}
\|\hspace{1pt}e^{\hspace{.5pt}i\hspace{.75pt}t\hspace{.5pt}\Delta}\hspace{1pt}
\mathscr{P}_{\lambda\hspace{.5pt},1}\|_{L^{q'}\!\to L^q}\,
\lesssim\begin{cases}
\,\lambda^{1-\frac2q}\hspace{1pt}(1\!+\hspace{-1pt}t\lambda)^{-N}
&\textit{if \,$t$ is small\hspace{1pt},}\\
\,t^{-\frac32}\hspace{1pt}\lambda^{-N}
&\textit{if \,$t$ is large\hspace{1pt}.}
\end{cases}\end{equation*}
\end{proposition}

\begin{remark}
Compared to other dispersive estimates available in the literature, note that our frequency localization here is of a different type. Indeed the window has size $1$ around $\lambda$.
\end{remark}

\begin{proof}
(a) Assume that $t$ is small,
say $0\hspace{-1pt}<\hspace{-1pt}t\hspace{-1pt}\le\!1$\hspace{1pt}.
Then the $L^{q'}\hspace{-3pt}\to\hspace{-2pt}L^q$ estimate is obtained by interpolation
between the trivial $L^2\hspace{-2.5pt}\to\hspace{-2pt}L^2$ estimate
\begin{equation*}
\|\hspace{1pt}e^{\hspace{.5pt}i\hspace{.75pt}t\hspace{.5pt}\Delta}\hspace{1pt}
\mathscr{P}_{\lambda\hspace{.5pt},1}\|_{L^2\to L^2}\le1
\end{equation*}
and the $L^1\hspace{-2.5pt}\to\hspace{-2pt}L^\infty$ estimate
\begin{equation*}
\|\hspace{1pt}e^{\hspace{.5pt}i\hspace{.75pt}t\hspace{.5pt}\Delta}\hspace{1pt}
\mathscr{P}_{\lambda\hspace{.5pt},1}\|_{L^1\to L^\infty}
\lesssim\lambda\hspace{1pt}(1\!+\hspace{-1pt}t\lambda)^{-N},
\end{equation*}
which follows from Corollary \ref{CorollaryH}.

\noindent
(b) Assume that $t$ is large, say $t\hspace{-1pt}\ge\!1$\hspace{1pt}.
On the one hand, the $L^1\hspace{-2.5pt}\to\hspace{-2pt}L^\infty$ estimate is
an immediate consequence of Corollary \ref{CorollaryH}.
On the other hand, the $L^{q'}\hspace{-3pt}\to\hspace{-2pt}L^q$
estimate for $2\hspace{-1pt}<\hspace{-1pt}q\hspace{-1pt}<\hspace{-1pt}\infty$
follows from the Kunze--Stein phenomenon on $\H$, as stated in Lemma \ref{KunzeSteinH},
combined with Corollary \ref{CorollaryH}.
\end{proof}

Proposition \ref{DispersiveH} imply in turn the following Strichartz estimate.

\begin{proposition}\label{StrichartzH}
Let \,$2\hspace{-1pt}\le\hspace{-1pt}p\hspace{-1pt}\le\hspace{-1pt}\infty$
and \,$2\hspace{-1pt}<\hspace{-1pt}q\hspace{-1pt}\le\hspace{-1pt}\infty$\hspace{1pt}.
Then, for every \,$\lambda\hspace{-1pt}>\!1$ and $f\!\in\!L^2(\H)$,
\begin{equation*}
\|\hspace{1pt}e^{\hspace{.5pt}i\hspace{.75pt}t\hspace{.5pt}\Delta}\hspace{1pt}
\mathscr{P}_{\lambda\hspace{.5pt},1}f\hspace{1pt}\|_{L_t^pL_x^q}
\lesssim\lambda^{\frac12-\frac1p-\frac1q}\,\|f\|_{L^2}\hspace{1pt}.
\end{equation*}
\end{proposition}

\begin{proof}
By the standard $TT^*$ argument, it suffices to prove the dual estimate
\begin{equation}\label{DualStrichartzH2}
\Bigl\|\hspace{1pt}\int_{-\infty}^{\hspace{.5pt}\infty}
e^{\hspace{.5pt}i\hspace{.75pt}t\hspace{.5pt}\Delta}\hspace{1pt}
\mathscr{P}_{\lambda\hspace{.5pt},1}F(t,\hspace{1pt}\cdot\,)\,dt\,\Bigr\|_{L^2}
\lesssim\lambda^{\frac12-\frac1p-\frac1q}\,\|F\|_{L_t^{P^{\hspace{1pt}\prime}}\!L_x^{q'}}\hspace{1pt}.
\end{equation}
The square of the left hand side of \eqref{DualStrichartzH2} is equal to
\begin{equation}\label{TripleIntegral}
\int_{-\infty}^{\hspace{.5pt}\infty}\int_{-\infty}^{\hspace{.5pt}\infty}
\int_\H\,\bigl[e^{\hspace{.5pt}i\hspace{.5pt}(s-t)\hspace{.25pt}\Delta}
\mathscr{P}^2_{\lambda\hspace{.5pt},1}F(s,\hspace{1pt}\cdot\,)\bigr]\hspace{1pt}
\overline{F(t,x)}\,dx\,ds\,dt\hspace{1pt}.
\end{equation}
Denote by
\begin{equation*}
B(t)=\begin{cases}
\,\lambda^{1-\frac2q}\,(1\!+\hspace{-1pt}\lambda|t|)^{-N}
&\textit{if \,$|t|\hspace{-1pt}\le\!1$}\\
\,\lambda^{-N}\hspace{1pt}|t|^{-\frac32}
&\textit{if \,$|t|\hspace{-1pt}>\!1$}
\end{cases}\end{equation*}
the bound obtained in Proposition \ref{DispersiveH} with $N\!>\!1$
and notice that
\begin{equation*}
\|B\|_{L^r}\lesssim\lambda^{1-\frac2q-\frac1r}
\qquad\forall\,r\hspace{-1pt}\ge\hspace{-1pt}1\hspace{1pt}.
\end{equation*}
By applying successively H\"older's inequality,
Proposition \ref{DispersiveH} and Young's inequality,
we estimate \eqref{TripleIntegral} by
\begin{align*}
&\int_{-\infty}^{\hspace{.5pt}\infty}\int_{-\infty}^{\hspace{.5pt}\infty}
\bigl\|\hspace{.5pt}e^{\hspace{.5pt}i\hspace{.5pt}(s\hspace{.25pt}-t)\hspace{.25pt}\Delta}\hspace{1pt}
\mathscr{P}_{\lambda\hspace{.5pt},1}^2F(s,\hspace{1pt}\cdot\,)\hspace{.5pt}\bigr\|_{L^q}
\hspace{1pt}\|F(t,\hspace{1pt}\cdot\,)\|_{L^{q'}}\hspace{1pt}ds\,dt\\
&\lesssim\int_{-\infty}^{\hspace{.5pt}\infty}\int_{-\infty}^{\hspace{.5pt}\infty}
\!B(s\hspace{-1pt}-\hspace{-1pt}t)\,\bigl\|F(s,\hspace{1pt}\cdot\,)\hspace{.5pt}\bigr\|_{L^{q'}}
\hspace{1pt}\|F(t,\hspace{1pt}\cdot\,)\|_{L^{q'}}\hspace{1pt}ds\,dt\\
&\lesssim\|B\|_{L^{p/2}}\,\|F\|_{L^{p^{\hspace{.5pt}\prime}}\!L^{q'}}^2
\lesssim\lambda^{1-\frac2p-\frac2q}\,\|F\|_{L^{p^{\hspace{.5pt}\prime}}\!L^{q'}}^2.
\end{align*}
\end{proof}

\begin{remark}
The results in this section extend again straightforwardly to all hyperbolic spaces $X$
(and even more generally to Damek--Ricci spaces).
In this case (\footnote{\,As usual,\qquad
$\rho\hspace{1pt}=\begin{cases}
\hspace{1pt}\frac{n-1}2\\
\,n\\
\,2\hspace{1pt}n\hspace{-1pt}+\!1\\
\,11\\
\end{cases}$\quad
and\qquad
$d\hspace{1pt}=\begin{cases}
\,n\\
\,2\hspace{1pt}n\\
\,4\hspace{1pt}n\\
\,16\\
\end{cases}$\quad
if\qquad
$\begin{cases}
\hspace{1pt}X\hspace{-1pt}=\H^n\hspace{-1pt}=\text{H}^n(\R)\hspace{.5pt},\\
\hspace{1pt}X\hspace{-1pt}=\text{H}^n(\C)\hspace{.5pt},\\
\hspace{1pt}X\hspace{-1pt}=\text{H}^n(\H)\hspace{.5pt},\\
\hspace{1pt}X\hspace{-1pt}=\text{H}^{\hspace{.5pt}2}(\mathbb{O})\hspace{.5pt}.\\
\end{cases}$}\hspace{.5pt})\hspace{.5pt},
Corollary \ref{CorollaryH} reads
\begin{equation*}
|K(r)|\lesssim(1\!+\hspace{-1pt}r)^{N+1}\hspace{1pt}e^{-\rho\hspace{1pt}r}
\times\begin{cases}
\,\lambda^{\frac{d-1}2}\,(1\!+\hspace{-1pt}t\lambda)^{-N}
&\text{if \,$0\hspace{-1pt}<\hspace{-1pt}t\hspace{-1pt}\le\!1$\hspace{.5pt},}\\
\,t^{-\frac32}\,\lambda^{-N}
&\text{if \,$t\hspace{-1pt}\ge\!1$\hspace{.5pt},}\\
\end{cases}\end{equation*}
Proposition \ref{DispersiveH}
\begin{equation*}
\|\hspace{1pt}e^{\hspace{.5pt}i\hspace{.75pt}t\hspace{.5pt}\Delta}\hspace{1pt}
\mathscr{P}_{\lambda\hspace{.5pt},1}\|_{L^{q'}\!\to L^q}\,
\lesssim\begin{cases}
\,\lambda^{(d\hspace{.5pt}-1)\hspace{.5pt}(1-\frac2q)}
\hspace{1pt}(1\!+\hspace{-1pt}t\lambda)^{-N}
&\text{if \,$t$ is small}\\
\,t^{-\frac32}\hspace{1pt}\lambda^{-N}
&\text{if \,$t$ is large}
\end{cases}\end{equation*}
and Proposition \ref{StrichartzH}
\begin{equation*}
\|\hspace{1pt}e^{\hspace{.5pt}i\hspace{.75pt}t\hspace{.5pt}\Delta}\hspace{1pt}
\mathscr{P}_{\lambda\hspace{.5pt},1}f\hspace{1pt}\|_{L_t^pL_x^q}
\lesssim\lambda^{(d\hspace{.5pt}-1)\hspace{.5pt}(\frac12-\frac1q)\hspace{.5pt}-\frac1p}\,\|f\|_{L^2}\,.
\end{equation*}
\end{remark}

\section{Convex cocompact hyperbolic surfaces with $0\le\delta<\frac12$}\label{SurfacesSmallDelta}

In this section, we consider hyperbolic cylinders and, more generally,
quotients $X\!=\hspace{-1pt}\Gamma\backslash\mathbb{H}$
where $\Gamma$ is a Fuchsian group such that
\begin{itemize}
\item
$\Gamma$ is torsion free,
\item
$\Gamma$ is convex cocompact,
\item
$\Gamma$ has exponent of convergence $\delta<\frac12$
(\footnote{\,Recall that $\delta\hspace{-1pt}=\hspace{-1pt}0$ for hyperbolic cylinders
and that the assumption $\delta\hspace{-1pt}<\hspace{-1pt}\frac12$ excludes cusps.}).
\end{itemize}
In this case, recall \cite{Borthwick2016} that
the spectrum of $\Delta$ on $L^2(X)$ is equal to $[\tfrac14,+\infty)$.
In order to distinguish them from their previous counterparts on $\H$,
we add a superscript $X$ to denote
the spectral projectors and their kernels on $X$.

In this section we show how carefully tracking the dependence on $r$ in the pointwise estimates of the kernels above now allows us to easily generalize our results to the quotients $X.$

\subsection{Boundedness of spectral projectors}\label{SubsectionSpectralProjectorsPressureCondition}
Let us prove that the spectral projectors on the hyperbolic surface $X$
enjoy the same bounds as on the hyperbolic plane $\H$.

\begin{theorem}[high frequency]\label{L2LpBoundsProjectorsX}
For every \,$0<\eta<1<\lambda$ and \,$p>2$,
\begin{equation*}
\|P_{\lambda,\eta}^{\hspace{1pt}X}\|_{L^2\to L^p}
\lesssim\bigl(1\!+\hspace{-1.5pt}\tfrac1{\sqrt{p-2}}\bigr)\,\lambda^{\gamma(p)}\,\eta^\frac12\hspace{1pt}.
\end{equation*}
\end{theorem}

\begin{remark}
We showed in Section \ref{lowerb} that this bound is sharp for the hyperbolic space.
\end{remark}

\begin{proof}
We resume the proof of Theorem \ref{L2LpBoundsProjectorsH} and its notation.
The kernel of $\mathscr{P}_{\lambda,\eta}^{\hspace{1pt}X}$ is given by
\begin{equation*}
p_{\lambda,\eta}^{\hspace{1pt}X}(\Gamma x,\Gamma x')
=\sum\nolimits_{\gamma\in\Gamma}p_{\lambda,\eta}(d(\gamma x,x'))
\end{equation*}
and similarly for $q_{\lambda,\eta}^{\hspace{1pt}X}$, $\sigma_{\lambda,y}^{\hspace{1pt}X}$.
Notice that these series converge absolutely,
owing to the exponential decay $e^{-r/2}$ in the kernel estimates in Lemma \ref{busard}
and to the assumption $\delta<\tfrac12$, which ensures the convergence of the Poincar\'e series
\begin{equation}\label{PoincareSeries}
\sum\nolimits_{\gamma\in\Gamma}e^{-s\hspace{.2mm}d(\gamma x,x')}
\quad\forall\,x,x'\!\in\H,
\end{equation}
for any \hspace{1pt}$\delta\hspace{-1pt}<\hspace{-1pt}s\hspace{-1pt}<\hspace{-1pt}\frac12$\hspace{1pt}.
Actually, under the additional assumption of convex cocompactness,
\eqref{PoincareSeries} is uniformly bounded in $x$ and $x'$
(see for instance \cite[Lemma 3.3]{Zhang2019}).
Thus Lemma \ref{busard} yields the following global kernel bounds,
for $\lambda\hspace{-1pt}>\!1$, $0\hspace{-1pt}<\hspace{-1pt}\eta\hspace{-1pt}\le\hspace{-1pt}\lambda$,
$x,x'\!\in\!X$, $\delta\hspace{-1pt}<\hspace{-1pt}s\hspace{-1pt}<\hspace{-1pt}\tfrac12$ and
$y\hspace{-1pt}\in\hspace{-1pt}\R:$
\vspace{-1mm}
\begin{equation*}\begin{cases}
\,\bigl|p_{\lambda,\eta}^{\hspace{1pt}X}(x,x')\bigr|
\lesssim\lambda\,\eta\,,\\
\,\bigl|q_{\lambda,\eta}^{\hspace{1pt}X}(x,x')\bigr|
\lesssim_{\hspace{1pt}s}\lambda^{\frac12}\,\eta\,(1\!+\hspace{-1pt}\eta)^{\frac12}\,e^{-\frac{1-2s}{8\eta}},\\
\,\bigl|\sigma_{\lambda,y}^{\hspace{1pt}X}(x,x')\bigr|
\lesssim\lambda^{\frac12}\,.
\end{cases}\end{equation*}
\vspace{-1mm}

\noindent
Then we proceed as in the proof of Theorem \ref{L2LpBoundsProjectorsH}.
\end{proof}

\begin{proposition}[low frequency]\label{LambdaSmallX}
When $\lambda$ and $\eta$ are both small,
say \,$0\hspace{-1pt}\le\hspace{-1pt}\lambda\hspace{-1pt}\le\!1$
and \,$0\hspace{-1pt}<\hspace{-1pt}\eta\hspace{-1pt}<\!1$,
we have
\begin{equation*}
\bigl\|P_{\lambda,\eta}^X\bigr\|_{L^2\to L^p}
\lesssim(\lambda\hspace{-1pt}+\hspace{-1pt}\eta)\,\eta^{\frac12}
\quad\text{and}\quad
\bigl\|\mathscr{P}_{\lambda,\eta}^X\bigr\|_{L^2\!\to L^p}
\lesssim(\lambda\hspace{-1pt}+\hspace{-1pt}\eta)\,\eta^{\frac12}
\end{equation*}
for \,$2\hspace{-1pt}<\hspace{-1pt}p\hspace{-1pt}\le\hspace{-1pt}\infty$\hspace{1pt}.
\end{proposition}
\begin{proof}
Resuming the proof of Proposition \ref{LambdaSmallH},
we estimate this time uniformly
$$
\bigl|\hspace{.5pt}p_{\lambda,\eta}^{\hspace{1pt}X}(x,x')\bigr|
\lesssim(\lambda^2\!+\hspace{-1pt}\eta^2)\,\eta
$$
and conclude by using the following version of the Kunze--Stein phenomenon
instead of Lemma~\ref{KunzeSteinH}.
\end{proof}

\begin{lemma}[see Lemma 3.4 in \cite{Zhang2019}]\label{KunzeSteinX}
Let \,$0\hspace{-1pt}<\hspace{-1pt}\epsilon\hspace{-1pt}<\hspace{-1pt}\frac12\!-\hspace{-1pt}\delta$
and \,$2\hspace{-1pt}\le\hspace{-1pt}q\hspace{-1pt}<\hspace{-1pt}\infty$\hspace{1pt}. Then
\begin{equation*}
\|\hspace{.5pt}f\!*\hspace{-1pt}K\hspace{1pt}\|_{L^q}
\lesssim\|f\|_{L^{q'}}\hspace{1pt}\Bigl[\hspace{1pt}\int_0^{\hspace{.5pt}\infty}\!
|K(r)|^{\frac q2}\hspace{1pt}(1\!+\hspace{-1pt}r)\,e^{-\frac r2}\,
e^{\hspace{1pt}(\frac q2-1)(\delta\hspace{.5pt}+\hspace{.5pt}\epsilon)\hspace{1pt}r}
\sinh r\,dr\hspace{1pt}\Bigr]^{\frac2q},
\end{equation*}
for every \,$f\hspace{-1pt}\in\hspace{-1pt}L^{q'}\hspace{-1pt}(X)$
and for every radial measurable function \,$K$ on \,$\H$.
\end{lemma}

\begin{remark}
The results in this subsection extend straightforwardly to
locally symmetric spaces \hspace{.5pt}$\Gamma\hspace{1pt}\backslash G/K$
such that
\begin{itemize}
\item
$\hspace{.5pt}\operatorname{rank}\hspace{1pt}(G/K)\hspace{-1pt}=\hspace{-1pt}1$\hspace{.5pt},
\item
$\Gamma$ is a discrete subgroup of $G$, which is torsion free and convex cocompact,
\item
$\Gamma$ has exponent of convergence \hspace{1pt}$\delta\hspace{-1pt}<\hspace{-1pt}\rho$\hspace{1pt}.
\end{itemize}
In this case,
the high frequency bound in Theorem \ref{L2LpBoundsProjectorsX} is given by \eqref{Generalization}
while the low frequency bound in Proposition \ref{LambdaSmallX} remains the same.
\end{remark}

\subsection{Smoothing estimates}
We first deduce from Subsection \ref{SubsectionSpectralProjectorsPressureCondition}
the following global $L^p$ smoothing estimates for the Schr\"odinger equation on $X$.

\begin{theorem} \label{SmoothX}
Let \,$p\hspace{-1pt}>\hspace{-1pt}2$\hspace{1pt}.
Then\footnote{\, Notice that
$0\hspace{-1pt}<\hspace{-1pt}\gamma(p)\hspace{-1pt}\le\hspace{-1pt}\frac12$
when $2\hspace{-1pt}<\hspace{-1pt}p\hspace{-1pt}\le\hspace{-1pt}\infty$\hspace{.5pt}.}
\begin{align}
\label{smo1}&\Bigl\|\,D_X^{\frac12-\gamma(p)}\,
e^{\hspace{.5pt}i\hspace{1pt}t\hspace{.5pt}\Delta_X}f(x)\,\Bigr\|_{L^p_xL^2_t}
\lesssim\bigl\|f(x)\bigr\|_{L^2_x}\,,\\
\label{smo2}&\Bigl\|\,D_X^{\frac12-\gamma(p)}\!\int_{-\infty}^{\hspace{1pt}\infty}\!
e^{\hspace{.5pt}i\hspace{1pt}t\hspace{.5pt}\Delta_X}F(t,x)\,dt\,\Bigr\|_{L^2_x}
\lesssim\bigl\|F(t,x)\bigr\|_{L^{p^{\hspace{.5pt}\prime}}_x\!L^2_t}\,,\\
\label{smo3}&\Bigl\|\,D_X^{1-2\hspace{.5pt}\gamma(p)}\!\int_{-\infty}^{\hspace{.5pt}\infty}\!
e^{\hspace{.5pt}i\hspace{.5pt}(t-s)\hspace{.5pt}\Delta_X}F(s,x)\,ds\,\Bigr\|_{L^p_xL^2_t}
\lesssim\bigl\|F(t,x)\bigr\|_{L^{p^{\hspace{.5pt}\prime}}_x\!L^2_t}\,.
\end{align}
\end{theorem}

\begin{proof}
We proceed along the same lines as \cite[Subsection 7.2]{GermainLeger2023}. 
Using functional calculus, a change of variables, Plancherel's identity in $t$
and finally Minkowski's inequality, we obtain
\begin{align*}
\Bigl\|\,D_X^{\frac12-\gamma(p)}\,
e^{\hspace{.5pt}i\hspace{1pt}t\hspace{.5pt}\Delta_X}f(x)\,\Bigr\|_{L^p_xL^2_t}\hspace{-1pt}
&=\biggl\|\hspace{1pt}\int_0^{\infty}\!
e^{\hspace{.5pt}i\hspace{1pt}t\hspace{.5pt}\lambda^2}\lambda^{\frac12-\gamma(p)}\,
\mathbf{P}^{\hspace{.5pt}X}_{\hspace{-1pt}\lambda}\!f(x)\,d\lambda\hspace{1pt}\biggr\|_{L^p_xL^2_t}\\
&=\tfrac12\,\biggl\|\hspace{1pt}\int_0^{\infty}\!
e^{\hspace{.5pt}i\hspace{1pt}t\hspace{.5pt}\lambda}\hspace{1pt}\lambda^{-\frac14-\frac{\gamma(p)}2}\,
\mathbf{P}^{\hspace{1pt}X}_{\hspace{-2pt}\sqrt{\lambda}}\hspace{1pt}f(x)\,
d\lambda\hspace{1pt}\biggr\|_{L^p_xL^2_t}\!\\
&=\sqrt{\tfrac\pi2}\,\Bigl\|\hspace{.5pt}\mathbf{1}_{(0,\infty)}(\lambda)\hspace{1pt}
\lambda^{-\frac14-\frac{\gamma(p)}2}\,
\mathbf{P}^{\hspace{1pt}X}_{\hspace{-2pt}\sqrt{\lambda}}
\hspace{1pt}f(x)\hspace{1pt}\Bigr\|_{L^p_xL^2_\lambda}\\
&=\sqrt{2\pi}\,\Bigl\|\hspace{.5pt}\mathbf{1}_{(0,\infty)}(\lambda)\hspace{1pt}
\lambda^{-\hspace{.5pt}\gamma(p)}\,\mathbf{P}^{\hspace{.5pt}X}_{\hspace{-1pt}\lambda}\!
f(x)\hspace{1pt}\Bigr\|_{L^p_xL^2_\lambda}\!\\
&\le\sqrt{2\pi}\,\Bigl\|\hspace{.5pt}\mathbf{1}_{(0,\infty)}(\lambda)\hspace{1pt}
\lambda^{-\hspace{.5pt}\gamma(p)}\,\mathbf{P}^{\hspace{.5pt}X}_{\hspace{-1pt}\lambda}\!
f(x)\hspace{1pt}\Bigr\|_{L^2_\lambda L^p_x}\,.
\end{align*}
Writing \,$\mathbf{P}^{\hspace{.5pt}X}_{\hspace{-1pt}\lambda}\hspace{-2pt}
=\mathbf{E}^{\hspace{.5pt}X}_\lambda\hspace{.5pt}\mathbf{R}^{\hspace{.5pt}X}_\lambda$
as in Subsection \ref{FTX} and using
$$
\|\hspace{1pt}\mathbf{E}^{\hspace{.5pt}X}_\lambda\|_{L^2\to L^p}
\lesssim\begin{cases}
\,\lambda^{\frac12}&\text{if \,$0\hspace{-1pt}\le\hspace{-1pt}\lambda\hspace{-1pt}\le\!1$}\\
\,\lambda^{\gamma(p)}&\text{if \,$\lambda\hspace{-1pt}>\!1$}\\
\end{cases}$$
(see Proposition \ref{EquivalentOperatorNormEstimates}), we conclude that
\begin{align*}
\bigl\|\hspace{.5pt}D_X^{\frac12-\gamma(p)}\hspace{1pt}
e^{\hspace{.5pt}i\hspace{.5pt}t\hspace{.5pt}\Delta_X}\hspace{-.5pt}f(x)
\hspace{.5pt}\bigr\|_{L^p_xL^2_t}\hspace{-1pt}
&\lesssim\bigl\|\hspace{.5pt}\mathbf{1}_{(0,\infty)}(\lambda)\,
\mathbf{R}^{\hspace{.5pt}X}_\lambda\!
f(\theta)\hspace{.5pt}\bigr\|_{L^2_\lambda L^2_\theta}
\le\|f(x)\|_{L^2_x}\,.
\end{align*}
\end{proof}

\begin{remark}
Theorem \ref{SmoothX} holds actually with $D_X^{\frac12-\gamma(p)}$ replaced by $D_{\hspace{-.5pt}X}^{\hspace{.5pt}\alpha\vphantom{\frac00}}\hspace{1pt}\Delta_X^{\hspace{-.5pt}\frac12\left(\frac12-\gamma(p)-\alpha\right)}$, where $\alpha\hspace{-1pt}\ge\hspace{-1pt}-\hspace{1pt}\frac12\hspace{.5pt}.$
\end{remark}

We next deduce from Theorem \ref{SmoothX}
the following local $L^2$ smoothing estimate for the Schr\"odinger equation on $X$.

\begin{proposition}\label{KatosmoothingX}
Let \,$\zeta\!\in\hspace{-.5pt}\mathcal{C}_c^\infty(X)$
and \,$0\hspace{-1pt}<\hspace{-1pt}\epsilon\hspace{-1pt}<\hspace{-1pt}\tfrac12$\hspace{.5pt}.
Then the following estimate holds, for every \,$f\hspace{-1pt}\in\hspace{-1pt}L^2(X)\!:$
\begin{align*}
\bigl\|\hspace{1pt}\zeta(x)\,D_X^{\frac12-\hspace{1pt}\epsilon}
e^{\hspace{.5pt}i\hspace{1pt}t\hspace{.5pt}\Delta_X}f(x)\bigr\|_{L^2_{t,x}}
\lesssim\bigl\|f(x)\bigr\|_{L^2_x}\,.
\end{align*}
\end{proposition}

\begin{proof}
Set
\begin{align*}
u(t,x)=D_X^{\frac12-\hspace{1pt}\epsilon}\hspace{1pt}
e^{\hspace{.5pt}i\hspace{1pt}t\hspace{.5pt}\Delta_X}f(x)
\end{align*}
and let $1\!<\hspace{-1pt}\widetilde{p}\hspace{-1pt}<\hspace{-1pt}\infty$
be such that $\gamma(2\hspace{1pt}\widetilde{p})\hspace{-1pt}=\hspace{-1pt}\varepsilon$\hspace{1pt},
namely
\begin{equation*}
\widetilde{p}=\begin{cases}
\frac{1}{1\hspace{.5pt}-\hspace{1pt}4\hspace{1pt}\varepsilon}
&\text{if \,}0\hspace{-1pt}<\hspace{-1pt}\varepsilon\hspace{-1pt}\le\hspace{-1pt}\frac16\hspace{1pt},\\
\frac{2}{1\hspace{.5pt}-\hspace{1pt}2\hspace{1pt}\varepsilon}
&\text{if \,}\frac16\hspace{-1pt}\le\hspace{-1pt}\varepsilon\hspace{-1pt}<\hspace{-1pt}\frac12\hspace{1pt}.\\
\end{cases}\end{equation*}
Then the desired estimate is obtained by permuting the time and space variables,
by applying H\"{o}lder's inequality and by using \eqref{smo1}. Specifically,

\begin{align*}
\bigl\|\hspace{1pt}\zeta(x)\,u(t,x)\bigr\|_{L^2_{t,x}}^{\hspace{1pt}2}\!
&=\int_X|\hspace{.5pt}\zeta(x)|^2\hspace{1pt}
\Bigl(\int_{-\infty}^{\hspace{1pt}\infty}\!
|\hspace{.5pt}u(t,x)|^2\hspace{1pt}dt{\Bigr)}\hspace{1pt}dx\\
&\le\underbrace{\bigl\|\hspace{.5pt}\zeta(x)
\bigr\|_{L^{2\hspace{.5pt}\widetilde{p}^{\hspace{.5pt}\prime}}_x}^{\hspace{1pt}2}}_{\const}\hspace{1pt}
\bigl\|\hspace{.5pt}u(t,x)\bigr\|_{L^{2\hspace{.5pt}\widetilde{p}}_xL^2_t}^{\hspace{1pt}2}
\lesssim\bigl\|f(x)\bigr\|_{L^2_x}^{\hspace{1pt}2}\,.
\end{align*}

\end{proof}


\begin{remark}
We might also adapt the proof of Theorem \ref{SmoothX},
replacing the $L^2\!\to \hspace{-1pt}L^p$ bounds on the restriction operators by corresponding weighted $L^2$ bounds.
\end{remark}

\subsection{Refined dispersive and Strichartz estimates}\label{DispersiveStrichartzX}
This subsection is devoted to the analogs of Propositions
\ref{DispersiveH} and \ref{StrichartzH} on $X$.
Consider the operators
$e^{\hspace{.5pt}i\hspace{.75pt}t\hspace{.5pt}\Delta_X}\mathscr{P}_{\lambda\hspace{.5pt},1}^X$
on $X$\hspace{-1pt}, where $t\hspace{-1pt}\in\hspace{-1pt}\R^*$\hspace{-1pt}, $\lambda\hspace{-1pt}>\!1$
and $\chi\hspace{-1pt}\in\hspace{-1pt}\mathcal{S}(\R)$ is an even bump function.
By arguing as in the proof of Theorem \ref{L2LpBoundsProjectorsX},
we deduce from Corollary \ref{CorollaryH} the following uniform kernel estimates on~$X$\,:
\begin{equation}\label{KernelEstimateX}
\bigl|k_{\hspace{.5pt}t,\lambda}^{\hspace{.5pt}X}(x,x')\bigr|
\lesssim\begin{cases}
\,\lambda\,(1\!+\hspace{-1pt}t\lambda)^{-N}
&\text{if \,$0\hspace{-1pt}<\hspace{-1pt}t\hspace{-1pt}\le\!1$}\\
\,t^{-\frac32}\,\lambda^{-N}
&\text{if \,$t\hspace{-1pt}\ge\!1$}\\
\end{cases}
\qquad\forall\,x,x'\!\in\!X.
\end{equation}
On the one hand, we obtain by interpolation the following dispersive estimate
for $2\hspace{-1pt}<\hspace{-1pt}q\hspace{-1pt}\le\hspace{-1pt}\infty$ and $t$ small,
say $0\hspace{-1pt}<\hspace{-1pt}t\hspace{-1pt}\le\!1:$
\begin{equation}\label{DispersiveXtSmall}
\|\hspace{1pt}e^{\hspace{.5pt}i\hspace{.75pt}t\hspace{.5pt}\Delta_X}
\mathscr{P}_{\lambda\hspace{.5pt},1}^{\hspace{.5pt}X}\|_{L^{q'}\!\to L^q}\,
\lesssim\lambda^{1-\frac2q}\hspace{1pt}(1\!+\hspace{-1pt}t\lambda)^{-N}.
\end{equation}
On the other hand, when $2\hspace{-1pt}<\hspace{-1pt}q\hspace{-1pt}<\hspace{-1pt}\infty$
and $t$ is large, we obtain
\begin{equation}\label{DispersiveXtLarge}
\|\hspace{1pt}e^{\hspace{.5pt}i\hspace{.75pt}t\hspace{.5pt}\Delta_X}
\mathscr{P}_{\lambda\hspace{.5pt},1}^{\hspace{.5pt}X}\|_{L^{q'}\!\to L^q}\,
\lesssim t^{-\frac32}\hspace{1pt}\lambda^{-N}
\end{equation}
by using again Corollary \ref{CorollaryH} and
by applying the following version of the Kunze--Stein phenomenon on~$X$.
In the limit case $q\hspace{-1pt}=\hspace{-1pt}\infty$\hspace{.5pt},
note that \eqref{DispersiveXtLarge} is a trivial consequence of \eqref{KernelEstimateX}.

\begin{lemma}[see Lemma 3.4 in \cite{Zhang2019}]\label{KunzeSteinX}
Let \,$0\hspace{-1pt}<\hspace{-1pt}\epsilon\hspace{-1pt}<\hspace{-1pt}\frac12\!-\hspace{-1pt}\delta$
and \,$2\hspace{-1pt}\le\hspace{-1pt}q\hspace{-1pt}<\hspace{-1pt}\infty$\hspace{1pt}. Then
\begin{equation*}
\|\hspace{.5pt}f\!*\hspace{-1pt}k\hspace{1pt}\|_{L^q}
\lesssim\|f\|_{L^{q'}}\hspace{1pt}\Bigl[\hspace{1pt}\int_0^{\hspace{.5pt}\infty}\!
|K(r)|^{\frac q2}\hspace{1pt}(1\!+\hspace{-1pt}r)\,e^{-\frac r2}\,
e^{\hspace{1pt}(\frac q2-1)(\delta\hspace{.5pt}+\hspace{.5pt}\epsilon)\hspace{1pt}r}
\sinh r\,dr\hspace{1pt}\Bigr]^{\frac2q},
\end{equation*}
for every \,$f\hspace{-1pt}\in\hspace{-1pt}L^{q'}\hspace{-1pt}(X)$
and for every radial measurable function \,$k$ on \,$\H$.
\end{lemma}

Finally the following Strichartz estimate is deduced
from the dispersive estimates \eqref{DispersiveXtSmall} and \eqref{DispersiveXtLarge},
as Proposition \ref{StrichartzH} from Proposition \ref{DispersiveH}.

\begin{proposition}\label{StrichartzX}
Let \,$2\hspace{-1pt}\le\hspace{-1pt}p\hspace{-1pt}\le\hspace{-1pt}\infty$
and \,$2\hspace{-1pt}<\hspace{-1pt}q\hspace{-1pt}\le\hspace{-1pt}\infty$\hspace{1pt}.
Then, for every \,$\lambda\hspace{-1pt}>\!1$ and $f\!\in\!L^2(X)$,
\begin{equation*}
\|\hspace{1pt}e^{\hspace{.5pt}i\hspace{.75pt}t\hspace{.5pt}\Delta_{X}}
\mathscr{P}_{\hspace{-.5pt}\lambda\hspace{.5pt},1}^{X}\hspace{1pt}f\hspace{1pt}\|_{L_t^p L_x^q}
\lesssim \lambda^{\frac12-\frac1p-\frac1q} \,\|f\|_{L^2}\hspace{1pt}.
\end{equation*}
\end{proposition}

\section{Convex cocompact hyperbolic surfaces with infinite area}
\label{SurfacesLargeDelta}

In this section, we extend our high frequency bounds for $\|P_{\lambda,\eta}^X\|_{L^2\to L^p}$,
though with a small loss, to convex cocompact hyperbolic surfaces $X$ with exponent of convergence
\hspace{.5pt}$\frac12\hspace{-1pt}\le\hspace{-1pt}\delta\hspace{-1pt}<\hspace{-1pt}1\hspace{.5pt}$.
This obtained by piecing together the resolvent estimate in \cite{BourgainDyatlov2018}
with the results of the previous section, following a strategy going back to \cite{StaffilaniTataru2002}
(see also \cite{Wang2019}).

According to \cite{Borthwick2016},
any such surface can be decomposed
into a compact component $X_0$ and finitely many funnels $F_j$,
which are halves of cylinders $C_j$ and whose boundaries are circles\,:
$$
X=X_0\cup\bigcup_{j=1}^nF_j
\quad\text{with}\quad
\partial X_0=\bigsqcup_{j=1}^n\partial F_j\,.
$$
Consider an associated smooth partition of unity
\begin{equation}\label{PartitionUnity}
1=\zeta_0+\sum_{j=1}^n\zeta_j\,,
\end{equation}
where $\zeta_0\hspace{-2pt}\in\hspace{-1pt}\mathcal{C}_c^{\hspace{.5pt}\infty}(X)$
is equal to $1$ on a  neighborhood of $X_0$
and $\zeta_j\!\in\hspace{-1pt}\mathcal{C}^{\hspace{.5pt}\infty}(X)$
is supported in $F_j$ away from $\partial F_j$, say
\hspace{.5pt}$d\hspace{.5pt}(\supp\zeta_j,\partial F_j)\hspace{-1pt}>\hspace{-1pt}3$\hspace{1pt}.

\begin{theorem}
The following estimate holds,
for every \,$p\hspace{-1pt}>\hspace{-1pt}2$\hspace{.5pt},
$\epsilon\hspace{-1pt}>\hspace{-1pt}0$\hspace{.5pt},
$M\!>\hspace{-1pt}0$
and for every \,$0\hspace{-1pt}<\hspace{-1pt}\eta\hspace{-1pt}<\!1\!<\hspace{-1pt}\lambda$
such that \,$\lambda^M\eta\hspace{-1pt}>\!1\!:$
$$
\|P_{\lambda,\eta}f\|_{L^2\to L^p}
\lesssim\lambda^{\gamma(p)+\hspace{.5pt}\epsilon}\,\eta^{\frac12}\hspace{1pt}.
$$
\end{theorem}

\begin{proof}
According to Lemma \ref{comparison}, we can replace
$P_{\lambda,\eta}\hspace{-1pt}=\hspace{-.5pt}\mathbf{1}_{[-1,1]}\bigl(\tfrac{D-\lambda}\eta\bigr)$
by $\phi\bigl(\tfrac{D-\lambda}\eta\bigr)^{\hspace{-.5pt}2}$,
where $\phi$ is a smooth even bump function such that \hspace{1pt}$\supp\phi\hspace{-1pt}\subset\hspace{-1pt}[-\frac12,\frac12\hspace{1pt}]$\hspace{1pt}.
Let us factorize
$$
\phi\bigl(\tfrac{D-\lambda}\eta\bigr)
=\mathscr{P}_{\hspace{-1pt}\lambda,1}^{\hspace{1pt}3}\hspace{1pt}
\mathbb{P}_{\hspace{-.5pt}\lambda,\eta}\hspace{1pt},
$$
where
$$
\mathscr{P}_{\lambda,1}=\chi\bigl(D\hspace{-1pt}-\hspace{-1pt}\lambda\bigr)
+\chi\bigl(D\hspace{-1pt}+\hspace{-1pt}\lambda\bigr)
$$
and
$$
\mathbb{P}_{\hspace{-.5pt}\lambda,\eta}
=\frac{\phi\bigl(\frac{D-\lambda}\eta\bigr)}
{\bigl[\hspace{1pt}\chi\bigl(D\hspace{-1pt}-\hspace{-1pt}\lambda\bigr)
+\chi\bigl(D\hspace{-1pt}+\hspace{-1pt}\lambda\bigr)\bigr]^3}\,.
$$
Here $\chi$ is an even Schwartz function on $\R$ such that $\chi\hspace{-1pt}>\hspace{-1pt}0$, $\chi(0)\hspace{-1pt}=\hspace{-1pt}1$ and $\supp\widehat{\chi}\hspace{-1pt}\subset\hspace{-1pt}[-1,1]$.
The operators $\mathscr{P}_{\lambda,1}$ and $\mathbb{P}_{\hspace{-.5pt}\lambda,\eta}$
can be expressed as follows, by means of the wave and the Schr\"odinger groups\,:
\begin{equation}\label{SubordinationWave}
\mathscr{P}_{\lambda,1}=\sqrt{\hspace{-.5pt}\tfrac8\pi}
\int_{\hspace{.5pt}0}^{\hspace{.5pt}1}\cos(tD)\cos(\lambda t)\,\widehat{\chi}(t)\,dt
\end{equation}
and
\begin{equation}\label{SubordinationSchrodinger}
\mathbb{P}_{\hspace{-.5pt}\lambda,\eta}
=\tfrac1{\sqrt{2\pi}}\int_{-\infty}^{\infty}\!
Z_{\lambda,\eta}(t)\,e^{\hspace{1pt}i\hspace{1pt}t\hspace{.5pt}D^2}dt\hspace{1pt},
\end{equation}
where \hspace{1pt}$Z_{\lambda,\eta}$ denotes the Fourier transform of
$$
\tau\longmapsto\begin{cases}
\frac{\phi\bigl(\frac{\sqrt{\tau}\hspace{1pt}-\hspace{1pt}\lambda}\eta\bigr)}
{\bigl[\hspace{1pt}\chi\bigl(\sqrt{\tau}-\lambda\bigr)
+\hspace{.5pt}\chi\bigl(\sqrt{\tau}+\lambda\bigr)\bigr]^3}
&\text{if \,$\tau\hspace{-1.5pt}\ge\hspace{-1pt}0$\hspace{1pt},}\\
\hspace{46pt}0
&\text{if \,$\tau\hspace{-1.5pt}<\hspace{-1pt}0$\hspace{1pt}.}
\end{cases}
$$
Notice on the one hand that $\mathscr{P}_{\lambda,1}$ has range $\le\hspace{-1pt}1$, 
by finite propagation speed for the wave equation.
Notice on the other hand that\,{\footnote{
\;Proof of \eqref{PointwiseEstimateZ}\,:
\begin{align*}
|t|^L\,|Z_{\lambda,\eta}(t)|
&\le\tfrac1{\sqrt{2\pi}}\int_{\hspace{1pt}0}^{\hspace{.5pt}\infty}\,
\biggl|\hspace{1pt}\bigl(\tfrac\partial{\partial\tau}\bigr)^{\hspace{-1pt}L}
\frac{\phi\bigl(\frac{\sqrt{\tau}\hspace{1pt}-\hspace{1pt}\lambda}\eta\bigr)}
{\bigl[\hspace{1pt}\chi\bigl(\sqrt{\tau}-\lambda\bigr)\hspace{-1pt}
+\hspace{-1pt}\chi\bigl(\sqrt{\tau}+\lambda\bigr)\bigr]^3}\,\biggr|\,d\tau\\
&=\tfrac1{\sqrt{2\pi}}\hspace{1pt}
(2\hspace{1pt}\eta\hspace{.5pt})^{1\hspace{-.5pt}-\hspace{-.5pt}L}\!
\int_{-\frac12}^{\hspace{.5pt}\frac12}\,
\biggl|\hspace{1pt}\bigl(\tfrac1{\lambda+\hspace{.5pt}\eta\hspace{.5pt}\tau}
\tfrac\partial{\partial\tau}\bigr)^{\hspace{-1pt}L}
\frac{\phi(\tau)\vphantom{\big|}}{\bigl[\hspace{1pt}\chi(\eta\hspace{1pt}\tau)\hspace{-1pt}+\hspace{-1pt}
\chi(2\hspace{.5pt}\lambda\hspace{-1pt}+\hspace{-1pt}\eta\hspace{1pt}\tau)\bigr]^3}\,\biggr|
\,(\lambda\hspace{-1pt}+\hspace{-1pt}\eta\hspace{1pt}\tau)\,d\tau\\
&\lesssim(\lambda\hspace{1pt}\eta\hspace{.5pt})^{1\hspace{-.5pt}-\hspace{-.5pt}L}.
\end{align*}}}
\begin{equation}\label{PointwiseEstimateZ}
|Z_{\lambda,\eta}(t)|\lesssim_{\hspace{.5pt}L}\hspace{-1pt}\lambda\hspace{1pt}\eta\,
\bigl(1\!+\!\lambda\hspace{1pt}\eta\hspace{1pt}|t|\hspace{.5pt})^{-L}
\qquad\forall\,L\!\in\hspace{-1pt}\mathbb{N}\hspace{1pt},
\end{equation}
hence
\begin{equation}\label{LrEstimateZ}
\|Z_{\lambda,\eta}\|_{L^r}\lesssim(\lambda\hspace{1pt}\eta)^{\frac1{r^{\hspace{.5pt}\prime}}}
\qquad\forall\;1\!\le\hspace{-1pt}r\hspace{-1pt}\le\hspace{-1pt}\infty\hspace{1pt}.
\end{equation}
Let us finally decompose
$$
\phi\bigl(\tfrac{D-\lambda}\eta\bigr)^{\hspace{-.5pt}2}\hspace{-1pt}
=\mathscr{P}_{\hspace{-1pt}\lambda,1}^{\hspace{1pt}3}\hspace{1pt}\zeta_0
\underbrace{\mathbb{P}_{\lambda,\eta}\,\phi\bigl(\tfrac{D-\lambda}\eta\bigr)
}_{\widetilde{\mathbb{P}}_{\lambda,\eta}}
+\sum_{j=1}^n\mathscr{P}_{\hspace{-1pt}\lambda,1}^{\hspace{1pt}3}\hspace{1pt}\zeta_j\,
\underbrace{\mathbb{P}_{\lambda,\eta}\,\phi\bigl(\tfrac{D-\lambda}\eta\bigr)
}_{\widetilde{\mathbb{P}}_{\lambda,\eta}}
$$
by using the partition of unity \eqref{PartitionUnity}.
\medskip

\noindent\underline{The compact part.}
We learn from Bourgain-Dyatlov~\cite[Theorem 2]{BourgainDyatlov2018} that, for any $\epsilon>0$,
there exists $C_\epsilon\hspace{-1pt}>\hspace{-1pt}0$ such that
$$
\|\hspace{1pt}\zeta_0\hspace{1pt}(D^{\hspace{.5pt}2}\!-\hspace{-1pt}\lambda^2\hspace{-1pt}\pm\hspace{-.5pt}i\hspace{1pt}0\hspace{1pt})^{-1}\hspace{1pt}\zeta_0\hspace{1pt}
\|_{L^2\to L^2}\lesssim_{\hspace{1pt}\epsilon}\lambda^{-1+\hspace{1pt}2\hspace{1pt}\epsilon}
\qquad\forall\,\lambda\hspace{-1pt}\ge\hspace{-1pt}{C_\epsilon}\hspace{1pt}.
$$
Assume first that
$\lambda\hspace{-1pt}>\hspace{-1pt}\max\hspace{1pt}\{1,C_\epsilon\}$\hspace{.5pt}. Then it
follows from Kato's local $L^2$ smoothing theorem (see~\cite[Theorem~7.2]{DyatlovZworski2019}) that
\begin{equation}
\label{KatoLocalSmoothing}
\left\|\hspace{1pt}\zeta_0\,\phi\bigl(\tfrac{D-\lambda}\eta\bigr)\hspace{1pt}
e^{\hspace{.5pt}i\hspace{.5pt}tD^2}\hspace{-1pt}f\hspace{1pt}\right\|_{L^2_tL^2_x}
\lesssim_{\hspace{1pt}\epsilon}\lambda^{-\frac12+\hspace{1pt}\epsilon}\,\| f\|_{L^2}
\qquad\forall\,f\hspace{-1pt}\in\hspace{-1pt}L^2(X)\hspace{.5pt}.
\end{equation}
Moreover, by applying the Cauchy-Schwarz inequality to \eqref{SubordinationSchrodinger}
and by using \eqref{LrEstimateZ} together with \eqref{KatoLocalSmoothing}, we obtain
\begin{align*}
\bigl\|\hspace{1pt}\zeta_0\,\widetilde{\mathbb{P}}_{\hspace{-.5pt}\lambda,\eta}\,
f\hspace{1pt}\bigr\|_{L^2}
&=\bigl\|\hspace{1pt}\zeta_0\,\mathbb{P}_{\hspace{-.5pt}\lambda,\eta}\,
\phi\bigl(\tfrac{D-\lambda}\eta\bigr)\hspace{.5pt}f\hspace{1pt}\bigr\|_{L^2}
=\tfrac1{\sqrt{2\pi}}\left\|\hspace{1pt}\int_{-\infty}^{+\infty}\!
Z_{\lambda,\eta}(t)\,\zeta_0\,\phi\bigl(\tfrac{D-\lambda}\eta\bigr)\,
e^{\hspace{1pt}i\hspace{1pt}t\hspace{.5pt}D^2}\hspace{-1pt}f\,dt\,\right\|_{L^2_x}\\
&\lesssim\bigl\|Z_{\lambda,\eta}(t)\bigr\|_{L^2_t}\,
\bigl\|\hspace{1pt}\zeta_0\,\phi\bigl(\tfrac{D-\lambda}\eta\bigr)\,
e^{\hspace{1pt}i\hspace{1pt}t\hspace{.5pt}D^2}\hspace{-1pt}f\hspace{1pt}\bigr\|_{L^2_tL^2_x}
\lesssim\lambda^\epsilon\,\eta^{\frac12}\,\|f\|_{L^2}\hspace{1pt}.
\end{align*}
Assume next that $C_\epsilon\hspace{-1.5pt}>\!1$
and $1\!<\hspace{-1pt}\lambda\hspace{-1pt}<\hspace{-1pt}C_{\varepsilon}$.
Then, by using \eqref{LrEstimateZ} with $r\hspace{-1pt}=\!1$\hspace{.5pt},
we estimate
\begin{align*}
\bigl\|\hspace{1pt}\zeta_0\,\widetilde{\mathbb{P}}_{\hspace{-.5pt}\lambda,\eta}\,
f\hspace{1pt}\bigr\|_{L^2}
&=\tfrac1{\sqrt{2\pi}}\left\|\hspace{1pt}\int_{-\infty}^{+\infty}\!
Z_{\lambda,\eta}(t)\,\zeta_0\,\phi\bigl(\tfrac{D-\lambda}\eta\bigr)\,
e^{\hspace{1pt}i\hspace{1pt}t\hspace{.5pt}D^2}\hspace{-1pt}f\,dt\,\right\|_{L^2_x}\\
&\lesssim\bigl\|Z_{\lambda,\eta}(t)\bigr\|_{L^1_t}\,
\bigl\|\hspace{1pt}\zeta_0\,\phi\bigl(\tfrac{D-\lambda}\eta\bigr)\,
e^{\hspace{1pt}i\hspace{1pt}t\hspace{.5pt}D^2}\hspace{-1pt}f\hspace{1pt}\bigr\|_{L^{\infty}_tL^2_x}
\lesssim\|f\|_{L^2}\hspace{1pt},
\end{align*}

\noindent
which is $\lesssim\hspace{-1pt}\lambda^\epsilon\,\eta^{\frac12}\,\|f\|_{L^2}$
under the assumption $\lambda^M\eta\hspace{-1pt}>\!1$\hspace{.5pt}.
In both cases, by applying $\mathscr{P}_{\lambda,1}^{\hspace{1pt}3}$
and by using Sogge's estimate as stated in Remark \ref{RemarkSogge},
we conclude that
$$
\bigl\|\hspace{1pt}\mathscr{P}_{\lambda,1}^{\hspace{1pt}3}\hspace{1pt}\zeta_0\,
\widetilde{\mathbb{P}}_{\hspace{-.5pt}\lambda,\eta}\hspace{1pt}f\hspace{1pt}\bigr\|_{L^p} 
\lesssim\lambda^{\gamma(p)}\hspace{1pt}\bigl\|\hspace{1pt}\zeta_0\,
\widetilde{\mathbb{P}}_{\hspace{-.5pt}\lambda,\eta}\hspace{1pt}f\hspace{1pt}\bigr\|_{L^2}
\lesssim\lambda^{\gamma(p)+\hspace{.5pt}\epsilon}\,\eta^{\frac12}\,\|f\|_{L^2} \hspace{1pt}.
$$\vspace{-2mm}

\noindent\underline{The non-compact part.} 
There remains to estimate the $L^2\!\to\!L^p$ operator norm of
\begin{equation*}\label{EstimateOnCj}
\mathscr{P}_{\lambda,1}^{\hspace{1pt}3}\hspace{1pt}\zeta_j\,
\widetilde{\mathbb{P}}_{\hspace{-.5pt}\lambda,\eta}
=\mathscr{P}_{\lambda,1}^{\hspace{1pt}3}\hspace{1pt}\zeta_j\,
\mathbb{P}_{\hspace{-.5pt}\lambda,\eta}\,\phi\bigl(\tfrac{D-\lambda}\eta\bigr)
\end{equation*}
for $1\!\le\hspace{-1pt}j\hspace{-1pt}\le\hspace{-1pt}n$\hspace{.5pt}.
Given $f\hspace{-1pt}\in\hspace{-1pt}L^2(X)$, let \hspace{1pt}$u(t,\cdot\,)\hspace{-1pt}
=\hspace{-1pt}e^{\hspace{1pt}i\hspace{1pt}t\hspace{.5pt}\Delta}\hspace{1pt}\phi\bigl(\tfrac{D-\lambda}\eta\bigr)\hspace{.5pt}f$
be the solution to the Cauchy problem
$$
\begin{cases}
\,i\hspace{1pt}\partial_{\hspace{.5pt}t}\hspace{.5pt}u+\Delta\hspace{1pt}u=0\\
\,u\hspace{1pt}|_{\hspace{.5pt}t\hspace{.5pt}=\hspace{.5pt}0}=\phi\bigl(\tfrac{D-\lambda}\eta\bigr)f
\end{cases}
$$
on $X$.
On the hyperbolic cylinder $C_j$, whose half is the funnel $F_j$, the function
$$
u_j(t,x)=\begin{cases}
\,\zeta_j(x)\hspace{1pt}u(t,x)
&\text{if \,$x\hspace{-1pt}\in\hspace{-1pt}F_j$}\\
\hspace{25pt}0
&\text{if \,$x\hspace{-1pt}\in\hspace{-1pt}C_j\!\smallsetminus\hspace{-1pt}F_j$}\\
\end{cases}
$$
solves the Cauchy problem
$$
\begin{cases}
\,i\hspace{1pt}\partial_{\hspace{.5pt}t}\hspace{.5pt}u_j\hspace{-.5pt}
+\Delta_{\hspace{.5pt}C_j}\hspace{0pt}u_j
=[\Delta_{\hspace{.5pt}C_j},\zeta_j\hspace{.5pt}]\hspace{1pt}u\hspace{1pt},\\
\,u_j\hspace{.5pt}|_{\hspace{.5pt}t\hspace{.5pt}=0}
=\zeta_j\hspace{1pt}\phi\bigl(\tfrac{D-\lambda}\eta\bigr)\hspace{.5pt}f.
\end{cases}$$
Here we add the subscript $C_j$
to emphasize the fact that we are now working on $C_j$ rather than $X$.
By Duhamel's formula,
$$
u_j(t,\cdot\,)=e^{\hspace{1pt}i\hspace{1pt}t\hspace{.75pt}\Delta_{C_j}}\hspace{-.5pt}
\zeta_j\hspace{1pt}\phi\bigl(\tfrac{D-\lambda}\eta\bigr)\hspace{.5pt}f
-i\int_{\hspace{1pt}0}^{\hspace{1pt}t}e^{\hspace{1pt}i\hspace{1pt}(t-s)\hspace{.5pt}\Delta_{C_j}} [\Delta_{\hspace{.5pt}C_j},\zeta_j]\hspace{1pt}u(s,\cdot\,)\,ds\hspace{1pt}.
$$
Thus
\begin{equation}\label{eqzeta}\begin{aligned}
\zeta_j\,\widetilde{\mathbb{P}}_{\hspace{-.5pt}\lambda,\eta}\hspace{1pt}f
&=\zeta_j\,\mathbb{P}_{\hspace{-.5pt}\lambda,\eta}\,
\phi\bigl(\tfrac{D-\lambda}\eta\bigr)\hspace{.5pt}f
=\tfrac1{\sqrt{2\pi}}\int_{-\infty}^{+\infty}\!Z_{\lambda,\eta}(t)\,
\zeta_j\,e^{\hspace{1pt}i\hspace{1pt}t\hspace{.5pt}D^2}\hspace{-1pt}
\phi\bigl(\tfrac{D-\lambda}\eta\bigr)\hspace{.5pt}f\,dt\\
&=\tfrac1{\sqrt{2\pi}}\int_{-\infty}^{+\infty}\!e^{\hspace{1pt}i\frac t4}\hspace{1pt}
Z_{\lambda,\eta}(t)\,u_j(t\hspace{.5pt},\cdot\,)\,dt\\
&=\underbrace{\tfrac1{\sqrt{2\pi}}\int_{-\infty}^{+\infty}\!
e^{\hspace{1pt}i\frac t4}\hspace{1pt}Z_{\lambda,\eta}(t)\,
e^{\hspace{1pt}i\hspace{1pt}t\hspace{.75pt}\Delta_{C_j}}\hspace{-.5pt}
\zeta_j\hspace{1pt}\phi\bigl(\tfrac{D-\lambda}\eta\bigr)\hspace{.5pt}f\,dt
}_{\mathbb{P}_{\hspace{-1pt}\lambda,\eta}^{\hspace{1pt}C_j}\hspace{1pt}
\zeta_j\hspace{1pt}\phi(\frac{D-\lambda}\eta)\hspace{.5pt}f}\\
&-\tfrac i{\sqrt{2\pi}}\int_{-\infty}^{+\infty}\!e^{\hspace{1pt}i\frac t4}\hspace{1pt}
Z_{\lambda,\eta}(t)\hspace{1pt}\Bigl(\int_{\hspace{.5pt}0}^{\hspace{1pt}t}
e^{\hspace{1pt}i\hspace{1pt}(t-s)\hspace{.5pt}\Delta_{C_j}}
[\Delta_{\hspace{.5pt}C_j},\zeta_j]\hspace{1pt}u(s,\cdot\,)\,ds\Bigr)dt\,.
\end{aligned}\end{equation}
Let us comment about \eqref{eqzeta}.
The functions $\phi(\frac{D-\lambda}\eta)\hspace{.5pt}f$ and $u$ are initially defined on~$X$
and spectrally localized around \hspace{.5pt}$\lambda$\hspace{1pt}.
Once multiplied by $\zeta_j$,
all expressions are supported inside the funnel $F_j$
and may be considered as functions on the cylinder $C_j$,
which vanish outside~$F_j$.
Thus it makes sense to apply the Schr\"odinger group
\hspace{1pt}$e^{\hspace{1pt}i\,\cdot\,\Delta_{C_j}}$ to them.
Finally, the resulting expressions matter only inside $F_j$.

We finally apply the spectral projector $\mathscr{P}_{\hspace{-1pt}\lambda,1}^{\hspace{1pt}3}$
to \eqref{eqzeta}. Notice that $\mathscr{P}_{\hspace{-.5pt}\lambda,1}\!
=\hspace{-1pt}\mathscr{P}_{\hspace{-1pt}\lambda,1}^{\hspace{1pt}X}$
coincides with $\smash{\mathscr{P}_{\hspace{-1pt}\lambda,1}^{\hspace{.5pt}C_j}}$
on $\supp\zeta_j$,
as waves propagate at speed $\le\!1\hspace{.5pt}$.
Hence
\begin{equation}\begin{aligned}\label{finaleq}
\mathscr{P}_{\hspace{-1pt}\lambda,1}^{\hspace{1pt}3}\hspace{1pt}
\zeta_j\hspace{1pt}\widetilde{\mathbb{P}}_{\hspace{-.5pt}\lambda,\eta}\hspace{1pt}f
&=\phi\Bigl(\tfrac{D_{C_j}-\lambda}\eta\Bigr)\hspace{1pt}
\zeta_j\,\phi(\tfrac{D-\lambda}\eta)\hspace{.5pt}f\\
&-\tfrac i{\sqrt{2\pi}}\int_{-\infty}^{+\infty}\!e^{\hspace{1pt}i\frac t4}\hspace{1pt}
Z_{\lambda,\eta}(t)\hspace{1pt}\Bigl(\int_{\hspace{.5pt}0}^{\hspace{1pt}t}\!
\mathscr{P}_{\hspace{-1pt}\lambda,1}^{\hspace{1pt}3}\hspace{1pt}
e^{\hspace{1pt}i\hspace{1pt}(t-s)\hspace{.5pt}\Delta_{C_j}}
[\Delta_{\hspace{.5pt}C_j},\zeta_j]\hspace{1pt}u(s,\cdot\,)\,ds\Bigr)dt\,.
\end{aligned}\end{equation}
We know from Section \ref{SurfacesSmallDelta} that
the first term on the right-hand side of \eqref{finaleq} enjoys optimal bounds on $C_j$\,:
$$
\Bigl\|\hspace{1pt}\phi\Bigl(\tfrac{D_{C_j}-\lambda}\eta\Bigr)\hspace{1pt}
\zeta_j\,\phi(\tfrac{D-\lambda}\eta)\hspace{.5pt}f\hspace{1pt}\Bigr\|_{L^p}
\lesssim\lambda^{\gamma(p)}\hspace{1pt}\eta^{\frac12}\hspace{1pt}\underbrace{
\bigl\|\hspace{1pt}\zeta_j\,\phi(\tfrac{D-\lambda}\eta)\hspace{.5pt}f\hspace{1pt}\bigr\|_{L^2}
}_{\lesssim\,\|f\|_{L^2}}\hspace{1pt}.
$$
As for the second term on the right-hand side of \eqref{finaleq},
its control will require the following three lemmas.

\begin{lemma}[Duhamel Lemma]\label{DuhamelLemma}
Let \,$\zeta\!\in\hspace{-1pt}\mathcal{C}_{\hspace{1pt}c}^{\hspace{.5pt}\infty}(C_j)$\hspace{.5pt}.
Then the following inequality holds on \,$\R\hspace{-1pt}\times\hspace{-1pt}C_j$,
for every \,$q\hspace{-1pt}>\hspace{-1pt}2$ and \,$\epsilon\hspace{-1pt}>\hspace{-1pt}0:$
\begin{equation}\label{EstimateLemma1}
\Bigl\|\hspace{1pt}Z_{\lambda,\eta}(t)\!\int_{\hspace{.5pt}0}^{\hspace{1pt}t}\!
\mathscr{P}_{\hspace{-1pt}\lambda,1}^{\hspace{1pt}2}\hspace{1pt}
e^{\hspace{1pt}i\hspace{1pt}(t-s)\hspace{.5pt}\Delta_{C_j}}
\zeta\hspace{1pt}F(s,\cdot\,)\,ds\hspace{1pt}\Bigr\|_{L^1_tL^q_x}
\lesssim_{\hspace{1pt}q,\epsilon}\lambda^{-\frac1q+\hspace{1pt}\epsilon}\,
\eta^{\frac12-\hspace{1pt}\epsilon}\,\|F(t,x)\|_{L^2_tL^2_x}\hspace{1pt}.
\end{equation}
\end{lemma}

\begin{proof}
Let $r\hspace{-1pt}>\hspace{-1pt}2$
and $2\hspace{-1pt}<\hspace{-1pt}\tilde{r}\hspace{-1pt}<\hspace{-1pt}6$\hspace{.5pt}.
Firstly, by using H\"older's inequality and \eqref{LrEstimateZ},
the left hand side of \eqref{EstimateLemma1} is bounded above by
\begin{equation*}
\underbrace{\|Z_{\lambda,\eta}(t)\|_{L_t^{r^{\hspace{.5pt}\prime}}}
}_{\lesssim\,(\lambda\hspace{.5pt}\eta)^{1/r}}\,
\Bigl\|\hspace{.5pt}\int_{\hspace{.5pt}0}^{\hspace{1pt}t}\!
\mathscr{P}_{\hspace{-1pt}\lambda,1}^{\hspace{1pt}2}\hspace{1pt}
e^{\hspace{1pt}i\hspace{1pt}(t-s)\hspace{.5pt}\Delta_{C_j}}
\zeta\hspace{1pt}F(s,\cdot\,)\,ds\hspace{1pt}\Bigr\|_{L^r_tL^q_x}\,.
\end{equation*}
Secondly, by applying Proposition \ref{StrichartzX},
\begin{align*}
\Bigl\|\hspace{.5pt}\int_{-\infty}^{+\infty}\!
\mathscr{P}_{\hspace{-1pt}\lambda,1}^{\hspace{1pt}2}\hspace{1pt}
e^{\hspace{1pt}i\hspace{1pt}(t-s)\hspace{.5pt}\Delta_{C_j}}
\zeta\hspace{1pt}F(s,\cdot\,)\,ds\hspace{1pt}\Bigr\|_{L^r_tL^q_x}
\lesssim\lambda^{\frac12-\frac1q-\frac1r}\,
\Bigl\|\hspace{.5pt}\mathscr{P}_{\hspace{-.5pt}\lambda,1}\!\int_{-\infty}^{\hspace{.5pt}+\infty}\!
e^{-\hspace{1pt}i\hspace{1pt}s\hspace{1pt}\Delta_{C_j}}
\zeta\hspace{1pt}F(s,\cdot\,)\,ds\hspace{1pt}\Bigr\|_{L^2_x}
\end{align*}
Thirdly, let us estimate separately
\begin{equation*}
I=\Bigl\|\hspace{1pt}\mathbf{1}_{[\hspace{.5pt}0,\frac\lambda2]}(D_{C_j}\hspace{-.5pt})\hspace{1pt}
\mathscr{P}_{\hspace{-1pt}\lambda,1}\!\int_{-\infty}^{+\infty}\!e^{-i\hspace{1pt}s\hspace{.5pt}\Delta_{C_j}}
\zeta\hspace{1pt}F(s,\cdot\,)\,ds\hspace{1pt}\Bigr\|_{L_x^2}
\end{equation*}
and
\begin{equation*}
I\!I=\Bigl\|\hspace{1pt}\mathbf{1}_{(\frac\lambda2,+\infty)}(D_{C_j}\hspace{-.5pt})\hspace{1pt}
\mathscr{P}_{\hspace{-1pt}\lambda,1}\!\int_{-\infty}^{+\infty}\!e^{-i\hspace{1pt}s)\hspace{.5pt}\Delta_{C_j}}
\zeta\hspace{1pt}F(s,\cdot\,)\,ds\hspace{1pt}\Bigr\|_{L^1_tL^q_x}\,.
\end{equation*}
On the one hand, by using the elementary estimate
\begin{equation*}
\chi(\,\cdot-\hspace{-1pt}\lambda)+\chi(\,\cdot+\hspace{-1pt}\lambda)
=\text{O}\bigl(\lambda^{-N}\bigr)
\quad\text{on \,}[\hspace{.5pt}0,\hspace{-.5pt}\tfrac\lambda2\hspace{.5pt}]
\end{equation*}
and Kato's local $L^2$ smoothing theorem on $C_j$ (see Proposition \ref{KatosmoothingX}),
we get
\begin{align*}
I&\lesssim\lambda^{-N}\,
\Bigl\|\hspace{1pt}\mathbf{1}_{[\hspace{.5pt}0,\frac\lambda2]}(D_{C_j}\hspace{-.5pt})\!
\int_{-\infty}^{+\infty}\!e^{-i\hspace{1pt}s\hspace{.5pt}\Delta_{C_j}}
\zeta\hspace{1pt}F(s,\cdot\,)\,ds\hspace{1pt}\Bigr\|_{L_x^2}\\
&\lesssim\lambda^{-N}\,\|F(t,x)\|_{L^2_tL^2_x}\,.
\end{align*}
On the other hand, by using the $L^{\tilde{r}}$ smoothing estimate~\eqref{smo2}, we obtain
\begin{align*}
I\!I&\lesssim\lambda^{\gamma(\tilde{r})-\frac12}\hspace{1pt}
\Bigl\|\hspace{1pt}\mathbf{1}_{[\frac\lambda2,+\infty)}(D_{C_j}\hspace{-.5pt})\hspace{1pt}
\mathscr{P}_{\hspace{-.5pt}\lambda,1}\hspace{1pt}D_{C_j}^{\frac12-\gamma(\tilde{r})}\!
\int_{-\infty}^{\hspace{.5pt}+\infty}\!e^{-\hspace{1pt}i\hspace{1pt}s\hspace{1pt}\Delta_{C_j}}
\zeta\hspace{1pt}F(s,\cdot\,)\,ds\hspace{1pt}\Bigr\|_{L^2_x}\\
&\lesssim\lambda^{-\frac14-\frac1{2\tilde{r}}}\,
\Bigl\|\hspace{1pt}D_{C_j}^{\frac12-\gamma(\tilde{r})}\!
\int_{-\infty}^{\hspace{.5pt}+\infty}\!e^{-\hspace{1pt}i\hspace{1pt}s\hspace{1pt}\Delta_{C_j}}
\zeta\hspace{1pt}F(s,\cdot\,)\,ds\hspace{1pt}\Bigr\|_{L^2_x}\\
&\lesssim\lambda^{-\frac14-\frac1{2\tilde{r}}}\,
\|\hspace{1pt}\zeta(x)\hspace{.5pt}F(t,x)\|_{L^{\tilde{r}^{\hspace{.5pt}\prime}}_x\!L^2_t}
\lesssim\lambda^{-\frac14-\frac1{2\tilde{r}}}\,\|F(t,x)\|_{L^2_tL^2_x}\,.
\end{align*}
Hence
\begin{equation*}
\Bigl\|\hspace{.5pt}\mathscr{P}_{\hspace{-.5pt}\lambda,1}\!\int_{-\infty}^{\hspace{.5pt}+\infty}\!
e^{-\hspace{1pt}i\hspace{1pt}s\hspace{1pt}\Delta_{C_j}}
\zeta\hspace{1pt}F(s,\cdot\,)\,ds\hspace{1pt}\Bigr\|_{L^2_x}
\le I\hspace{-1pt}+\hspace{-1pt}I\!I
\lesssim\lambda^{-\frac14-\frac1{2\tilde{r}}}\,\|F(t,x)\|_{L^2_tL^2_x}
\end{equation*}
and consequently
\begin{equation}\label{UntruncatedInequality}
\Bigl\|\hspace{.5pt}\int_{-\infty}^{+\infty}\!
\mathscr{P}_{\hspace{-1pt}\lambda,1}^{\hspace{1pt}2}\hspace{1pt}
e^{\hspace{1pt}i\hspace{1pt}(t-s)\hspace{.5pt}\Delta_{C_j}}
\zeta\hspace{1pt}F(s,\cdot\,)\,ds\hspace{1pt}\Bigr\|_{L^r_tL^q_x}
\lesssim\lambda^{\frac14-\frac1q-\frac1r-\frac1{2\tilde{r}}}\,.
\end{equation}
Fourthly, the Christ-Kiselev lemma allows us to replace $\int_{-\infty}^{+\infty}$
by the truncated integral $\int_{\hspace{1pt}0}^{\hspace{1pt}t}$ in \eqref{UntruncatedInequality}.
In conclusion,
$$
\Bigl\|\hspace{1pt}Z_{\lambda,\eta}(t)\!\int_{\hspace{.5pt}0}^{\hspace{1pt}t}\!
\mathscr{P}_{\hspace{-1pt}\lambda,1}^{\hspace{1pt}2}\hspace{1pt}
e^{\hspace{1pt}i\hspace{1pt}(t-s)\hspace{.5pt}\Delta_{C_j}}
\zeta\hspace{1pt}F(s,\cdot\,)\,ds\hspace{1pt}\Bigr\|_{L^1_tL^q_x}
\lesssim_{\hspace{1pt}q,r,\tilde{r}}\lambda^{\frac14-\frac1q-\frac1{2\tilde{r}}}\,
\eta^{\frac1r}\,\|F(t,x)\|_{L^2_tL^2_x}
$$
and \eqref{EstimateLemma1} is obtained by taking $r$ and $\tilde{r}$ sufficiently close to $2$.
\end{proof}


\begin{lemma}[Commutator Lemma]\label{CommutatorLemma}
For any \,$1\!\le\hspace{-1pt}j\hspace{-1pt}\le\hspace{-1pt}n$\hspace{.5pt},
$\epsilon\hspace{-1pt}>\hspace{-1pt}0$ and \hspace{1pt}$N\hspace{-1pt}>\hspace{-1pt}0$\hspace{.5pt},
there exist functions
\,$\zeta,\widetilde{\zeta}\!\in\hspace{-1pt}\mathcal{C}_{\hspace{.5pt}c}^{\hspace{.5pt}\infty}(C_j)$
and bounded operators \,$A\hspace{.5pt},R_{\hspace{.5pt}\lambda}$ on \hspace{1pt}$L^2(C_j)$
such that
$$
\mathscr{P}_{\lambda,1}\hspace{1pt}[\Delta_{\hspace{.5pt}C_j},\zeta_j]\hspace{1pt}u
=\lambda^{1+\hspace{.5pt}\epsilon}\hspace{1pt}{\zeta}\hspace{1pt}
A\hspace{.5pt}{(\widetilde{\zeta}u)}
+R_{\hspace{.5pt}\lambda}\hspace{.5pt}{(\widetilde{\zeta}u)}
$$
and \,$\|R_{\hspace{.5pt}\lambda}\|_{L^2\to L^2}\hspace{-1pt}=\hspace{-1pt}\text{\rm O}(\lambda^{-N})$.
\end{lemma}

\begin{proof}
First, given a slightly enlarged cut-off $\widetilde{\zeta_j}$ (meaning it equals $1$ on the support of $\zeta_j$) we can write that 
\begin{align*}
[\Delta_{\hspace{.5pt}C_j},\zeta_j] = \widetilde{\zeta_j} [\Delta_{\hspace{.5pt}C_j},\zeta_j],
\end{align*}
as can be seen by working in local coordinates.

Next, since the operator $\mathscr{P}_{\hspace{-.5pt}\lambda,1}$ has finite range by finite speed of propagation of the wave semi-group, there exists an enlarged cut-off $\widetilde{\widetilde{\zeta_j}}$ of $\widetilde{\zeta_j}$ such that 
\begin{align*}
\mathscr{P}_{\hspace{-.5pt}\lambda,1} [\Delta_{\hspace{.5pt}C_j},\zeta_j] = \mathscr{P}_{\hspace{-.5pt}\lambda,1} \widetilde{\zeta_j} [\Delta_{\hspace{.5pt}C_j},\zeta_j]  = \widetilde{\widetilde{\zeta_j}} \mathscr{P}_{\hspace{-.5pt}\lambda,1} \widetilde{\zeta_j} [\Delta_{\hspace{.5pt}C_j},\zeta_j]= \widetilde{\widetilde{\zeta_j}} \mathscr{P}_{\hspace{-.5pt}\lambda,1} [\Delta_{\hspace{.5pt}C_j},\zeta_j].
\end{align*}

Next by the spectral localization of $u$, we can write, for
$\varphi\hspace{-1pt}\in\hspace{-1pt}\mathcal{C}_{\hspace{.5pt}c}^{\hspace{.5pt}\infty}$
equal to $1$ in a neighborhood of the origin,
\begin{align*}
\mathscr{P}_{\hspace{-.5pt}\lambda,1}
(\Delta_{\hspace{.5pt}C_j}\zeta_j-\zeta_j\Delta_{\hspace{.5pt}C_j})
\hspace{1pt}u 
&=\mathscr{P}_{\hspace{-.5pt}\lambda,1}\bigl[\varphi\bigl(\lambda^{-2-\delta}\Delta_{\hspace{.5pt}C_j}\bigr)\Delta_{\hspace{.5pt}C_j}\zeta_j-\zeta_j\varphi\bigl( {\lambda^{-2-\delta}}\Delta_{\hspace{.5pt}C_j}\bigr)\Delta_{\hspace{.5pt}C_j}\bigr]u\\
&+\mathscr{P}_{\hspace{-.5pt}\lambda,1}\bigl[1-\varphi\bigl( {\lambda^{-2-\delta}}\Delta_{\hspace{.5pt}C_j}\bigr)\bigr]\Delta_{\hspace{.5pt}C_j}\zeta_ju.
\end{align*}
To the first summand on the right-hand side, we can apply the commutator lemma in \cite{Wang2019}, which gives the desired decomposition. Turning to the second summand on the right-hand side, we observe that
$$
 \mathscr{P}_{\lambda,1}\bigl[1-\varphi\bigl( {\lambda^{-2-\delta}}\Delta_{\hspace{.5pt}C_j}\bigr)\bigr]\Delta_{\hspace{.5pt}C_j}=\text{O}(\lambda^{-N})
$$
as an operator on $L^2$ for any $N$. This follows by considering the symbol of this operator, and using the rapid decay of $\chi$, from which $\mathscr{P}_{\lambda,1}$ is defined.
\end{proof}

\begin{lemma}\label{Lemma3}
Let \,$\lambda\hspace{-1pt}>\!1$, $p\hspace{-1pt}>\hspace{-1pt}2$
and \,$\epsilon\hspace{-1pt}>\hspace{-1pt}0$\hspace{1pt}.
Then
$$
\|\mathscr{P}_{\lambda,1}\|_{L^q\to L^p}\lesssim\lambda^{\gamma(p)+\hspace{.5pt}\epsilon}
$$
on \,$C_j$,
for \,$q\hspace{-1pt}>\hspace{-1pt}2$ sufficiently close to \,$2$\hspace{1pt}.
\end{lemma}

\begin{remark}
This result holds more generally in the setting of Section \ref{SurfacesSmallDelta}.
\end{remark}

\begin{proof}
Let \hspace{1pt}$\tilde{p}\hspace{-1pt}>\hspace{-1pt}p\hspace{-1pt}>\hspace{-1pt}
\tilde{q}\hspace{-1pt}>\hspace{-1pt}q\hspace{-1pt}>\hspace{-1pt}2$ \hspace{1pt}with
$\tilde{p}$ close to $p$ and $\tilde{q}\hspace{.5pt},\hspace{-.5pt}q$ close to $2$\hspace{.5pt}.
On the one hand,
\begin{equation}\label{L2Lptilde}
\|\mathscr{P}_{\lambda,1}\|_{L^2\to L^{\tilde{p}}}
\lesssim\lambda^{\gamma(\tilde{p})}
\end{equation}
according to Sogge's upper bound
(see Remark \ref{RemarkSogge} in Appendix \ref{AppendixSogge}).
On the other hand,
\begin{equation}\label{LqtildeLqtilde}
\|\mathscr{P}_{\lambda,1}\|_{L^{\tilde{q}}\to L^{\tilde{q}}}
\lesssim\lambda^{4\hspace{1pt}(\frac12-\frac1{\tilde{q}})}
\end{equation}
according to the multiplier theorem in \cite{Taylor1989} or \cite{FotiadisMarias2010}.
We conclude by interpolation between \eqref{L2Lptilde} and \eqref{LqtildeLqtilde}.
\end{proof}

We can now turn to estimating the second term on the right-hand side of~\eqref{finaleq}.
By the Lemma \ref{CommutatorLemma},
\begin{align*}
&\int_{-\infty}^{+\infty}\!e^{\hspace{1pt}i\frac t4}\hspace{1pt}
Z_{\lambda,\eta}(t)\hspace{1pt}\Bigl(\int_{\hspace{.5pt}0}^{\hspace{1pt}t}\!
\mathscr{P}_{\hspace{-1pt}\lambda,1}^{\hspace{1pt}3}\hspace{1pt}
e^{\hspace{1pt}i\hspace{1pt}(t-s)\hspace{.5pt}\Delta_{C_j}}
[\Delta_{\hspace{.5pt}C_j},\zeta_j]\hspace{1pt}u(s,\cdot\,)\,ds\Bigr)dt\\
&\qquad=\int_{-\infty}^{+\infty}\!e^{\hspace{1pt}i\frac t4}\hspace{1pt}
Z_{\lambda,\eta}(t)\hspace{1pt}\Bigl(\int_{\hspace{.5pt}0}^{\hspace{1pt}t}\!
\mathscr{P}_{\hspace{-1pt}\lambda,1}^{\hspace{1pt}3}\hspace{1pt}
e^{\hspace{1pt}i\hspace{1pt}(t-s)\hspace{.5pt}\Delta_{C_j}}\bigl[\hspace{.5pt}
\lambda^{1+\hspace{.5pt}\epsilon}\hspace{1pt}\zeta\hspace{1pt}
A\,\widetilde{\zeta}\hspace{1pt}u(s,\cdot\,)\hspace{-1pt}
+\hspace{-1pt}R_{\hspace{.5pt}\lambda}\hspace{.5pt}u(s,\cdot\,)
\hspace{.5pt}\bigr]\hspace{1pt}ds\Bigr)dt\\
&\qquad=I+I\hspace{-1pt}I.
\end{align*}
In order to estimate the localized contribution \hspace{1pt}$I$,
we apply successively Lemma \ref{Lemma3}, Lemma \ref{DuhamelLemma} with 
$\tfrac12\hspace{-1pt}-\hspace{-1pt}\epsilon\hspace{-1pt}
<\hspace{-1pt}\tfrac1q\hspace{-1pt}<\hspace{-1pt}\frac12$\hspace{1pt},
Lemma \ref{CommutatorLemma} and Proposition \ref{KatosmoothingX}
to obtaino
\begin{equation}\label{EstimateI}\begin{aligned}
\|I\|_{L^p}
&\lesssim\lambda^{1+\hspace{.5pt}\epsilon}\hspace{1pt}
\bigl\|\hspace{.5pt}\mathscr{P}_{\hspace{-.5pt}\lambda,1}\bigr\|_{L^q\to L^p}\hspace{1pt}
\Bigl\|\hspace{1pt}Z_{\lambda,\eta}(t)\!\int_{\hspace{.5pt}0}^{\hspace{1pt}t}\!
\mathscr{P}_{\hspace{-1pt}\lambda,1}^{\hspace{1pt}2}\hspace{1pt}
e^{\hspace{1pt}i\hspace{1pt}(t-s)\hspace{.5pt}\Delta_{C_j}}\hspace{.5pt}\zeta\hspace{1pt}
A\,\widetilde{\zeta}\hspace{.5pt}u(s,\cdot\,)\,ds\hspace{1pt}\Bigr\|_{L^1_tL^q_x}\\
&\lesssim\lambda^{\gamma(p)+\frac12+\hspace{.5pt}4\hspace{.5pt}\epsilon}\hspace{1pt}
\eta^{\frac12-\epsilon}\hspace{1pt}
\|A\,\widetilde{\zeta}\hspace{.5pt}u(t,\,\cdot)\|_{L^2_tL^2_x} \lesssim \lambda^{\gamma(p) + 5 \varepsilon} \eta^{\frac12 - \varepsilon} \hspace{1pt}  \Vert f \Vert_{L^2}.
\end{aligned}\end{equation}

\smallskip

In order to estimate the contribution of the remainder term $I\hspace{-1pt}I$,
we use Sogge's upper bound, \eqref{PointwiseEstimateZ} with $L\hspace{-1pt}\ge\hspace{-1pt}3$
and Lemma \ref{CommutatorLemma} to estimate
\begin{equation}\label{EstimateII}\begin{aligned}
\|I\hspace{-1pt}I\|_{L^p}
&\lesssim\bigl\|\hspace{.5pt}\mathscr{P}_{\hspace{-.5pt}\lambda,1}^{\hspace{1pt}3}\bigr\|_{L^2\to L^p}
\int_{-\infty}^{+\infty}\!|Z_{\lambda,\eta}(t)|\,\Bigl\|\hspace{.5pt}\int_{\hspace{.5pt}0}^{\hspace{1pt}t}\!
e^{\hspace{1pt}i\hspace{1pt}(t-s)\hspace{.5pt}\Delta_{C_j}}R_\lambda\hspace{.5pt}u(s,\cdot\,)
\,ds\,\Bigr\|_{L^2_x}dt\\
&\lesssim\lambda^{\gamma(p)}\!\int_{-\infty}^{+\infty}\!
\lambda\hspace{1pt}\eta\,(1\!+\!\lambda\hspace{1pt}\eta\hspace{1pt}|t|\hspace{.5pt})^{-L}\,
\Bigl(\int_{-|t|}^{\hspace{1pt}|t|}\|R_\lambda\hspace{.5pt}u(s,\cdot\,)\|_{L^2_x}\hspace{1pt}ds\Bigr)\,dt\\
&\lesssim\lambda^{\gamma(p)\hspace{.5pt}-N}\hspace{-1pt}\underbrace{
\Bigl(\lambda\hspace{.5pt}\eta\int_{-\infty}^{+\infty}\hspace{-1pt}
(1\!+\!\lambda\hspace{1pt}\eta\hspace{1pt}|t|\hspace{.5pt})^{-L}\hspace{1pt}|t|\,dt\Bigr)
}_{\sim\,(\lambda\hspace{.5pt}\eta\hspace{.5pt})^{-1}}\|f\|_{L^2_x}\\
&\lesssim \lambda^{\gamma(p)\hspace{.5pt}-N-1}\hspace{1pt}\eta^{-1}\hspace{1pt}\|f\|_{L^2_x}\hspace{1pt}.
\end{aligned}\end{equation}

This is $\lesssim \lambda^{\gamma(p)} \hspace{1pt} \eta^{\frac12} \hspace{1pt} \Vert f \Vert_{L^2_x}$ under the assumption $\lambda^M \eta >1$ since $N$ can be chosen arbitrarily large.

\end{proof}

\appendix

\section{Sogge's theorem on complete manifolds with bounded geometry}\label{AppendixSogge}

We show in this appendix how Sogge's bound \eqref{soggestatement} on the operator norm of spectral projectors can be extended from compact manifolds to complete manifolds with bounded geometry, which means here
\begin{itemize}
\item
uniform bound on derivatives of any order of the metric,
\item
injectivity radius bounded from below.
\end{itemize}

These conditions are satisfied for the quotients of the hyperbolic plane considered in Sections \ref{SurfacesSmallDelta} and \ref{SurfacesLargeDelta}

Starting from a complete manifold $X$ with bounded geometry, the strategy of the proof will be to reduce matters to the setting of a compact manifold, where the proof in the textbook by Sogge~\cite{Sogge2017} applies. We restrict to the case of dimension $2$, but all arguments extend straightforwardly to higher dimensions.

\bigskip

\noindent\underline{The upper bound\,: reduction to a finite range operator.}
We aim at proving the bound
$$
\|P^{\hspace{1pt}\prime}_{\lambda,\eta_0}\|_{L^2\to L^p}\lesssim\lambda^{\gamma(p)},
$$
for any fixed $\eta_0>0$ and $\lambda>1$.
Firstly, note that it suffices to prove this bound for the operator
\begin{equation}\label{Plambda}
\mathscr{P}^{\hspace{1pt}\prime}_{\lambda,1}
=\chi(\sqrt{\Delta}\hspace{-1pt}-\hspace{-1pt}\lambda)
+\chi(\sqrt{\Delta}\hspace{-1pt}+\hspace{-1pt}\lambda)\hspace{1pt},
\end{equation}
where $\chi$ is an even Schwartz function such that
$\supp\widehat{\chi}\subset[-\epsilon_0,-\frac{\epsilon_0}2]\cup[\frac{\epsilon_0}2,\epsilon_0]$
(where $5\hspace{1pt}\epsilon_0$ is smaller than the injectivity radius of $X$)
and $\chi(0) =1$.
Indeed, there exists $0<\eta_1<1$ such that $\mathbf{1}_{[-\eta_1,\eta_1]}\le2\hspace{1pt}|\chi|$,
hence
$$
\|P^{\hspace{1pt}\prime}_{\lambda,{\eta}}\|_{L^2\to L^p}
\le\|P^{\hspace{1pt}\prime}_{\lambda,{\eta_1}}\|_{L^2\to L^p}
\lesssim\|\mathscr{P}^{\hspace{1pt}\prime}_{\lambda,1}\|_{L^2 \to L^p}
\quad\forall\,0<\eta\le\eta_1\hspace{1pt}.
$$
If $\eta_0>\eta_1$, we split up the interval $[-\eta_0,\eta_0]$ into
$N\sim\tfrac1{\eta_1}$ disjoint subintervals $I_j$ of length $\le\eta_1$,
we estimate
$$
\|P^{\hspace{1pt}\prime}_{\lambda,\eta_0}\|_{L^2\to L^p}
=\sum\nolimits_{j=1}^N\|P^{\hspace{1pt}\prime}_{\lambda+I_j}\|_{L^2\to L^p}
\lesssim\|\mathscr{P}^{\hspace{1pt}\prime}_{\lambda,1}\|_{L^2 \to L^p}\,.
$$
Secondly, the operator $\mathscr{P}^{\hspace{1pt}\prime}_{\lambda,1}$ can be written under the form
$$
\mathscr{P}^{\hspace{1pt}\prime}_{\lambda,1}=\sqrt{\hspace{-.5pt}\tfrac2\pi}\int_{-\infty}^{+\infty}
\widehat{\chi}(t)\hspace{.5pt}\cos\hspace{.5pt}(t\hspace{.5pt}\sqrt{\hspace{-.5pt}\Delta}\hspace{1pt})\hspace{1pt}\cos(t\lambda)\,dt\,.
$$
Therefore, by finite propagation speed for the wave equation,
its kernel $K_{\lambda}(x,x')$ vanishes if ${d(x,x')}>\epsilon_0$\hspace{1pt}.

\bigskip

\noindent\underline{Proof of the upper bound.}
Consider smooth bump functions \hspace{.5pt}$\phi_i\hspace{-1pt}:\!X\!\to\hspace{-1pt}[0,1]$
and \hspace{.5pt}$\psi_i\hspace{-1pt}:\!X\!\to\hspace{-1pt}[0,1]$ such that
\begin{itemize}
\item
$\sum_{i\in I}\phi_i^2=1$ on $X$,
\item
each $\phi_i$ is supported in a ball $B(x_i,\epsilon_0)$,
\item
{$\psi_i=1$ on $B(x_i,2\epsilon_0)$ and $\supp\psi_i\subset B(x_i,3\epsilon_0)$,}
\item
{the balls $B(x_i,3\epsilon_0)$} have uniformly bounded overlap,
hence $\sum_i \psi_i\lesssim 1$ on $X$.
\end{itemize}
Notice that $\mathscr{P}^{\hspace{1pt}\prime}_{\lambda,1}\phi_i
=\psi_i\hspace{1pt}\mathscr{P}^{\hspace{1pt}\prime}_{\lambda,1}\phi_i$
by finite propagation speed.
Now that everything has been localized around $x_i$,
we can consider a compact manifold $K_i$ agreeing with $X$ on $B(x_i,4\epsilon_0)$
and deduce from Sogge's result that
$$
\|\hspace{1pt}\psi_i\hspace{1pt}\mathscr{P}^{\hspace{1pt}\prime}_{\lambda,1}\hspace{1pt}
\phi_i\hspace{1pt}\|_{L^2\to L^p}\lesssim\lambda^{\gamma(p)}
$$
uniformly in $i\!\in\!I$.
By using successively the locality of $\mathscr{P}^{\hspace{1pt}\prime}_{\lambda,1}$, the bounded overlap,
the boundedness of $\mathscr{P}^{\hspace{1pt}\prime}_{\lambda,1}$,
the inclusion $\ell^2(I)\subset\ell^p(I)$, and once again the bounded overlap, we obtain
\begin{align*}
\bigl\|\hspace{1pt}\mathscr{P}^{\hspace{1pt}\prime}_{\lambda,1}\hspace{1pt}f\hspace{1pt}\bigr\|_{L^p}
&=\Bigl\|\,\sum_{i\in I}\psi_i\hspace{1pt}\mathscr{P}^{\hspace{1pt}\prime}_{\lambda,1}
\bigl(\phi_i^2f\bigr)\Bigr\|_{L^p} 
\sim\Bigl(\hspace{1pt}\sum_{i\in I}\,\bigl\|\hspace{1pt}\psi_i\hspace{1pt}
\mathscr{P}^{\hspace{1pt}\prime}_{\lambda,1}\hspace{1pt}\phi_i\hspace{1pt}\bigl(\phi_i\hspace{1pt}f\bigr)
\bigr\|_{L^p}^p\Bigr)^{\hspace{-1pt}1/p} \\
&\lesssim\lambda^{\gamma(p)}\Bigl(\hspace{1pt}\sum_{i\in I}\,
\bigl\|\phi_if\bigr\|_{L^2}^p\Bigr)^{\hspace{-1pt}1/p}
\leq\lambda^{\gamma(p)}\Bigl(\hspace{1pt}\sum_{i\in I}\,
\bigl\|\phi_if\bigr\|_{L^2}^2\Bigr)^{\hspace{-1pt}1/2}
\lesssim \lambda^{\gamma(p)}\hspace{1pt}\|f\|_{L^2}\hspace{1pt},
\end{align*}
which is the desired upper bound.

\begin{remark}\label{RemarkSogge}
Notice that this upper bound holds true for
\hspace{1pt}$\mathscr{P}^{\hspace{1pt}\prime}_{\lambda,1}\!
=\chi(\sqrt{\Delta}\hspace{-1pt}-\hspace{-1pt}\lambda)
+\chi(\sqrt{\Delta}\hspace{-1pt}+\hspace{-1pt}\lambda)$\hspace{1pt},
where $\chi$ is any Schwartz function.
Indeed,
$$
\bigl\|\mathscr{P}^{\hspace{1pt}\prime}_{\lambda,1}\bigr\|_{L^2\to L^p}
=\Bigl\|\hspace{1pt}\mathscr{P}^{\hspace{1pt}\prime}_{\lambda,1}
\hspace{-6pt}\sum_{n\in2\mathbb{N}+1}\hspace{-6pt}
\mathbf{1}_{[n-1,n+1)}(\sqrt{\Delta}\hspace{1pt})\hspace{1pt}\Bigr\|_{L^2\to L^p}
\le\sum_n\,|w_{\lambda,n}|\,\bigl\|P^{\hspace{1pt}\prime}_{n,1}\bigr\|_{L^2\to L^p}\hspace{1pt},
$$
where
\begin{equation*}
w_{{\lambda},n}=\hspace{-2pt}\sup_{{\xi}\in[n-1,n+1]}\hspace{-.5pt}
|\hspace{.5pt}\chi(\xi\hspace{-1pt}-\hspace{-1pt}\lambda)\hspace{-1pt}
+\hspace{-1pt}\chi(\xi\hspace{-1pt}+\hspace{-1pt}\lambda)\hspace{.5pt}|
\end{equation*}
is $\text{O}\bigl((1\!+\hspace{-1pt}|\lambda\hspace{-1pt}-\hspace{-1pt}n|)^{-N}\bigr)$.
Hence
$$
\bigl\|\mathscr{P}^{\hspace{1pt}\prime}_{\lambda,1}\bigr\|_{L^2\to L^p}\lesssim\sum_n\hspace{1pt}
(1\!+\hspace{-1pt}|\lambda\hspace{-1pt}-\hspace{-1pt}n|)^{-\gamma(p)-2}\hspace{1pt}n^{\gamma(p)}
\lesssim\lambda^{\gamma(p)}\hspace{1pt}.
$$
\end{remark}
\noindent\underline{The lower bound\,: reduction to a finite range operator.}
We aim at proving that there exists a constant ${\eta}_0>0$ such that,
for {every $\lambda>1$},
$$
\|P^{\hspace{1pt}\prime}_{\lambda,{\eta_0}}\|_{L^2\to L^p}\gtrsim
\lambda^{\gamma(p)}.
$$
Defining $\mathscr{P}^{\hspace{1pt}\prime}_{\lambda,1}$ as above, it suffices to prove that
$$
\|\mathscr{P}^{\hspace{1pt}\prime}_{\lambda,1}\|_{L^2\to L^p}\gtrsim\lambda^{\gamma(p)}
\quad\text{or equivalently}\quad
\|\mathscr{P}^{\hspace{1pt}\prime}_{\lambda,1}\|_{L^{p^{\hspace{.5pt}\prime}}\!\to L^2}
\gtrsim\lambda^{\gamma(p)}
$$
for all $\lambda >1$. Indeed, if $f$ satisfies
$$
{\|f\|_{L^{p^{\hspace{.5pt}\prime}}}\hspace{-1pt}=1\quad\text{and}\quad}
\|\mathscr{P}^{\hspace{1pt}\prime}_{\lambda,1}f\|_{L^2}\gtrsim
\lambda^{\gamma(p)},
$$
then
\begin{align*}
\lambda^{\gamma(p)}
&\lesssim\|\mathscr{P}^{\hspace{1pt}\prime}_{\lambda,1}f\|_{L^2}
\lesssim\Bigl\|\,\mathscr{P}^{\hspace{1pt}\prime}_{\lambda,1}
\hspace{-6pt}\sum_{n\in2\mathbb{N}+1}\hspace{-6pt}
{\mathbf{1}_{[n-1,n+1)}(\sqrt{\Delta}\hspace{1pt})}\,f\,\Bigr\|_{L^2}\\
&\lesssim\Bigl(\,\sum_nw_{{\lambda},n}^2\,
\|P^{\hspace{1pt}\prime}_{n,1}f\|_{L^2}^2\hspace{1pt}\Bigr)^{\hspace{-1pt}1/2}
\le\Bigl(\,\sum_nw_{\lambda,n}^2\,
\|P^{\hspace{1pt}\prime}_{n,1}\|_{L^{p^{\hspace{.5pt}\prime}}\!\to L^2}^2
\hspace{1pt}\Bigr)^{\hspace{-1pt}1/2},
\end{align*}
where $w_{\lambda,n}$ is defined as above.
Due to the rapid decay of $\chi$
and to the {upper} bound on $\|P^{\hspace{1pt}\prime}_{n,1}\|_{L^{p^{\hspace{.5pt}\prime}}\!\to L^2}$,
there exists a constant $C_0\,{>0}$ such that
$$
\Bigl(\,\sum\nolimits_{|n-\lambda|\ge C_0}w_{\lambda,n}^2\,
\|P^{\hspace{1pt}\prime}_{n,1}\|_{L^{p^{\hspace{.5pt}\prime}}\!\to L^2}^2
\hspace{1pt}\Bigr)^{\hspace{-1pt}1/2}
\ll\lambda^{\gamma(p)},	
$$
{hence}
$$
\lambda^{\gamma(p)}
\lesssim\Bigl(\,\sum\nolimits_{|n-\lambda|<C_0}
\|P^{\hspace{1pt}\prime}_{n,1}\|_{L^{p^{\hspace{.5pt}\prime}}\!\to L^2}^2
\hspace{1pt}\Bigr)^{\hspace{-1pt}1/2}.
$$
We deduce the lower bound
$$
\|P^{\hspace{1pt}\prime}_{\lambda,\eta_0}\|_{L^{p^{\hspace{.5pt}\prime}}\!\to L^2}
\gtrsim\lambda^{\gamma(p)}
$$
for $\eta_0=C_0+1$ by using 
$\|P^{\hspace{1pt}\prime}_{n,1}\|_{L^{p^{\hspace{.5pt}\prime}}\!\to L^2}
\le\|P^{\hspace{1pt}\prime}_{\lambda,\eta_0}\|_{L^{p^{\hspace{.5pt}\prime}}\!\to L^2}$\hspace{1pt}.
\medskip

\noindent\underline{Proof of the lower bound.}
By the previous reduction, it suffices to prove a lower bound
for $\psi\hspace{1pt}\mathscr{P}^{\hspace{1pt}\prime}_{\lambda,1}\hspace{1pt}\phi$,
where $\phi$ is a smooth bump function supported in a ball $B(x_0,\epsilon_0)$,
while $\psi=1$ on $B(x_0,2\epsilon_0)$ and $\supp\psi\subset B(x_0,3\epsilon_0)$.
It is now possible to consider a compact manifold $K$ agreeing with $X$ on $B(x_0,{4}\epsilon_0)$.
By finite propagation speed, the operator
$\psi\hspace{1pt}\mathscr{P}^{\hspace{1pt}\prime}_{\lambda,1}\hspace{1pt}\phi$
is identical on $X$ and $K$.
Therefore, the ``proof of sharpness'' for compact manifolds in \cite[{pp.~144--147}]{Sogge2017} applies,
with the following few modifications\,:
\begin{itemize}
\item
$\mathscr{P}^{\hspace{1pt}\prime}_{\lambda,1}$ should be substituted
to $\chi_\lambda$ and $\widetilde\chi_\lambda$,
\item
the local version of Weyl's law (see for instance~\cite{Sogge2014})
should be applied instead of the estimate on the counting function.
\end{itemize}

\section{Unboundedness of spectral projectors for the parabolic cylinder} \label{unbounded-cusp}
This appendix is devoted to the unboundedness of the spectral projectors $P_{\lambda,\eta}$ on
the parabolic cylinder $C$ in the half-plane model as in~\cite[\S~5.3]{Borthwick2016}.
We know that $y^{\hspace{.5pt}1\hspace{-.5pt}/2\hspace{1pt}\pm\hspace{1pt}i\hspace{1pt}\xi}$
are generalized eigenfunction associated with the eigenvalue $\frac14+\xi^2$.
Consider the function
$$   
f(x+iy)=\tfrac1{\sqrt{2\pi}}\int_{-\infty}^{+\infty}\!
\phi\bigl(\tfrac{\xi-\lambda}\eta\bigr)\,y^{\frac12-i\xi}\,d\xi\,,
$$
which doesn't depend on $x$ and which is spectrally localized
between $\lambda\hspace{-1pt}-\hspace{-1pt}\eta$
and $\lambda\hspace{-1pt}+\hspace{-1pt}\eta$\hspace{.5pt},
if $\phi$ is supported in $[-1,1]$\hspace{1pt}.
A simple computation yields
$$
f(y)=\eta\,y^{\frac12-i\lambda}\,\widehat{\phi}(\eta\log y)\hspace{.5pt}.
$$
Then
$$
\|f\|_{L^2(C)}=\eta\,\biggl(\hspace{.5pt}\int_0^{+\infty}\!
y\,|\widehat{\phi}(\eta\log y)|^2\,\tfrac{dy}{y^2}\biggr)^{\!\frac12}\!
=\eta\,\biggl(\hspace{.5pt}\int_{-\infty}^{+\infty}\!
|\widehat{\phi}(\eta\hspace{1pt}u)|^2\,du\biggr)^{\!\frac12}\!
=\eta^{\frac12}\,\|\widehat{\phi}\|_{L^2(\R)}
$$
while
$$
\|f\|_{L^p(C)}=\eta\,\biggl(\int_0^{+\infty}\!
y^{\frac p2}\,|\widehat{\phi}(\eta\log y)|^p\,\tfrac{dy}{y^2}\biggr)^{\!\frac1p}\!
=\eta\,\biggl(\hspace{.5pt}\int_{-\infty}^{+\infty}\!e^{\hspace{1pt}(\frac p2-1)\hspace{1pt}u}\,
|\widehat{\phi}(\eta\hspace{1pt}u)|^p\,du\biggr)^{\!\frac1p}.
$$
We may choose $\phi$ such that
\begin{equation}\label{BehaviorFourierTransform}
|\widehat{\phi}(u)|\sim(1\!+\!|u|)^{-\alpha}
\qquad\forall\,u\hspace{-1pt}\in\hspace{-1pt}\R
\end{equation}
for some $\alpha\hspace{-1pt}>\!\frac12$\hspace{.5pt}.
Then $\|f\|_{L^2(C)}\!<\hspace{-1pt}\infty$ while $\|f\|_{L^p(C)}\!=\hspace{-1pt}\infty$\hspace{.5pt}.
As a conclusion,
$$
\|\hspace{1pt}P_{\lambda,\eta}\hspace{1pt}\|_{L^2\to L^p}=\infty\,.
$$

\section{Analytic continuation of Fourier multipliers}\label{HolomorphicMultipliers}

The analytic continuation of $L^p\!\to\!L^p$ Fourier multipliers on noncompact symmetric spaces
was first observed in \cite{ClercStein1974}.
This phenomenon extends to $L^p\!\to\!L^q$ Fourier multipliers with $1\le p,q<2$ or $2<p,q\le\infty$.
Let us explain it for the hyperbolic plane $\H$.
\smallskip

Spherical functions on $\H$ are given by Harish-Chandra's integral formula
\begin{equation}\label{SphericalFunctions}
\varphi_\lambda(r)=\tfrac1\pi\int_0^\pi
\bigl(\cosh r-\sinh r\cos\theta\bigr)^{i\lambda-\frac12}\,d\theta
\end{equation}
 (see for instance \cite[(20) p. 39]{Helgason1984},
 \cite[(5.28)]{Koornwinder1984} or \cite[p.~112]{vanDijk2009}).
It follows from this formula that
\begin{itemize}[leftmargin=*]
\item
\,$\lambda\longmapsto\varphi_\lambda(r)$ is a holomorphic function on $\C$,
\item
\,$\lambda\longmapsto\varphi_\lambda(r)$ is a convex function on $i\R$,
\item
\,$|\varphi_{\lambda}(r)|\le\varphi_{i(\Im\!\lambda)}(r)$,
\end{itemize}
\smallskip 

Next let $0<\epsilon\le\frac12$ be fixed. We estimate $\varphi_{-i\epsilon}(r)=\frac1\pi\int_0^\pi(\cosh r-\sinh r\cos\theta)^{-(\frac12-\epsilon)}d\theta$.
Since $e^{-r}\le\cosh r-\sinh r\cos\theta\le e^r$, we have $e^{-(\frac12-\epsilon)r}\le\varphi_{-i\epsilon}(r)\le e^{(\frac12-\epsilon)r}$.

To improve the upper bound, we use 
\begin{align*}
\cosh r-\sinh r\cos\theta=\frac{1-\cos\theta}2e^r+\frac{1+\cos\theta}2e^{-r}\ge(\sin\frac\theta2)^2e^r\ge\pi^{-2}\theta^2e^r,
\end{align*}
and obtain that 
\begin{align*}
\varphi_{-i\epsilon}(r)\lesssim e^{-(\frac12-\epsilon)r}\int_0^\pi\theta^{2\epsilon-1}d\theta\lesssim\frac1\epsilon\,e^{-(\frac12-\epsilon)r}.
\end{align*}

We conclude that:
\smallskip
\begin{itemize}[leftmargin=*]
\item
\,for $0<\epsilon\le\frac12$ fixed, $\varphi_{-i\epsilon}(r)\sim e^{-(\frac12-\epsilon)r}$.
\end{itemize}


\noindent
Finally \footnote{\,This doesn't follow from \eqref{SphericalFunctions}.} we know that

\noindent$\bullet$
\,$\varphi_\lambda(r)=\varphi_{-\lambda}(r)$.
\vspace{.5mm}

\noindent
We deduce that
\begin{equation*}
|\varphi_\lambda(r)|\le\varphi_{-i\epsilon}(r)\lesssim e^{-(\frac12-\epsilon)r}
\end{equation*}
for every $\lambda$ in the strip $S_\epsilon=\{\lambda\in\C\mid|\Im\!\lambda\,|\le\epsilon\}$.
Hence
\begin{equation}\label{Lp}
\int_0^\infty|\varphi_\lambda(r)|^p\,\sinh r\,dr
\lesssim\!\int_0^\infty\!e^{-(\frac12-\frac1p-\epsilon)pr}\,dr
<+\infty
\end{equation}
if $0<\frac1p<\frac12-\epsilon<\frac12$, i.e., $0<\epsilon<\frac12-\frac1p<\frac12$.
\smallskip

Consider next convolution operators $T\hspace{-1pt}f=f*k$ on $\H$
corresponding to bounded Fourier multipliers $m=\widetilde{k}$.

\begin{proposition}
Assume that \,$T$ is bounded from $L^p$ to $L^q$,
with $1<p,q<2$ or $2<p,q<\infty$.
Then \,$m$ extends to a holomorphic function in the interior of the strip \,$S_\epsilon$,
where $\epsilon=\min\bigl\{|\frac12-\frac1p|,|\frac12-\frac1q|\bigr\}$.
\end{proposition}

\begin{proof}
By duality\footnote{\,If $T$ is bounded from $L^p$ to $L^q$,
then $T$ is bounded from $L^{q'}$ \!to $L^{p^{\hspace{.5pt}\prime}}$\!.},
we may assume that $2<p,q<\infty$.
Then $\varphi_\lambda\in L^p\cap L^q$, for every $\lambda$ in the interior of $S_\epsilon$,
and
\begin{equation}\label{Tphi}
T\varphi_\lambda=m(\lambda)\,\varphi_\lambda
\quad\forall\,\lambda\in\R.
\end{equation}
Let $f\in L^{q'}(\H)$ be a radial function whose spherical Fourier transform doesn't vanish in the interior of $S_\epsilon$. For instance the heat kernel at any time $t>0$, whose spherical Fourier transform is equal to $e^{-t(\lambda^2+1/4)}$.
By integrating \eqref{Tphi} against $f$, we obtain
\begin{equation}\label{AnalyticContinuation}
m(\lambda)=\tfrac{2\pi}{\widetilde{f}(\lambda)}
\int_\H T\hspace{-1pt}f(r)\,\varphi_\lambda(r)\,\sinh r\,dr,
\end{equation}
The right hand of \eqref{AnalyticContinuation} provides
a holomorphic extension of $m$ in the interior of $S_\epsilon$.
\end{proof}
\begin{remark}
Moreover \eqref{Tphi} implies the following additional properties\,:

\noindent\underline{Case} \hspace{1pt}$p\hspace{-.5pt}=\hspace{-.5pt}q$\,:
\,$m$ is bounded inside $S_\epsilon$,

\noindent\underline{Limit case} \hspace{1pt}$p\hspace{-.5pt}=\hspace{-.5pt}q\hspace{-.5pt}=\hspace{-1pt}1$
or \hspace{1pt}$p\hspace{-.5pt}=\hspace{-.5pt}q\hspace{-.5pt}=\hspace{-.5pt}\infty$\,:
\,$m$ extends continuously to the boundary of $S_{\frac12}$.
\end{remark}

\bibliographystyle{siam}
\bibliography{referencesMay10}

\end{document}